\pdfoutput=1

\documentclass[twoside,11pt]{article}
\usepackage{arxiv-isipta}

\usepackage[utf8]{inputenc}
\inputencoding{utf8}

\usepackage{enumitem}
\usepackage{booktabs}
\usepackage[group-digits=integer, group-separator={,}, group-minimum-digits = 4, separate-uncertainty=true, detect-family=true]{siunitx}
\usepackage{multicol}
\usepackage{color}
\usepackage{nicefrac}
\usepackage[colorlinks=true, citecolor=blue, pdfborder={0 0 0}]{hyperref}
\usepackage{doi} 
\usepackage[capitalise]{cleveref}
\usepackage{microtype}
\usepackage[ruled, linesnumbered, noend]{algorithm2e}
\SetKw{Break}{break}
\SetKw{And}{and}
\SetKw{Or}{or}

\SetEnumitemKey{twocol}{%
	itemsep=1\itemsep,
	parsep=1\parsep,
	before=\raggedcolumns\begin{multicols}{2},
	after=\end{multicols}
}
\SetEnumitemKey{threecol}{%
	itemsep=.1\itemsep,
	parsep=.1\parsep,
	before=\raggedcolumns\begin{multicols}{3},
	after=\end{multicols}
}

\newtheorem{binex}{Example}
\newtheorem{theorem}{Theorem}
\newtheorem{lemma}[theorem]{Lemma}
\newtheorem{proposition}[theorem]{Proposition}

\newtheorem{corollary}[theorem]{Corollary}
\newtheorem{definition}[theorem]{Definition}

\newtheorem*{counterex}{Counterexample for \cref{prop:CoeffOfErgod:ErgodicUpperBound}}

\newcommand{\abs}[1]{\left\vert {#1} \right\vert}
\newcommand{\norm}[1]{\left\Vert {#1} \right\Vert}
\newcommand{\card}[1]{\left\vert {#1} \right\vert}

\newcommand{\reals}{\mathbb{R}}
\newcommand{\nonnegreals}{\reals_{\geq 0}}
\newcommand{\strictlyposreals}{\reals_{> 0}}
\newcommand{\nats}{\mathbb{N}}
\newcommand{\natz}{\nats_{0}}

\newcommand{\statespace}{\mathcal{X}}
\newcommand{\statespacesub}[1]{\statespace_{\mathit{#1}}}
\newcommand{\setoffn}{\mathcal{L}}
\newcommand{\setoffna}[1][\statespace]{\setoffn(#1)}

\newcommand{\cent}[1]{\tilde{#1}}

\newcommand{\indic}[1]{\mathbb{I}_{#1}}
\newcommand{\indica}[2]{\indic{#1}(#2)}

\newcommand{\lowtranop}{\underline{T}}
\newcommand{\lowtranopa}[1]{\lowtranop_{#1}}
\newcommand{\uptranop}{\overline{T}}
\newcommand{\uptranopa}[1]{\uptranop_{#1}}

\newcommand{\lowrateop}{\underline{Q}}
\newcommand{\uprateop}{\overline{Q}}
\newcommand{\setofdomratemat}[1][\lowrateop]{\mathcal{Q}_{#1}}
\newcommand{\lowq}[1]{\underline{q}_{#1}}
\newcommand{\upq}[1]{\overline{q}_{#1}}

\newcommand{\prob}{\mathrm{P}}
\newcommand{\prev}{\mathrm{E}}

\newcommand{\linpreva}[1]{\prev(#1)}

\newcommand{\lowprev}{\underline{\prev}}
\newcommand{\lowpreva}[1]{\lowprev(#1)}
\newcommand{\lowprevacond}[2]{\lowprev(#1|#2)}

\newcommand{\upprev}{\overline{\prev}}
\newcommand{\uppreva}[1]{\upprev(#1)}
\newcommand{\upprevacond}[2]{\upprev(#1|#2)}
\newcommand{\credset}{\mathcal{M}}
\newcommand{\credseta}[1]{\credset(#1)}

\newcommand{\coefferg}{\rho}
\newcommand{\coefferga}[1]{\coefferg(#1)}

\newcommand{\upreachd}{\rightsquigarrow}
\newcommand{\upreachda}[1]{\overset{#1}{\rightsquigarrow}}

\newcommand{\upreachc}[1][\uprateop]{\mathrel{\ooalign{\hss$#1$\hss\cr$\longrightarrow$}}}
\newcommand{\lowreachc}[1][\lowrateop]{\mathrel{\ooalign{\hss$#1$\hss\cr$\longrightarrow$}}}

\title{Imprecise Continuous-Time Markov Chains: \\ Efficient Computational Methods with Guaranteed Error Bounds}
\author{%
	\name Alexander~Erreygers \email alexander.erreygers@ugent.be\\
	\addr Ghent University, SMACS Research Group \\
	\AND
	\name Jasper~De~Bock \email jasper.debock@ugent.be\\
	\addr Ghent University - imec, IDLab, ELIS \\
}
\ShortHeadings{ICTMCs: Efficient Computational Methods with Guaranteed Error Bounds}{Erreygers \& De~Bock}

\begin{document}
\maketitle

\begin{abstract}
	Imprecise continuous-time Markov chains are a robust type of continuous-time Markov chains that allow for partially specified time-dependent parameters.
	Computing inferences for them requires the solution of a non-linear differential equation.
	As there is no general analytical expression for this solution, efficient numerical approximation methods are essential to the applicability of this model.
	We here improve the uniform approximation method of \cite{2016Krak} in two ways and propose a novel and more efficient adaptive approximation method.
	For ergodic chains, we also provide a method that allows us to approximate stationary distributions up to any desired maximal error.
\end{abstract}

	\section{Introduction}
\label{sec:Intro}
Markov chains are a popular type of stochastic processes that can be used to model a variety of systems with uncertain dynamics, both in discrete and continuous time.
In many applications, however, the core assumption of a Markov chain---i.e., the Markov property---is not entirely justified.
Moreover, it is often difficult to exactly determine the parameters that characterise the Markov chain.
In an effort to handle these modelling errors in an elegant manner, several authors have recently turned to imprecise probabilities \citep{decooman2009,2013Skulj,2012Hermans,2015Skulj,2016Krak,2017DeBock}.

As \cite{2016Krak} thoroughly demonstrate, making inferences about an imprecise continuous-time Markov chain---determining lower and upper expectations or probabilities---requires the solution of a non-linear vector differential equation.
To the best of our knowledge, this differential equation cannot be solved analytically, at least not in general.
\cite{2016Krak} proposed a method to numerically approximate the solution of the differential equation, and argued that it outperforms the approximation method that \cite{2015Skulj} previously introduced.
One of the main results of this contribution is a novel approximation method that outperforms that of \cite{2016Krak}.

An important property---both theoretically and practically---of continuous-time Markov chains is the behaviour of the solution of the differential equation as the time parameter recedes to infinity.
If regardless of the initial condition the solution converges, we say that the chain is ergodic.
We show that in this case the approximation is guaranteed to converge as well.
This constitutes the second main result of this contribution and serves as a motivation behind the novel approximation method.
Furthermore, we also quantify a worst-case convergence rate for the approximation.
This unites the work of \cite{2015Skulj}, who studied the rate of convergence for discrete-time Markov chains, and \cite{2017DeBock}, who studied the ergodic behaviour of continuous-time Markov chains from a qualitative point of view.
One of the uses of our worst-case convergence rate is that it allows us to approximate the limit value of the solution up to a guaranteed error.

This paper is an extended preprint of \citep{2017erreygers}.
Recently, it has come to our attention that one of the results in that paper, namely \cref{prop:CoeffOfErgod:ErgodicUpperBound}, is false.
Fortunately, none of the other results in \citep{2017erreygers}---and hence also in this preprint---depend on \cref{prop:CoeffOfErgod:ErgodicUpperBound} and the main conclusions and contributions of the paper therefore remain intact.
For that reason, we have only made the following two modifications with respect to the previous version: we have omitted the proof of \cref{prop:CoeffOfErgod:ErgodicUpperBound}, and we have added a counterexample to show that the statement is indeed incorrect.

To ensure the readability of the main text, we have gathered the proofs of all the results in the Appendix.
In this Appendix, we also discuss the ergodicity of both discrete and continuous-time Markov chains more thoroughly.

	\section{Mathematical preliminaries}
\label{sec:Preliminaries}
Throughout this contribution, we denote the set of real, non-negative real and strictly positive real numbers by $\reals$, $\nonnegreals$ and $\strictlyposreals{}$, respectively.
The set of natural numbers is denoted by $\nats$, if we include zero we write $\natz \coloneqq \nats \cup \{ 0 \}$.
For any set $S$, we let $\card{S}$ denote its cardinality.
If $a$ and $b$ are two real numbers, we say that $a$ is lower (greater) than $b$ if $a \leq b$ ($a \geq b$), and that $a$ is strictly lower (greater) than $b$ if $a < b$ ($a > b$).


	\subsection{Gambles and norms}
We consider a finite \emph{state space} $\statespace$, and are mainly concerned with real-valued functions on $\statespace$.
All of these real-valued functions on $\statespace$ are collected in the set $\setoffna$, which is a vector space.
If we identify the state space $\statespace$ with $\{1, \dots, \card{\statespace}\}$, then any function $f \in \setoffna$ can be identified with a vector: for all $x\in\statespace$, the $x$-component of this vector is $f(x)$.
A special function on $\statespace$ is the indicator $\indic{A}$ of an event $A$.
For any $A \subseteq \statespace$, it is defined for all $x \in \statespace$ as $\indica{A}{x} = 1$ if $x \in A$ and $\indica{A}{x} = 0$ otherwise.
In order not to obfuscate the notation too much, for any $y \in \statespace$ we write $\indic{y}$ instead of $\indic{\{y\}}$.
If it is required from the context, we will also identify the real number $\gamma \in \reals$ with the map $\gamma$ from $\statespace$ to $\reals$, defined as $\gamma(x) = \gamma$ for all $x \in \statespace$.

We provide the set $\setoffna$ of functions with the standard maximum norm $\norm{\cdot}$, defined for all $f\in\setoffna$ as $\norm{f}	\coloneqq \max \left\{ \abs{f(x)} \colon x \in \statespace \right\}$.
A seminorm that captures the variation of $f \in \setoffna$ will also be of use; we therefore define the variation seminorm $\norm{f}_{v} \coloneqq \max f - \min f$.
Since the value $\norm{f}_{v} / 2$ occurs often in formulas, we introduce the shorthand notation $\norm{f}_{c} \coloneqq \norm{f}_{v}/2$.

	\subsection{Non-negatively homogeneous operators}
An operator $A$ that maps $\setoffna$ to $\setoffna$ is \emph{non-negatively homogeneous} if for all $\mu \in \nonnegreals$ and all $f \in \setoffna$, $A (\mu f) = \mu A f$.
The maximum norm $\norm{\cdot}$ for functions induces an operator norm:
\[
	\norm{A}
	\coloneqq \sup \{ \norm{A f} \colon f \in \setoffna, \norm{f} = 1 \}.
\]
If for all $\mu \in \reals$ and all $f,g \in \setoffna$, $A(\mu f + g) = \mu A f + A g$, then the operator $A$ is \emph{linear}.
In that case, it can be identified with a matrix of dimension $\card{\statespace}\times\card{\statespace}$, the $(x,y)$-component of which is $[A \indic{y}](x)$.
The identity operator $I$ is an important special case, defined for all $f \in \setoffna$ as $I f \coloneqq f$.

Two types of non-negatively homogeneous operators play a vital role in the theory of imprecise Markov chains: lower transition operators and lower transition rate operators.
\begin{definition}
\label{def:LowerTransitionOperator}
	An operator $\lowtranop{}$ \/from $\setoffna$ to $\setoffna$ is called a \emph{lower transition operator} if for all $f \in \setoffna$ and all $\mu \in \nonnegreals$:
	\begin{enumerate}[threecol, label=L\arabic*:, ref=(L\arabic*), series=LTO]
		\item \label{def:LTO:DominatesMin}
		$\lowtranop f \geq \min f$;
		\item \label{def:LTO:SuperAdditive}
		$\lowtranop(f + g) \geq \lowtranop f + \lowtranop g$;
		\item \label{def:LTO:NonNegativelyHom}
		$\lowtranop (\mu f) = \mu \lowtranop f$.
	\end{enumerate}
\end{definition}
Every lower transition operator $\lowtranop$ has a conjugate upper transition operator $\uptranop$, defined for all $f \in \setoffna$ as $\uptranop f \coloneqq - \lowtranop (- f)$.

\begin{definition}
\label{def:LowerTransitionRateOperator}
	An operator $\lowrateop{}\,$ from $\setoffna$ to $\setoffna$ is called a \emph{lower transition rate operator} if for any $f,g \in \setoffna$, any $\mu \in \nonnegreals$, any $\gamma\in\reals$ and any $x,y \in \statespace$ such that $x\neq y$:
	\begin{enumerate}[twocol, label=R\arabic*:, ref=(R\arabic*), series=LTRO]
		\item \label{def:LTRO:Constant}
			$\lowrateop \gamma = 0$;
		\item \label{def:LTRO:SuperAdditive}
			${\lowrateop(f + g) \geq \lowrateop f + \lowrateop g}$;
		\item \label{def:LTRO:NonNegativelyHom}
			$\lowrateop (\mu f) = \mu \lowrateop f$;
		\item \label{def:LTRO:Sign}
			$[\lowrateop \indic{x}](y) \geq 0$.
	\end{enumerate}
\end{definition}
The conjugate lower transition rate operator $\uprateop$ is defined for all $f \in \setoffna$ as $\uprateop f \coloneqq - \lowrateop (-f)$.

As will become clear in Section~\ref{sec:MCs}, lower transition operators and lower transition rate operators are tightly linked.
For instance, we can use a lower transition rate operator to construct a lower transition operator.
One way is to use Eqn.~\eqref{eqn:TDLTO:FunctionDifferentialEquation} further on.
Another one is given in the following proposition, which is a strengthened version of \citep[Proposition~5]{2017DeBock}.
\begin{proposition}
\label{prop:IPlusDeltaQLowTranOp}
	Consider any lower transition rate operator $\lowrateop$ and any $\delta \in \nonnegreals$.
	Then the operator $(I + \delta \lowrateop)$ is a lower transition operator if and only if $\delta \norm{\lowrateop} \leq 2$.
\end{proposition}

We end this section with the first---although minor---novel result of this contribution.
The norm of a lower transition rate operator is essential for all the approximation methods that we will discuss.
The following proposition supplies us with an easy formula for determining it.
\begin{proposition}
\label{prop:LTRO:PropositionNorm}
	Let $\lowrateop$ be a lower transition rate operator.
	Then $\norm{\lowrateop} = 2 \max \{ \abs{[\lowrateop \indic{x}](x)} \colon x \in \statespace \}$.
\end{proposition}

\begin{binex}
\label{binex:LTRO}
	Consider a binary state space $\statespace = \{0, 1\}$ and two closed intervals $[\lowq{0}, \upq{0}] \subset \nonnegreals$ and $[\lowq{1}, \upq{1}] \subset \nonnegreals$.
	Let\vspace{-5pt}
	\begin{equation*}
		\lowrateop f
		\coloneqq \min \left\{
		\begin{bmatrix}
			q_0 (f(1) - f(0)) \\
			q_1 (f(0) - f(1))
		\end{bmatrix}
		\colon q_0 \in [\lowq{0}, \upq{0}], q_1 \in [\lowq{1}, \upq{1}] \right\}
		\text{ for all } f \in \setoffna.
	\end{equation*}
	Then one can easily verify that $\lowrateop$ is a lower transition rate operator.

	\cite{2016Krak} also consider a running example with a binary state space, but they let $\statespace \coloneqq \{ \texttt{healthy}, \texttt{sick} \}$.
	We here identify \texttt{healthy} with $0$ and \texttt{sick} with $1$.
	In \cite[Example~18]{2016Krak}, they propose the following values for the transition rates: $[\lowq{0}, \upq{0}] \coloneqq [1/52, 3/52]$ and $[\lowq{1}, \upq{1}] \coloneqq [1/2,2]$.
	It takes \citeauthor{2016Krak} a lot of work to determine the exact value of the norm of $\lowrateop$, see \cite[Example~19]{2016Krak}.
	We simply use Proposition~\ref{prop:LTRO:PropositionNorm}: $\smash{\norm{\lowrateop} = 2 \max\{ 3/52, 2 \} = 4}$.
\end{binex}

	\section{Imprecise continuous-time Markov chains}
\label{sec:MCs}
For any lower transition rate operator $\lowrateop$ and any $f \in \setoffna$, \cite{2015Skulj} has shown that the differential equation
\vspace{-4pt}
\begin{equation}
\label{eqn:TDLTO:FunctionDifferentialEquation}
	\frac{\mathrm{d}}{\mathrm{d} t} \lowtranopa{t} f = \lowrateop \lowtranopa{t} f
\vspace{3pt}
\end{equation}
with initial condition $\lowtranopa{0} f \coloneqq f$ has a unique solution for all $t \in \nonnegreals$.
Later, \cite{2017DeBock} proved that the time-dependent operator $\lowtranopa{t}$ itself satisfies a similar differential equation, and that it is a lower transition operator.
Finding the unique solution of Eqn.~\eqref{eqn:TDLTO:FunctionDifferentialEquation} is non-trivial.
Fortunately, we can approximate this solution, as by \cite[Proposition~10]{2017DeBock}
\begin{equation}
\label{eqn:TDLTO:LimitFormula}
	\lowtranopa{t}
	= \lim_{n \to \infty} \left( I + \frac{t}{n} \lowrateop \right)^{n}.
\end{equation}


\begin{binex}
\label{binex:AnalyticalExpressionsForAppliedLTO}
	In the simple case of Example~\ref{binex:LTRO}, we can use Eqn.~\eqref{eqn:TDLTO:LimitFormula} to obtain analytical expressions for the solution of Eqn.~\eqref{eqn:TDLTO:FunctionDifferentialEquation}.
	Assume that $\lowq{0} + \upq{1}> 0$ and fix some $t \in \nonnegreals$.
Then
	\begin{align*}
		[\lowtranopa{t} f](0)
		= f(0) + \lowq{0} h(t)
		~~\text{and}~~
		[\lowtranopa{t} f](1)
		= f(1) - \upq{1} h(t)
		~~\text{for all $f\in\setoffna$ with $f(0)\leq f(1)$,}
	\end{align*}
	where $h(t) \coloneqq \norm{f}_{v} (\lowq{0} + \upq{1})^{-1} \big(1 - e^{-t (\lowq{0} + \upq{1})}\big)$.
	The case $f(0) \geq f(1)$ yields similar expressions.
\end{binex}

For a linear lower transition rate operator $\lowrateop$---i.e., if it is a transition rate matrix $Q$---Eqn.~\eqref{eqn:TDLTO:LimitFormula} reduces to the definition of the matrix exponential.
It is well-known---see \citep{1991Anderson}---that this matrix exponential $T_t = e^{t Q}$ can be interpreted as the transition matrix at time $t$ of a time-homogeneous or stationary continuous-time Markov chain: the $(x,y)$-component of $T_t$ is the probability of being in state $y$ at time $t$ if the chain started in state $x$ at time $0$.
Therefore, it follows that the expectation of the function $f \in \setoffna$ at time $t \in \nonnegreals$ conditional on the initial state $x \in \statespace$, denoted by $\prev(f(X_{t})|X_{0} = x)$, is equal to $[T_t f](x)$.

As Eqn.~\eqref{eqn:TDLTO:LimitFormula} is a non-linear generalisation of the definition of the matrix exponential, we can interpret $\lowtranopa{t}$ as the non-linear generalisation of the matrix exponential $T_t = e^{t Q}$.
Extending this parallel, we might interpret $\lowtranopa{t}$ as the non-linear generalisation of the transition matrix---i.e., as the lower transition operator---at time $t$ of a generalised continuous-time Markov chain.
In fact, \cite{2016Krak} prove that this is indeed the case.
They show that---under some conditions on $\lowrateop$---$[\lowtranopa{t} f](x)$ can be interpreted as the tightest lower bound for $\prev(f(X_{t})|X_{0} = x)$ with respect to a set of---not necessarily Markovian---stochastic processes that are consistent with $\lowrateop$.
\cite{2016Krak} argue that, just like a transition rate matrix $Q$ characterises a (precise) continuous-time Markov chain, a lower transition rate operator $\lowrateop$ characterises a so-called imprecise continuous-time Markov chain.

The main objective of this contribution is to determine $\lowtranopa{t} f$ for some $f \in \setoffna$ and some $t \in \strictlyposreals$.
Our motivation is that, from an applied point of view on imprecise continuous-time Markov chains, what one is most interested in are tight lower and upper bounds on expectations of the form $\prev(f(X_t)|X_{0} = x)$.
As explained above, the lower bound is given by $\lowprevacond{f(X_t)}{X_0 = x} = [\lowtranopa{t} f](x)$.
Similarly, the upper bound is given by $\upprevacond{f(X_t)}{X_0 = x} = -[\lowtranopa{t} (-f)](x)$.
Note that the lower (or upper) probability of an event $A \subseteq \statespace$ conditional on the initial state $x$ is a special case of a lower (or upper) expectation: $\underline{\prob}(X_{t} \in A | X_0 = x) = \lowprevacond{\indica{A}{X_t}}{X_0 = x}$ and similarly for the upper probability.
Hence, for the sake of generality we can focus on $\lowtranopa{t} f$ and forget about its interpretation.
As in most cases analytically solving Eqn.~\eqref{eqn:TDLTO:FunctionDifferentialEquation} is infeasible or even impossible, we resort to methods that yield an approximation up to some guaranteed maximal error.

	\section{Approximation methods}
\label{sec:EfficientComputation}
\cite{2015Skulj} was, to the best of our knowledge, the first to propose methods that approximate the solution $\lowtranopa{t} f$ of Eqn.~\eqref{eqn:TDLTO:FunctionDifferentialEquation}.
He proposes three methods: one with a uniform grid, a second with an adaptive grid and a third that is a combination of the previous two.
In essence, he determines a step size $\delta$ and then approximates $\lowtranopa{t + \delta} f$ with $e^{\delta Q} \lowtranopa{t} f$, where $Q$ is a transition rate matrix determined from \smash{$\lowrateop$ and $\lowtranopa{t} f$}.
One drawback of this method is that it needs the matrix exponential $e^{\delta Q}$, which---in general---needs to be approximated as well.
\cite{2015Skulj} mentions that his methods turn out to be quite computationally heavy, even if the uniform and adaptive methods are combined.

We consider two alternative approximation methods---one with a uniform grid and one with an adaptive grid---that both work in the same way.
First, we pick a small step $\delta_1 \in \nonnegreals$ and apply the operator $(I + \delta_1 \lowrateop)$ to the function $g_0 = f$, resulting in a function $g_1 \coloneqq (I + \delta_1 \lowrateop) f$.
Recall from Proposition~\ref{prop:IPlusDeltaQLowTranOp} that if we want $(I + \delta_1 \lowrateop)$ to be a lower transition operator, then we need to demand that $\delta_1 \norm{\lowrateop} \leq 2$.
Next, we pick a (possibly different) small step $\delta_2 \in \nonnegreals$ such that $\delta_2 \norm{\lowrateop} \leq 2$ and apply the lower transition operator $(I + \delta_2 \lowrateop)$ to the function $g_1$, resulting in a function $g_2 \coloneqq (I + \delta_2 \lowrateop) g_1$.
If we continue this process until the sum of all the small steps is equal to $t$, then we end up with an approximation for $\lowtranopa{t} f$.
More formally, let $s \coloneqq (\delta_1, \dots, \delta_k)$ denote a sequence in $\nonnegreals$ such that, for all $i \in \{ 1, \dots, k \}$, $\delta_i \norm{\lowrateop} \leq 2$.
Using this sequence $s$ we define the \emph{approximating lower transition operator}
\vspace{-6pt}
\begin{equation*}
	\Phi(s)
	\coloneqq (I + \delta_k \lowrateop) \cdots (I + \delta_1 \lowrateop).\vspace{3pt}
\end{equation*}
What we are looking for is a convenient way to determine the sequence $s$ such that the error $\norm{\lowtranopa{t} f - \Phi(s) f}$ is guaranteed to be lower than some desired maximal error $\epsilon \in \strictlyposreals$.

	\subsection{Using a uniform grid}
\label{ssec:UniformGrid}
\cite{2016Krak} provide one way to determine the sequence $s$.
They assume a uniform grid, in the sense that all elements of the sequence $s$ are equal to $\delta$.
The step size $\delta$ is completely determined by the desired maximal error $\epsilon$, the time $t$, the variation norm of the function $f$ and the norm of $\lowrateop$; \cite[Proposition~8.5]{2016Krak} guarantees that the actual error is lower than $\epsilon$.
Algorithm~\ref{alg:Uniform} provides a slightly improved version of \cite[Algorithm~1]{2016Krak}.
The improvement is due to Proposition~\ref{prop:IPlusDeltaQLowTranOp}: we demand that $n \geq t \norm{\lowrateop} / 2$ instead of $n \geq t \norm{\lowrateop}$.
\begin{algorithm}
	\caption{Uniform approximation \label{alg:Uniform}}
	\DontPrintSemicolon

	\KwData{A lower transition rate operator $\lowrateop$, a function $f \in \setoffna$, a maximal error $\epsilon \in \strictlyposreals$, and a time point $t \in \nonnegreals$.}
	\KwResult{$\lowtranopa{t} f \pm \epsilon$}

	$g_{0} \gets f$\; \nllabel{line:Uniform:First}
	\lIf{$\norm{f}_{c} = 0$ \Or $\norm{\lowrateop} = 0$ \Or $t = 0$}{$(n, \delta) \gets (0, 0)$}
	\Else{
		$n \gets \big\lceil \max \{ t \norm{\lowrateop} / 2, t^2 \norm{\lowrateop}^2 \norm{f}_{c} / \epsilon \}\big\rceil$\; \nllabel{line:Uniform:DetermineN}
		$\delta \gets t / n$\;
		\For{$i = 0, \dots, n-1$}{
			$g_{i+1} \gets g_{i} + \delta \lowrateop g_{i}$\; \nllabel{line:Uniform:IncrementOfG}
		}
	}
	\Return $g_{n}$
\end{algorithm}

More formally, for any $t \in \nonnegreals$ and any $n \in \nats$ such that $t \norm{\lowrateop} \leq 2 n$, we consider the \emph{uniformly approximating lower transition operator}
\vspace{-3pt}
\begin{equation*}
	\Psi_{t}(n)
	\coloneqq \left(I + \frac{t}{n} \lowrateop\right)^{n}.\vspace{3pt}
\end{equation*}
As a special case, we define $\Psi_{t}(0) \coloneqq I$.
The following theorem then guarantees that the choice of $n$ in Algorithm~\ref{alg:Uniform} results in an error $\norm{\lowtranopa{t} f - \Psi_{t}(n) f}$ that is lower than the desired maximal error $\epsilon$.
\begin{theorem}
\label{the:UniformApproximationWithError}
	Let $\lowrateop$ be a lower transition rate operator and fix some $f\in\setoffna$, $t \in \nonnegreals$ and $\epsilon \in \strictlyposreals$.
	If we use Algorithm~\ref{alg:Uniform} to determine $n$, $\delta$ and $g_{0}, \dots, g_n$, then we are guaranteed that
	\[
		\norm{\lowtranopa{t} f - \Psi_{t}(n) f}
		= \norm{\lowtranopa{t} f - g_n}
		\leq \epsilon'
		\coloneqq
		\delta^2 \norm{\lowrateop}^2 \sum_{i=0}^{n-1} \norm{g_{i}}_{c}
		\leq \epsilon.
	\]
\end{theorem}

Theorem~\ref{the:UniformApproximationWithError} is an extension of \cite[Proposition~8.5]{2016Krak}.
We already mentioned that the demand $n \geq t \norm{\lowrateop}$ can be relaxed to $n \geq t \norm{\lowrateop} / 2$.
Furthermore, it turns out that we can compute an upper bound $\epsilon'$ on the error that is (possibly) lower than the desired maximal error $\epsilon$.
If we want to determine this $\epsilon'$ while running Algorithm~\ref{alg:Uniform}, we simply need to add $\epsilon' \gets 0$ to line~\ref{line:Uniform:First} and insert $\epsilon' \gets \epsilon' + \delta^2 \norm{\lowrateop}^2 \norm{g_{i}}_{c}$ just before line~\ref{line:Uniform:IncrementOfG}.

\begin{binex}
\label{binex:UniformApproximation}
	We again consider the simple case of Example~\ref{binex:LTRO} and illustrate the use of Theorem~\ref{the:UniformApproximationWithError} with a numerical example based on \cite[Example~20]{2016Krak}.
	\cite{2016Krak} use Algorithm~\ref{alg:Uniform} to approximate $\lowtranopa{1} \indic{1}$, and find that $n = \num{8000}$ guarantees an error lower than the desired maximal error $\epsilon \coloneqq \num{1e-3}$.
	As reported in Table~\ref{tab:ComparisonOfCompDuration}, we use Theorem~\ref{the:UniformApproximationWithError} to compute $\epsilon'$.
	We find that $\epsilon' \approx \num{0.430e-3}$, which is approximately a factor two smaller than the desired maximal error $\epsilon$.
	\begin{table}
		\caption{Comparison of the presented approximation methods, obtained using a naive, unoptimised implementation of the algorithms in Python.
		$N$ is the total number of iterations, $D_{\epsilon}$ ($D_{\epsilon'}$) is the average duration---in seconds, averaged over 50 independent runs---without (with) keeping track of $\epsilon'$, and $\epsilon_a$ is the actual error.
		 The Python code is made available at \href{https://github.com/alexander-e/ictmc}{github.com/alexander-e/ictmc}.}
		\label{tab:ComparisonOfCompDuration}
		\begin{center}
		\begin{tabular}{rS[table-format=4.]S[table-format=1.4]S[table-format=1.4]S[table-format=1.4]S[table-format=1.4]}
			\toprule
			{Method} & {$N$} & {$D_{\epsilon}$} & {$D_{\epsilon'}$} & {$\epsilon' \times 10^3$} & {$\epsilon_a \times 10^3$} \\
			\midrule
			Uniform                      & 8000 & 0.0345 & 0.0574 & 0.430 & 0.0335 \\
			Uniform                      &  250 & 0.00171 & 0.0264 & 13.8  & 1.07 \\
			Adaptive with $m = 1$        & 3437 & 0.0371 & 0.0428 & 1.000 & 0.108 \\
			Adaptive with $m = 20$       & 3456 & 0.0143 & 0.0254 & 0.992 & 0.107 \\
			Uniform ergodic with $m = 1$ & 6133 & 0.0264 & 0.0449 & 0.560 & 0.0437 \\
			\bottomrule
		\end{tabular}
		\end{center}
	\end{table}

	In this case, since we know the analytical expression for $\lowtranopa{1} \indic{1}$ from Example~\ref{binex:AnalyticalExpressionsForAppliedLTO}, we can determine the actual error $\epsilon_{a} = \norm{\lowtranopa{1} \indic{1} - \Psi_{1}(8000) \indic{1}}$.
	Quite remarkably, the actual error is approximately $\num{3.35e-5}$, which is roughly 30 times smaller than the desired maximal error.
	This leads us to think that the number of iterations used by the uniform method is too high.
	In fact, we find that using as few as \num{250} iterations---roughly \num{8000 / 30}---already results in an actual error that is approximately equal to the desired one: $\norm{\lowtranopa{1} \indic{1} - \Psi_{1}(250) \indic{1}} \approx \num{1.07e-3}$.
\end{binex}

	\subsection{Using an adaptive grid}
\label{ssec:AdaptiveGrid}
In Example~\ref{binex:UniformApproximation}, we noticed that the maximal desired error was already satisfied for a uniform grid that was much coarser than that constructed by Algorithm~\ref{alg:Uniform}.
Because of this, we are led to believe that we can find a better approximation method than the uniform method of Algorithm~\ref{alg:Uniform}.

To this end, we now consider grids where, for some integer $m$, every $m$ consecutive time steps in the grid are equal.
In particular, we consider a sequence $\delta_1, \dots, \delta_n$ in $\nonnegreals$ and some $k \in \nats$ such that $1 \leq k \leq m$ and, for all $i \in \{ 1, \dots, n \}$, $\delta_i \norm{\lowrateop} \leq 2$.
From such a sequence, we then construct the \emph{$m$-fold approximating lower transition operator}:
\[
	\Phi_{m,k}(\delta_1, \dots, \delta_n)
	\coloneqq (I + \delta_n \lowrateop)^{k} (I + \delta_{n-1} \lowrateop)^{m} \cdots (I + \delta_{1} \lowrateop)^{m},
\]
where if $n = 1$ only $(I + \delta_1 \lowrateop)^{k}$ remains and if $n = 2$ only $(I + \delta_2 \lowrateop)^{k} (I + \delta_{1} \lowrateop)^{m}$ remains.

The uniform approximation method of before is a special case of the $m$-fold approximating lower transition operator; a more interesting method to construct an $m$-fold approximation is Algorithm~\ref{alg:Adaptive}.
In this algorithm, we re-evaluate the time step every $m$ iterations, possibly increasing its length.
\begin{algorithm}
	\caption{Adaptive approximation \label{alg:Adaptive}}
	\DontPrintSemicolon

	\KwData{A lower transition rate operator $\lowrateop$, a gamble $f \in \setoffna$, an integer $m \in \nats$, a tolerance $\epsilon \in \strictlyposreals$, and a time period $t \in \nonnegreals$.}
	\KwResult{$\lowtranopa{t} f \pm \epsilon$}
	$(g_{(0,m)}, \Delta, i) \gets (f, t, 0)$\;
	\lIf{$\norm{f}_{c} = 0$ \Or $\norm{\lowrateop} = 0$ \Or $t = 0$}{$(n, k) \gets (0, m)$}
	\Else{
		\While{$\Delta > 0$ \And $\norm{g_{(i,m)}}_{c} > 0$}{
			$i \gets i+1$\;
			$\delta_i \gets \min \{ \Delta, 2 / \norm{\lowrateop}, \epsilon / (t \norm{\lowrateop}^2 \norm{g_{(i-1,m)}}_{c}) \}$\; \nllabel{line:Adaptive:delta}
			\If{$m \delta_i > \Delta$}{
				$k_i \gets \lceil \Delta / \delta_i \rceil$\;
				$\delta_i \gets \Delta / k_i$\;
			}\lElse{$k_i \gets m$}
			$g_{(i,0)} \gets g_{(i-1,m)}, \Delta\gets\Delta - k_i \delta_i$\;
			\For{$j = 0, \dots, k_i-1$}{
				$g_{(i,j+1)} \gets g_{(i,j)} + \delta_i \lowrateop g_{(i,j)}$\;
			}
		}
		$(n, k) \gets (i, k_i)$\;
	}
	\Return $g_{(n,k)}$
\end{algorithm}

From the properties of lower transition operators, it follows that for all $\smash{i \in \{ 2, \dots, n-1 \}}$, $\smash{\norm{g_{(i-1,m)}}_{c} \leq \norm{g_{(i-2,m)}}_{c}}$.
Hence, the re-evaluated step size $\delta_i$ is indeed larger than (or equal to) the previous step size $\delta_{i-1}$.
The only exception to this is the final step size $\delta_n$: it might be that the remaining time $\Delta$ is smaller than $m \delta_n$, in which case we need to choose $k$ and $\delta_n$ such that $k \delta_n = \Delta$.

Theorem~\ref{the:AdaptiveApproximation} guarantees that the adaptive approximation of Algorithm~\ref{alg:Adaptive} indeed results in an actual error lower than the desired maximal error $\epsilon$.
Even more, it provides a method to compute an upper bound $\epsilon'$ of the actual error that is lower than the desired maximal error.
Finally, it also states that the adaptive method of Algorithm~\ref{alg:Adaptive} needs at most an equal number of iterations than the uniform method of Algorithm~\ref{alg:Uniform}.

\begin{theorem}
\label{the:AdaptiveApproximation}
	Let $\lowrateop$ be a lower transition rate operator, $f\in\setoffna$, $t \in \nonnegreals$, $\epsilon \in \strictlyposreals$ and $m \in \nats$.
	We use Algorithm~\ref{alg:Adaptive} to determine $n$ and $k$, and if applicable also $k_i$, $\delta_{i}$ and $g_{(i,j)}$.
	If $\norm{f}_{c} = 0$, $\norm{\lowrateop} = 0$ or $t = 0$, then $\norm{\lowtranopa{t} f - g_{(n,k)}} = 0$.
	Otherwise, we are guaranteed that
	\begin{align*}
		\norm{\lowtranopa{t} f - \Phi_{m,k}(\delta_1 \dots, \delta_n) f}
		=
		\norm{\lowtranopa{t} f - g_{(n,k)}}
		\leq \epsilon'
		&\coloneqq \sum_{i = 1}^{n} \delta_i^2 \norm{\lowrateop}^2 \sum_{j=0}^{k_i - 1} \norm{g_{(i,j)}}_{c}
		\leq \epsilon
	\end{align*}
	and that the total number of iterations has an upper bound:
	\begin{equation*}
		\sum_{i=1}^{n} k_i
		= (n-1) m + k
		\leq \left\lceil \max \left\{ \norm{\lowrateop} t/2, t^2 \norm{\lowrateop}^2 \norm{f}_{c}/ \epsilon \right\} \right\rceil.
	\end{equation*}
\end{theorem}
Again, we can determine $\epsilon'$ while running Algorithm~\ref{alg:Adaptive}.
An alternate---less tight---version of $\epsilon'$ can be obtained by replacing the sum of $\norm{g_{(i,j)}}_{c}$ for $j$ from $0$ to $k_{i}-1$ by $k_i \norm{g_{(i,0)}}_{c} = k_i \norm{g_{(i-1,m)}}_{c}$.
Determining this alternative $\epsilon'$ while running Algorithm~\ref{alg:Adaptive} adds negligible computational overhead compared to the $\epsilon'$ of Theorem~\ref{the:AdaptiveApproximation}, as $\norm{g_{(i-1,m)}}_{c}$ is needed to re-evaluate the step size anyway.

The reason why we only re-evaluate the step size $\delta$ after every $m$ iterations is twofold.
First and foremost, all we currently know for sure is that for all $\delta \in \nonnegreals$ such that $\delta \norm{\lowrateop} \leq 2$, all $m \in \nats$ and all $f\in\setoffna$, $\norm{(I + \delta \lowrateop)^{m} f}_{c} \leq \norm{f}_{c}$.
Re-evaluating the step size every $m$ iterations is therefore only justified if a priori we are certain that $\smash{\norm{(I + \delta_{i} \lowrateop)^{m} g_{(i-1,m)}}_{c} < \norm{g_{(i-1,m)}}_{c}}$.
We come back to this in Section~\ref{sec:ergodicity}.
A second reason is that there might be a trade-off between the time it takes to re-evaluate the step size and the time that is gained by the resulting reduction of the number of iterations.
The following numerical example illustrates this trade off.
\begin{binex}
\label{binex:AdaptiveApproximation}
	Recall that in Example~\ref{binex:UniformApproximation} we wanted to approximate $\lowtranopa{1} \indic{1}$ up to a maximal desired error $\epsilon = \num{1e-3}$.
	Instead of using the uniform method of Algorithm~\ref{alg:Uniform}, we now use the adaptive method of Algorithm~\ref{alg:Adaptive} with $m = 1$.
	The initial step size is the same as that of the uniform method, but because we re-evaluate the step size we only need \num{3437} iterations, as reported in Table~\ref{tab:ComparisonOfCompDuration}.
	We also find that in this case $\epsilon' = \num{1.00e-3}$, which is a coincidence.
	Nevertheless, the actual error of the approximation is $\num{0.108e-3}$, which is about ten times smaller than what we were aiming for.

	However, fewer iterations do not necessarily imply a shorter duration of the computations.
	Qualitatively, we can conclude the following from Table~\ref{tab:ComparisonOfCompDuration}.
	First, keeping track of $\epsilon'$ increases the duration, as expected.
	Second, the adaptive method is faster than the uniform method, at least if we choose $m$ large enough.
	And third, both methods yield an actual error that is at least an order of magnitude lower than the desired maximal error.
\end{binex}

	\section{Ergodicity}
\label{sec:ergodicity}
Let $\Phi_{m,k}(\delta_1, \dots, \delta_n) f$ be an approximation constructed using the adaptive method of Algorithm~\ref{alg:Adaptive}.
Re-evaluating the step size is then only justified if a priori we are sure that
\begin{equation*}
	\nicefrac{1}{2} \norm{(I + \delta_{i} \lowrateop)^m \Phi_{i-1} f}_{v} = \norm{g_{(i,m)}}_{c} < \norm{g_{(i-1,m)}}_{c} = \nicefrac{1}{2} \norm{\Phi_{i-1} f}_{v} \text{ for all } i \in \{ 1, \dots, n-1 \},
	\vspace{2pt}
\end{equation*}
where $\Phi_0 \coloneqq I$ and $\Phi_{i} \coloneqq (I + \delta_i \lowrateop)^m \Phi_{i-1}$.
As $(\Phi_{i-1} f) \in \setoffna$, this is definitely true if we require that
\vspace{-4pt}
\begin{equation}
\label{eqn:Ergodicity:IntroEqn}
	(\forall \delta \in \{ \delta_1, \dots, \delta_{n-1} \}) (\forall f \in \setoffna) ~ \norm{(I + \delta \lowrateop)^{m} f}_{v} < \norm{f}_{v}. 
	\vspace{2pt}
\end{equation}
In fact, since this inequality is invariant under translation or positive scaling of $f$, it suffices if
\vspace{1pt}
\begin{equation*}
	(\forall \delta \in \{ \delta_1, \dots, \delta_{n-1} \}) (\forall f \in \setoffna \colon 0 \leq f \leq 1) ~ \norm{(I + \delta \lowrateop)^{m} f}_{v} < 1.
	\vspace{1pt}
\end{equation*}
Readers that are familiar with (the ergodicity of) imprecise discrete-time Markov chains---see \citep{2012Hermans} or \smash{\citep{2013Skulj}}---will probably recognise this condition, as it states that the (weak) coefficient of ergodicity of $\smash{(I + \delta \lowrateop)^{m}}$ should be strictly smaller than 1.
For all lower transition operators $\lowtranop$, \cite{2013Skulj} define this (weak) \emph{coefficient of ergodicity} as
\begin{equation}
\label{eqn:CoeffOfErgod}
	\coefferga{\lowtranop}
	\coloneqq \max \left\{ \norm{\lowtranop f}_{v} \colon f \in \setoffna, 0 \leq f \leq 1 \right\}.
\end{equation}

	\subsection{Ergodicity of lower transition rate operators}
As will become apparent, whether or not combinations of $m \in \nats$ and $\delta \in \nonnegreals$ exist such that $\delta \norm{\lowrateop} \leq 2$ and $\coefferga{(I + \delta \lowrateop)^m} < 1$ is tightly connected with the behaviour of $\lowtranopa{t} f$ for large $t$.
\cite{2017DeBock} proved that for all lower transition rate operators $\lowrateop$ and all $f \in \setoffna$, the limit $\lim_{t \to \infty} \lowtranopa{t} f$ exists.
An important case is when this limit is a constant function for all $f$.
\begin{definition}[Definition~2 of \citep{2017DeBock}]
	The lower transition rate operator $\lowrateop$ is \emph{ergodic} if for all $f\in\setoffna$, $\lim_{t \to \infty} \lowtranopa{t} f$ exists and is a constant function.
\end{definition}

As shown by \cite{2017DeBock}, ergodicity is easily verified in practice: it is completely determined by the signs of $[\uprateop \indic{x}](y)$ and $[\lowrateop \indic{A}](z)$, for all $x,y \in \statespace$ and certain combinations of $z \in \statespace$ and $A \subset \statespace$.
It turns out that an ergodic lower transition rate operator $\lowrateop$ does not only induce a lower transition operator $\lowtranopa{t}$ that converges, it also induces discrete approximations---of the form $(I + \delta_{k} \lowrateop) \cdots (I + \delta_1 \lowrateop)$---with special properties.
The following theorem, which we consider to be one of the main results of this contribution, highlights this.
\begin{theorem}
\label{the:ContinuousErgodicity:CoefficientOfErgodicityOfApproximation}
	The lower transition rate operator $\lowrateop$ is ergodic if and only if there is some $n<\card{\statespace}$ such that $\coefferga{\Phi(\delta_1,\dots,\delta_{k})} < 1$ for one (and then all) $k \geq n$ and one (and then all) sequence(s) $\delta_1, \dots, \delta_k$ in $\strictlyposreals$ such that $\delta_i \norm{\lowrateop} < 2$ for all $i \in \{1, \dots, k\}$.
\end{theorem}

	\subsection{Ergodicity and the uniform approximation method}
\label{ssec:Ergodicity:UniformImprovement}
Theorem~\ref{the:ContinuousErgodicity:CoefficientOfErgodicityOfApproximation} guarantees that the conditions that were discussed at the beginning of this section are satisfied.
In particular, if the lower transition rate operator is ergodic, then there is some $n < \card{\statespace}$ such that $\coefferga{(I + \delta \lowrateop)^{m}} < 1$ for all $m \geq n$ and all $\delta \in \strictlyposreals$ such that $\delta \norm{\lowrateop} < 2$.
Consequently, if we choose $m \geq \card{\statespace}-1$ then re-evaluating the step size $\delta$ will---except maybe for the last re-evaluation---result in a new step size that is strictly greater than the previous one.
Therefore, we conclude that if the lower transition rate operator is ergodic, then using the adaptive method of Algorithm~\ref{alg:Adaptive} is certainly justified; it will result in fewer iterations, provided we choose a large enough $m$.

Another nice consequence of the ergodicity of a lower transition rate operator $\lowrateop$ is that we can prove an alternate a priori guaranteed upper bound for the error of uniform approximations.
\begin{proposition}
\label{prop:UniformApproximationErgodicError}
	Let $\lowrateop$ be a lower transition rate operator and fix some $f \in \setoffna$, $m, n \in \nats$ and $\delta \in \strictlyposreals$ such that $\delta \norm{\lowrateop} < 2$.
	If $\beta \coloneqq \coefferga{(I + \delta \lowrateop)^{m}} < 1$, then
	\vspace{-7pt}
\begin{equation*}
		\norm{\lowtranopa{t} f - \Psi_{t}(n)}
		\leq \epsilon_{e} \coloneqq m \delta^2 \norm{\lowrateop}^2 \norm{f}_{c} \frac{1 - \beta^{k}}{1 - \beta}
		\leq \epsilon_{d} \coloneqq \frac{m \delta^2 \norm{\lowrateop}^2 \norm{f}_{c}}{1 - \beta},
\end{equation*}
	where $t \coloneqq n \delta$ and $k \coloneqq \lceil \nicefrac{n}{m} \rceil$.
	The same is true for $\beta = \coefferga{\lowtranopa{m\delta}}$.
\end{proposition}
Interestingly enough, the upper bound $\epsilon_{d}$ is not dependent on $t$ (or $n$) at all!
This is a significant improvement on the upper bound of Theorem~\ref{the:UniformApproximationWithError}, as that upper bound is proportional to $t^2$.

By Theorem~\ref{the:ContinuousErgodicity:CoefficientOfErgodicityOfApproximation}, there always is an $m < \card{\statespace}$ such that $\coefferga{(I + \delta \lowrateop)^{m}} < 1$ for all $\delta \in \strictlyposreals$ such that $\delta \norm{\lowrateop} < 2$.
Thus, given such an $m$, we can easily improve Algorithm~\ref{alg:Uniform}.
After we have determined $n$ and $\delta$ with Algorithm~\ref{alg:Uniform}, we can simply determine the upper bound of Proposition~\ref{prop:UniformApproximationErgodicError}.
If $m (1 - \beta^k) < n (1 - \beta)$ (or $m < n(1 - \beta)$), then this upper bound is smaller than the desired maximal error $\epsilon$, and we have found a tighter upper bound on the actual error.
We can even go the extra mile and replace line \ref{line:Uniform:DetermineN} with a method that looks for the smallest possible $n \in \nats$ that yields
\vspace{-5pt}
\begin{equation*}
	m \delta^2 \norm{\lowrateop}^2 \norm{f}_{c} (1 - \beta^{k}) \leq (1 - \beta) \epsilon,
	\vspace{2pt}
\end{equation*}
where $k=\lceil \nicefrac{n}{m} \rceil$ and $\delta=\nicefrac{t}{n}$---and therefore also $\beta$---are dependent of $n$.
This method could yield a smaller $n$, but the time we gain by having to execute fewer iterations does not necessarily compensate the time lost by looking for a smaller $n$.
In any case, to actually implement these improvements we need to be able to compute $\beta\coloneqq\coefferga{(I + \delta \lowrateop)^m}$.

\begin{binex}
\label{binex:UniformErgodic}
	For the simple case of Example~\ref{binex:LTRO}, we can derive an analytical expression for $\coefferga{(I + \delta \lowrateop)}$ that is valid for all $\delta\in\reals_{\geq0}$ such that $\delta\norm{\lowrateop}\leq 2$.
	Therefore, we can use Proposition~\ref{prop:UniformApproximationErgodicError} to a priori determine an upper bound for the error.
	If we choose $m = 1$, then $\epsilon_{e} = \num{0.767e-3}$ and $\epsilon_{d} = \num{1.79e-3}$.
	Note that $\epsilon_{e} < \epsilon$, so we can probably decrease the number of iterations $n$.
	As reported in Table~\ref{tab:ComparisonOfCompDuration}, we find that $n = \num{6133}$ still suffices, and that this results in an approximation correct up to $\epsilon' = \num{0.560e-3}$, roughly two times smaller than the desired maximal error $\epsilon$.
	The actual error is \num{0.0437e-3}, roughly ten times smaller than $\epsilon$.
\end{binex}

	\subsection{Approximating the coefficient of ergodicity}
\label{ssec:CoeffOfErgod:Approximation}
Unfortunately, determining the exact value of $\coefferga{(I + \delta \lowrateop)^{m}}$---and of $\coefferga{\lowtranop}$ in general---turns out to be non-trivial and is often even impossible.
Nevertheless, the following theorem gives some---actually computable---lower and upper bounds for the coefficient of ergodicity.
\begin{theorem}
\label{the:CoeffOfErgod:Approximation}
	Let $\lowtranop$ be a lower transition operator.
	Then
	\begin{align}
		\coefferga{\lowtranop}
		&\leq \max \big\{ \max \{ [\uptranop \indic{A}](x) - [\lowtranop \indic{A}](y) \colon x,y \in \statespace \} \colon \emptyset \neq A \subset \statespace\big\}, \label{eqn:CoeffOfErgod:UpperBound} \\
		\coefferga{\lowtranop}
		&\geq \max \big\{ \max \{ [\lowtranop \indic{A}](x) - [\lowtranop \indic{A}](y) \colon x,y \in \statespace \} \colon \emptyset \neq A \subset \statespace\big\}. \label{eqn:CoeffOfErgod:LowerBound}
	\end{align}
\end{theorem}
The upper bound in Theorem~\ref{the:CoeffOfErgod:Approximation} is particularly useful in combination with Proposition~\ref{prop:UniformApproximationErgodicError}, as it allows us to replace $\beta\coloneqq\coefferga{(I + \delta \lowrateop)^{m}}$ with a guaranteed upper bound.

Of course, this only makes sense if this upper bound is strictly smaller than one.
In the previous versions of this pre-print, we claimed that for ergodic lower transition rate operators $\lowrateop$, this is always the case.
Unfortunately---and to our great regret---we have since then discovered that this result is in fact incorrect.
We have nonetheless included the (incorrect) statement so that we can easily refer to it, and have added a counterexample that demonstrates that it is indeed incorrect.
\begin{proposition}[Incorrect]
\label{prop:CoeffOfErgod:ErgodicUpperBound}
	Let $\lowrateop$ be an ergodic lower transition rate operator.
	Then there is some $n < \card{\statespace}$ such that, for all $k \geq n$ and $\delta_{1}, \dots, \delta_{k}$ in $\strictlyposreals$ such that $\delta_{i} \norm{\lowrateop} < 2$ for all $i \in \{ 1, \dots, k \}$, the upper bound for $\coefferga{\Phi(\delta_{1}, \dots, \delta_{k})}$ that is given by Eqn.~\eqref{eqn:CoeffOfErgod:UpperBound} is strictly smaller than one.
\end{proposition}
\begin{counterex}
	Consider the lower transition rate operator defined in \cref{binex:LTRO}, with \(\lowq{0} = 0 = \lowq{1}\), \(\upq{0} > 0\) and \(\upq{1} > 0\).
	One can easily verify that this lower transition rate operator is ergodic.

	Note that if \cref{prop:CoeffOfErgod:ErgodicUpperBound} were to be true, then for all \(\delta \in \strictlyposreals{}\) such that \(\delta \norm{\lowrateop{}} < 2\),
	\[
		\max \big\{ \max \{ [(I + \delta \uprateop) \indic{A}](x) - [(I + \delta \lowrateop) \indic{A}](y) \colon x,y \in \statespace \} \colon \emptyset \neq A \subset \statespace\big\}
		< 1.
	\]
	However, after some straightforward computations we obtain that
	\begin{multline*}
		\max \big\{ \max \{ [(I + \delta \uprateop) \indic{A}](x) - [(I + \delta \lowrateop) \indic{A}](y) \colon x,y \in \statespace \} \colon \emptyset \neq A \subset \statespace\big\} \\
		\geq [(I + \delta \uprateop) \indic{0}](0) - [(I + \delta \lowrateop) \indic{0}](1)
		= 1.
	\end{multline*}
\end{counterex}

\subsection{Approximating limit values}
The results that we have obtained earlier in this section naturally lead to a method to approximate $\lowtranopa{\infty} f \coloneqq \lim_{t \to \infty} \lowtranopa{t} f$ up to some maximal error.
This is an important problem in applications; for instance, \cite{2015Troffaes} try to determine $\lowtranopa{\infty}f$ for an ergodic lower transition rate operator that arises in their specific reliability analysis application.
The method they use is rather ad hoc: they pick some $t$ and $n$ and then determine the uniform approximation $\Psi_{t}(n) f$.
As $\norm{\Psi_{t}(n) f}_{v}$ is small, they suspect that they are close to the actual limit value.
They also observe that $\Psi_{2t}(4n) f$ only differs from $\Psi_{t}(n) f$ after the fourth significant digit, which they regard as further empirical evidence for the correctness of their approximation.
While this ad hoc method seemingly works, the initial values for $t$ and $n$ have to be chosen somewhat arbitrarily.
Also, this method provides no guarantee that the actual error is lower than some desired maximal error.

Theorem~\ref{the:ContinuousErgodicity:CoefficientOfErgodicityOfApproximation}, Proposition~\ref{prop:UniformApproximationErgodicError}, Theorem~\ref{the:CoeffOfErgod:Approximation} and the following stopping criterion allow us to propose a method that corrects these two shortcomings.
\begin{proposition}
\label{prop:StoppingCriterionWithErgodicity}
	Let $\smash{\lowrateop}$ be an ergodic lower transition rate operator and let $f \in \setoffna$, \smash{$t \in \nonnegreals$} and $\epsilon \in \strictlyposreals$.
	Let $s$ denote a sequence $\delta_1, \dots, \delta_k$ in $\nonnegreals$ such that $\sum_{i = 1}^{k} \delta_i = t$ and, for all $i \in \{1,\dots,k\}$, $\delta_{i} \norm{\lowrateop} \leq 2$.
	If $\norm{\lowtranopa{t} f - \Phi(s) f} \leq\nicefrac{\epsilon}{2}$ and $\norm{\Phi(s) f}_{c} \leq\nicefrac{\epsilon}{2}$, then for all $\Delta \in \nonnegreals$:
	\begin{align*}
		\abs{\lowtranopa{t + \Delta} f - \frac{\max \Phi(s) f + \min \Phi(s) f}{2}}
		\leq \epsilon
		~~~\text{and }~~
		\abs{\lowtranopa{\infty} f - \frac{\max \Phi(s) f + \min \Phi(s) f}{2}}
		\leq \epsilon.
	\end{align*}
\end{proposition}
Without actually stating it, we mention that a similar---though less useful---stopping criterion can be proved for non-ergodic transition rate matrices as well.

Our method for determining $\lowtranopa{\infty} f$ is now relatively straightforward.
Let $\lowrateop$ be an ergodic lower transition rate operator and fix some $f \in \setoffna$.
We can then approximate $\lowtranopa{\infty} f$ up to any desired maximal error $\epsilon \in \strictlyposreals$ as follows.
First, we look for some $m \in \nats$ and some---preferably large---$\delta  \in \strictlyposreals$ such that $\delta \norm{\lowrateop}<2$ and
\begin{equation*}
	 2m \delta^2 \norm{\lowrateop}^2 \norm{f}_{c} \leq (1 - \beta)\epsilon,
\end{equation*}
where $\beta \coloneqq \coefferga{(I + \delta \lowrateop)^{m}}$.
From Theorem~\ref{the:ContinuousErgodicity:CoefficientOfErgodicityOfApproximation}, we know that a possible starting point for $m$ is $\card{\statespace} - 1$.
If we do not have an analytical expression for $\coefferga{(I + \delta \lowrateop)^{m}}$, then we can instead use the guaranteed upper bound of Theorem~\ref{the:CoeffOfErgod:Approximation}---provided it is strictly smaller than one.
If no such $m$ and $\delta$ exist---for instance because the guaranteed upper bound on $\beta$ is too conservative---then this method does not work.
If on the other hand we do find such an $m$ and $\delta$, then we can keep on running the iterative step (line \ref{line:Uniform:IncrementOfG}) of Algorithm~\ref{alg:Uniform} until we reach the first index $i \in \nats$ such that $\norm{g_{i}}_{c} \leq \nicefrac{\epsilon}{2}$.
By Propositions~\ref{prop:UniformApproximationErgodicError} and \ref{prop:StoppingCriterionWithErgodicity}, we are now guaranteed that $(\max g_{i} + \min g_{i}) / 2$ is an approximation of $\lowtranopa{\infty} f$ up to a maximal error $\epsilon$.

Alternatively, we can fix a step size $\delta$ ourselves and use the method of Theorem~\ref{the:UniformApproximationWithError} to compute~$\epsilon'$.
In that case, we simply need to run the iterative scheme until we reach the first index $i$ such that $\norm{g_{i}}_{c} \leq \epsilon'$.
By Proposition~\ref{prop:StoppingCriterionWithErgodicity}, we are then guaranteed that the error $(\max g_{i} + \min g_{i})/2$ is an approximation of $\lowtranopa{\infty} f$ up to a maximal error $\epsilon=2 \epsilon'$.
The same is true if we replace $\epsilon'$ by the error $\epsilon_{e}$ that is used in Proposition~\ref{prop:UniformApproximationErgodicError}.

\begin{binex}
	Using the analytical expressions of Example~\ref{binex:AnalyticalExpressionsForAppliedLTO}, we obtain $\lowtranopa{\infty} \indic{1} \approx \num{9.5238095e-3}$.

	We want to approximate $\lowtranopa{\infty} \indic{1}$ up to a maximum error $\epsilon \coloneqq \num{1e-6}$.
	We observe that $m = \num{1}$ and $\delta \approx \num{3.485e-8}$ yield an $\epsilon_{d}$ that is lower than $\nicefrac{\epsilon}{2}$.
	After \num{196293685} iterations, the norm of the approximation is sufficiently small, resulting in the approximation $\lowtranopa{\infty} \indic{1} = \num{9.524(1)e-3}$.
	Alternatively, choosing $\delta = \num{1e-7}$ and continuing until $\norm{g_{i}}_{c} \leq \epsilon'$ yields the approximation $\lowtranopa{\infty} \indic{1} = \num{9.5242(8)e-3}$ after only \num{69572154} iterations.

	Mimicking \cite{2015Troffaes}, we also tried the heuristic method of increasing $t$ and $n$ until we observe empirical convergence.
	After some trying, we find that $t = \num{7}$ and $n = 7 \cdot \num{250} = 1750$ already yield an approximation with sufficiently small error: $\norm{\lowtranopa{\infty} \indic{1} - \Psi_{7}(1750) \indic{1}} \approx \num{7e-7} < \epsilon$.
	Note however that for non-binary examples, where $\lowtranopa{\infty} f$ cannot be computed analytically, this heuristic approach is unable to provide a guaranteed bound.
\end{binex}

	\section{Conclusion}
\label{sec:Conclusion}
We have improved an existing method and proposed a novel method to approximate $\lowtranopa{t} f$ up to any desired maximal error, where $\lowtranopa{t}f$ is the solution of the non-linear differential equation~\eqref{eqn:TDLTO:FunctionDifferentialEquation} that plays an essential role in the theory of imprecise continuous-time Markov chains.
As guaranteed by our theoretical results, and as verified by our numerical examples, our methods outperform the existing method by~\cite{2016Krak}, especially if the lower transition rate operator is ergodic.
For these ergodic lower transition rate operators, we also proposed a method to approximate $\lim_{t \to \infty} \lowtranopa{t} f$ up to any desired maximal error.

For the simple case of a binary state space, we observed in numerical examples that there is a rather large difference between the theoretically required number of iterations and the number of iterations that are empirically found to be sufficient.
Similar differences can---although this falls beyond the scope of our present contribution---also be observed for the lower transition rate operator that is studied in \citep{2015Troffaes}.
The underlying reason for these observed differences remains unclear so far.
On the one hand, it could be that our methods are still on the conservative side, and that further improvements are possible.
On the other hand, it might be that these differences are unavoidable, in the sense that guaranteed theoretical bounds come at the price of conservatism.
We leave this as an interesting line of future research.
Additionally, the performance of our proposed methods for systems with a larger state space deserves further inquiry.


\appendix
\acks{Jasper~De~Bock is a Postdoctoral Fellow of the Research Foundation - Flanders (FWO) and wishes to acknowledge its financial support.
The work in this paper was also partially supported by the H2020-MSCA-ITN-2016 UTOPIAE, grant agreement 722734.
Finally, the authors would like to express their gratitude to three anonymous reviewers, for their time, effort and constructive feedback.}

\bibliography{extended-references}

	\section{Extra material and proofs for Section~\ref{sec:Preliminaries}}
\label{app:Preliminaries}

\begin{definition}
	An operator $\norm{\cdot}$ on a linear vector space $\setoffn$ is a \emph{norm} if it maps $\setoffn$ to $\nonnegreals$ and if for all $a, b \in \setoffn$ and all $\mu \in \reals$,
	\begin{enumerate}[twocol, label=N\arabic*:, ref=(N\arabic*), series=Norm]
		\item \label{def:Norm:ScalarMult}
		$\norm{\mu a} = \abs{\mu} \norm{a}$,
		\item \label{def:Norm:TriangleInequality}
		$\norm{a + b} \leq \norm{a} + \norm{b}$,
		\item \label{def:Norm:NormZeroOnly}
		$\norm{a} = 0 \Leftrightarrow a = 0$.
	\end{enumerate}
	If an operator only satisfies \ref{def:Norm:ScalarMult} and \ref{def:Norm:TriangleInequality}, then it is called a \emph{seminorm}.
\end{definition}
It can be immediately checked that the maximum norm $\norm{\cdot}$ on $\setoffna$ is a proper norm, and similarly for the induced operator norm on non-negatively homogeneous operators from $\setoffna$ to $\setoffna$.
For all $f \in \setoffna$ we define the variation seminorm $\norm{\cdot}_{v}$ and the centred seminorm $\norm{\cdot}_{c}$ as
\begin{equation}
	\label{eqn:VariationNorm}
	\norm{f}_{v}
	\coloneqq \norm{f - \min{f}}
	= \max \{ \abs{f(x) - \min{f}} \colon x \in \statespace \}
	= \max f - \min f
\end{equation}
and
\begin{equation}
	\label{eqn:CentredNorm}
	\norm{f}_{c}
	\coloneqq \norm{f - \cent{f}}
	= \max \left\{ \abs{f(x) - \cent{f}} \colon x \in \statespace \right\}
	= (\max f - \min f) / 2,
\end{equation}
where $\cent{f} \coloneqq (\max{f} + \min{f})/2$.
Verifying that $\norm{\cdot}_{v}$ and $\norm{\cdot}_{c}$ are seminorms and not norms is straightforward.

\begin{proposition}
\label{prop:norms:properties}
	For all $f\in\setoffna$, all $\mu\in\reals$ and any non-negatively homogeneous operator $A$,
	\begin{enumerate}[label=N\arabic*:, ref=(N\arabic*), start=4]
		\item \label{prop:norm:CenteredEqVar}
			$\norm{f}_{c} = \norm{f}_{v} / 2$,
		\item \label{prop:norm:CenteredLeqNormal}
			$\norm{f}_{c} \leq \norm{f}$,
		\item \label{prop:norm:VarAddConstant}
			$\norm{f + \mu}_{v} = \norm{f}_{v}$,
		\item \label{prop:norms:BoundOnNormOf_Af}
			$\norm{A f} \leq \norm{A} \norm{f}$,
		\item \label{prop:norms:NormOfAB}
			$\norm{A B} \leq \norm{A} \norm{B}$.
	\end{enumerate}
\end{proposition}
\begin{proof}
	Properties \ref{prop:norm:CenteredEqVar}, \ref{prop:norm:CenteredLeqNormal} and \ref{prop:norm:VarAddConstant} follow almost immediately from the definitions of the centred and variation seminorms.
	Proofs for \ref{prop:norms:BoundOnNormOf_Af} and \ref{prop:norms:NormOfAB} can be found in \citep{2017DeBock}.
\end{proof}

The following properties of lower transition operators will turn out to be useful in the proofs.
\begin{proposition}
\label{prop:LowerTransitionOperator:Properties}
	Let $\lowtranop$, $\lowtranopa{1}$, $\lowtranopa{2}$, $\underline{S}_{1}$ and $\underline{S}_{2}$ be lower transition operators.
	Then for all $f, g \in \setoffna$ and all $\mu \in \reals$:
	\begin{enumerate}[twocol,resume*=LTO]
		\item \label{prop:LTO:BoundedByMinAndMax}
			$\min f \leq \lowtranop f \leq \uptranop f \leq \max f$;
		\item \label{prop:LTO:AdditionOfConstant}
			$\lowtranop (f + \mu) = \lowtranop (f) + \mu$;
		\item \label{prop:LTO:Monotonicity}
			$f \geq g \Rightarrow \lowtranop f \geq \lowtranop g$ and $\uptranop f \geq \uptranop g$;
		\item
			$\abs{\lowtranop f - \lowtranop g} \leq \uptranop (\abs{f - g})$;
		\item \label{prop:LTO:NormLowerThan1}
			$\norm{\lowtranop} \leq 1$;
		\item \label{prop:LTO:NonExpansiveness}
			$\norm{\lowtranop f - \lowtranop g} \leq \norm{f - g}$;
		\item \label{prop:LTO:BoundOnNormTBTB}
			$\norm{\lowtranop A - \lowtranop B} \leq \norm{A - B}$;
		\item \label{prop:LTO:VarNormTf}
			$\norm{\lowtranop f}_{v} \leq \norm{f}_{v}$;
	\end{enumerate}
	\begin{enumerate}[label=L\arabic*:, ref=(L\arabic*),resume=LTO]
		\item \label{prop:LTO:CompositionIsAlsoLTO}
			$\lowtranopa{1} \lowtranopa{2}$ is a lower transition operator;
		\item \label{prop:LTO:DifferenceIsNonNegativeHomogeneous}
			$(\lowtranopa{1} - \lowtranopa{2})$ is a non-negatively homogeneous operator;
		\item \label{prop:LTO:BoundOnDifferenceTfSf}
			$\norm{\lowtranopa{1} f - \underline{S}_{1} f}_{c} \leq \norm{\lowtranopa{1} f - \underline{S}_{1} f} \leq \norm{\lowtranopa{1} - \underline{S}_{1}} \norm{f}_{c}$;
		\item \label{prop:LTO:BoundOnDifferenceTTfSSf}
			$\norm{\lowtranopa{1} \lowtranopa{2} f - \underline{S}_{1} \underline{S}_{2} f}_{c} \leq \norm{\lowtranopa{1} \lowtranopa{2} f - \underline{S}_{1} \underline{S}_{2} f} \leq \norm{\lowtranopa{2} f - \underline{S}_{2} f} + \norm{\lowtranopa{1} - \underline{S}_{1}} \norm{\underline{S}_{2} f}_{c}$.
	\end{enumerate}
\end{proposition}
\begin{proof}
	Proofs for \ref{prop:LTO:BoundedByMinAndMax}--\ref{prop:LTO:BoundOnNormTBTB} and \ref{prop:LTO:CompositionIsAlsoLTO} can be found in \citep{2017DeBock}.

	\ref{prop:LTO:VarNormTf} follows almost immediately from \ref{prop:LTO:BoundedByMinAndMax} and Eqn.~\eqref{eqn:VariationNorm}:
	\[
		\norm{\lowtranop f}_{v}
		= \max \lowtranop f - \min \lowtranop f
		\leq \max f - \min f
		= \norm{f}_{v}.
	\]


	Note that for all $f \in \setoffna$ and all $\gamma \in \nonnegreals$,
	\[
		(\lowtranopa{1} - \lowtranopa{2})(\gamma f)
		= \lowtranopa{1}(\gamma f) - \lowtranopa{2} (\gamma f)
		= \gamma (\lowtranopa{1} f - \lowtranopa{2} f)
		= \gamma (\lowtranopa{1} - \lowtranopa{2}) (f),
	\]
	which proves \ref{prop:LTO:DifferenceIsNonNegativeHomogeneous}.

	Next, we prove \ref{prop:LTO:BoundOnDifferenceTfSf}.
	The first inequality follows from \ref{prop:norm:CenteredLeqNormal}.
	By \ref{prop:LTO:DifferenceIsNonNegativeHomogeneous}, $(\lowtranopa{1} - \underline{S}_{1})$ is a non-negatively homogeneous operator, such that
	\begin{align*}
		\norm{\lowtranopa{1} f - \underline{S}_{1} f}
		&= \norm{\lowtranopa{1} f - \cent{f} - \underline{S}_{1} f + \cent{f}}
		= \norm{\lowtranopa{1} (f - \cent{f}) - \underline{S}_{1} (f - \cent{f})} \\
		&= \norm{(\lowtranopa{1} - \underline{S}_{1}) (f - \cent{f})}
		\leq \norm{\lowtranopa{1} - \underline{S}_{1}} \norm{f - \cent{f}}
		= \norm{\lowtranopa{1} - \underline{S}_{1}} \norm{f}_{c},
	\end{align*}
	where the second equality follows from \ref{prop:LTO:AdditionOfConstant}, the inequality follows from \ref{prop:LTO:DifferenceIsNonNegativeHomogeneous} and \ref{prop:norms:BoundOnNormOf_Af} and the last equality follows from Eqn.~\eqref{eqn:CentredNorm}.

	\ref{prop:LTO:BoundOnDifferenceTTfSSf} can be proved similarly.
	Again, the first inequality of \ref{prop:LTO:BoundOnDifferenceTTfSSf} follows from \ref{prop:norm:CenteredLeqNormal}.
	To prove the second inequality of \ref{prop:LTO:BoundOnDifferenceTTfSSf}, we observe that
	\begin{align*}
		\norm{\lowtranopa{1} \lowtranopa{2} f - \underline{S}_{1} \underline{S}_{2} f}
		&= \norm{\lowtranopa{1} \lowtranopa{2} f - \lowtranopa{1} \underline{S}_2 f + \lowtranopa{1} \underline{S}_2 f - \underline{S}_{1} \underline{S}_{2} f} \\
		&\leq \norm{\lowtranopa{1} \lowtranopa{2} f - \lowtranopa{1} \underline{S}_2 f} + \norm{\lowtranopa{1} \underline{S}_2 f - \underline{S}_{1} \underline{S}_{2} f} \\
		&\leq \norm{\lowtranopa{2} f - \underline{S}_2 f} + \norm{\lowtranopa{1} \underline{S}_2 f - \underline{S}_{1} \underline{S}_{2} f} \\
		&\leq \norm{\lowtranopa{2} f - \underline{S}_2 f} + \norm{\lowtranopa{1} - \underline{S}_{1}} \norm{\underline{S}_{2} f}_{c},
	\end{align*}
	where the first inequality follows from \ref{def:Norm:TriangleInequality}, the second inequality follows from \ref{prop:LTO:NonExpansiveness} and the third inequality follows from \ref{prop:LTO:BoundOnDifferenceTfSf}.
\end{proof}


A linear lower transition rate operator $\lowrateop$---one for which \ref{def:LTRO:SuperAdditive} holds with equality---can be identified with a matrix $Q$ of dimension $\card{\statespace} \times \card{\statespace}$.
This matrix is called a \emph{transition rate matrix}, the $(x,y)$-component $Q(x,y)$ of which is equal to $[\lowrateop \indic{y}](x)$.
\begin{lemma}
\label{lem:BoundsOnElementsOfTransitionMatrix}
	Let $Q$ be a transition rate matrix.
	Then for all $x,y \in \statespace$ such that $x \neq y$,
	\begin{enumerate}[twocol, label=Q\arabic*:, ref=(Q\arabic*)]
		\item $Q(x,y) \geq 0$, \label{lem:RateMatrix:XY}
		\item $Q(x,x) = - \sum_{y\neq x} Q(x,y)$. \label{lem:RateMatrix:XX}
	\end{enumerate}
	Also,
	\[
		\norm{Q} = 2 \max \left\{ \abs{Q(x,x)} \colon x\in\statespace \right\}.
	\]
\end{lemma}
\begin{proof}
	Note that \ref{lem:RateMatrix:XY} follows immediately from \ref{def:LTRO:Sign}.
	From \ref{def:LTRO:Constant}, we find that for all $x\in\statespace$, $[Q \indic{\statespace}](x) = 0$.
	Using the linearity and \ref{def:LTRO:Constant} yields
	\[
		Q(x,x)
		= [Q \indic{x}](x)
		= \left[Q \left(1 - \sum_{y \neq x} \indic{y}\right)\right](x)
		= - \sum_{y \neq x} [Q \indic{y}](x) = \sum_{y \neq x} Q(x,y).
	\]

	It is a matter of straightforward verification to prove that
		\[
			\norm{Q} = \max \left\{ \sum_{y \in \statespace} \abs{Q(x,y)} \colon x\in\statespace \right\} = 2 \max \left\{ \abs{Q(x,x)} \colon x\in\statespace \right\}. \qedhere
		\]
\end{proof}

\begin{proposition}[Proposition~7.6 in \citep{2016Krak}]
	Let $\lowrateop$ be a lower transition rate operator.
	The associated set of dominating rate matrices $\setofdomratemat$, defined as
	\[
		\setofdomratemat
		\coloneqq \left\{ Q \text{ a transition rate matrix} \colon (\forall f \in \setoffna)~\lowrateop f \leq Q f \right\},
	\]
	is non-empty and bounded, and for all $f\in\setoffna$ there is some $Q\in\setofdomratemat$ such that $\lowrateop f = Q f$.
\end{proposition}

\begin{lemma}[Lemma~G.3 in \citep{2016Krak}]
\label{lem:NormRateMatixLowerThanNormRateOperator}
	Let $\lowrateop$ be a lower rate operator, then for any $Q \in \setofdomratemat$, $\norm{Q} \leq \norm{\lowrateop}$.
\end{lemma}

\begin{proposition}
\label{prop:LowerTransitionRateOperator:Properties}
	Let $\lowrateop$ be a lower transition rate operator.
	Then for all $f\in\setoffna$, all $\mu \in \reals$ and all $x,y\in\statespace$ such that $x\neq y$:
	\begin{enumerate}[twocol, resume*=LTRO]
		\item \label{prop:LTRO:LowUp}
			$\lowrateop f \leq \uprateop f$;
		\item \label{prop:LTRO:AdditionOfConstant}
			$\lowrateop (f + \mu) = \lowrateop f$;
		\item \label{prop:LTRO:Ixx}
			$- \norm{\lowrateop} / 2 \leq [\lowrateop \indic{x}](x) \leq [\uprateop \indic{x}](x) \leq 0$;
		\item
			$0 \leq \sum_{y \neq x} [\lowrateop \indic{x}](y) \leq \norm{\lowrateop} / 2$;
		\item \label{prop:LTRO:Norm}
			$\norm{\lowrateop} = 2 \max \{ \abs{[\lowrateop \indic{x}](x)} \colon x \in \statespace \}$.
	\end{enumerate}
\end{proposition}
\begin{proof}
	The properties \ref{prop:LTRO:LowUp} and \ref{prop:LTRO:AdditionOfConstant} are proved in \cite{2017DeBock}.
	Hence, we only prove the remaining properties.
	\begin{enumerate}[label=R\arabic*:,start=7]
		\item
			By the conjugacy of $\lowrateop$ and $\uprateop$,
			\begin{align*}
				[\uprateop \indic{x}](x)
				&= \left[\uprateop \left( 1 - \sum_{z \neq x} \indic{z} \right)\right](x)
				= - \left[\lowrateop \left( -1 + \sum_{z \neq x} \indic{z} \right)\right](x) \\
				&\leq - [\lowrateop (-1)](x) - \sum_{z \neq x} [\lowrateop \indic{z}](x),
			\end{align*}
			where the inequality follows from \ref{def:LTRO:SuperAdditive}.
			By \ref{def:LTRO:Constant} the first term is zero, such that
			\[
				[\uprateop \indic{x}](x) \leq - \sum_{z \neq x} [\lowrateop \indic{z}](x) \leq 0,
			\]
			where the second inequality follows from \ref{def:LTRO:Sign}.

			Recall that there is some $Q \in \setofdomratemat$ such that $\lowrateop \indic{x} = Q \indic{x}$.
			It holds that
			\begin{align*}
				[\lowrateop \indic{x}](x) = [Q \indic{x}](x) = Q(x,x) \geq -\frac{\norm{Q}}{2} \geq -\frac{\norm{\lowrateop}}{2},
			\end{align*}
			where for the first inequality we used Lemma~\ref{lem:BoundsOnElementsOfTransitionMatrix} and for the second inequality we used Lemma~\ref{lem:NormRateMatixLowerThanNormRateOperator}.

			The property now follows by combining the obtained lower bound for $[\lowrateop \indic{x}](x)$ and the obtained upper bound for $[\uprateop \indic{x}](x)$ with \ref{prop:LTRO:LowUp}.
		\item
			Recall from \ref{def:LTRO:Sign} that $[\lowrateop \indic{y}](x)$ is non-negative if $y \neq x$, such that $\sum_{y \neq x} [\lowrateop \indic{y}](x)$ is non-negative.
			Some manipulations yield
			\begin{align*}
				0
				\leq \sum_{y \neq x} [\lowrateop \indic{y}](x)
				\leq \left[\lowrateop \left(\sum_{y \neq x} \indic{y}\right)\right](x)
				&= - \left[\uprateop \left(- \sum_{y \neq x} \indic{y}\right)\right](x) \\
				&= - \left[\uprateop \left(1 - \sum_{y \neq x} \indic{y}\right)\right](x)
				= - [\uprateop \indic{x}](x) \\
				&\leq - [\lowrateop \indic{x}](x),
			\end{align*}
			where the second inequality follows from \ref{def:LTRO:SuperAdditive}, the first equality follows from conjugacy, the second equality follows from \ref{prop:LTRO:AdditionOfConstant}, and the final inequality follows from \ref{prop:LTRO:Ixx}.
			Also by \ref{prop:LTRO:Ixx}, we know that $- [\lowrateop \indic{x}](x)$ is non-negative and bounded above by $\norm{\lowrateop}/2$, hence
			\[
				0 \leq \sum_{y \neq x} [\lowrateop \indic{y}](x) \leq \frac{\norm{\lowrateop}}{2}.
			\]
		\item
			Let $\lowrateop$ be a lower transition rate operator.
			From \cite[R9]{2017DeBock} it follows that
			\[
				\norm{\lowrateop} \leq 2 \max_{x \in \statespace} \abs{ [\lowrateop \indic{x}](x) }.
			\]
			From \ref{prop:LTRO:Ixx}, however, we know that for all $x \in \statespace$, $\abs{ [\lowrateop \indic{x}](x) } \leq \norm{\lowrateop} / 2$.
			Combining these two inequalities yields $\norm{\lowrateop} = 2 \max \{\abs{[\lowrateop \indic{x}](x)} \colon x \in \statespace \}$. \qedhere
	\end{enumerate}
\end{proof}

\begin{proof}[Proof of Proposition~\ref{prop:IPlusDeltaQLowTranOp}]
	Fix some lower transition rate operator $\lowrateop$ and some $\delta \in \nonnegreals$.
	We first prove that $\delta \norm{\lowrateop} \leq 2$ implies that the operator $(I + \delta \lowrateop)$ is a lower transition operator.
	The operator $(I + \delta \lowrateop)$ trivially satisfies \ref{def:LTO:SuperAdditive} and \ref{def:LTO:NonNegativelyHom}, such that we only need to prove that it satisfies \ref{def:LTO:DominatesMin}.
	In order to do so, we fix some arbitrary $x \in \statespace$ and $f \in \setoffna$.
	It holds that
	\begin{align*}
		[(I + \delta \lowrateop)f](x)
		&= f(x) + \delta [\lowrateop f](x) \\
		&= f(x) + \delta [\lowrateop(f - \min f)](x) \\
		&= f(x) + \delta \left[\lowrateop \left(\sum_{y\in\statespace} (f(y) - \min f) \indic{y}\right)\right](x) \\
		&\geq f(x) + \delta (f(x) - \min f) [\lowrateop \indic{x}](x) + \delta \sum_{y \neq x} (f(y) - \min f) [\lowrateop \indic{y}](x) \\
		&\geq f(x) + \delta (f(x) - \min f) [\lowrateop \indic{x}](x) \\
		&\geq f(x) - \delta (f(x) - \min f) \frac{\norm{\lowrateop}}{2},
	\intertext{%
	where the second equality follows \ref{prop:LTRO:AdditionOfConstant}, the first inequality follows from \ref{def:LTRO:SuperAdditive}, the second inequality follows from \ref{def:LTRO:Sign} and the third inequality follows from \ref{prop:LTRO:Ixx}.
	Recall that by assumption $\delta \norm{\lowrateop} \leq 2$, and therefore
	}
		[(I + \delta \lowrateop)f](x)
		&\geq \min f.
	\end{align*}

	Next, we prove the reverse implication.
	Assume that $(I + \delta \lowrateop)$ is a transition rate operator.
	By \ref{prop:LTRO:Ixx} and \ref{prop:LTRO:Norm}, there is some $x \in \statespace$ such that $[\lowrateop \indic{x}](x) = - \norm{\lowrateop}/2$.
	Hence,
	\begin{align*}
		[(I + \delta \lowrateop) \indic{x}](x)
		&= \indic{x}(x) + \delta [\lowrateop\indic{x}](x)
		= 1 - \delta \frac{\norm{\lowrateop}}{2}.
	\intertext{
	If we now assume that $\delta \norm{\lowrateop} > 2$, then
	}
		[(I + \delta \lowrateop) \indic{x}](x)
		&< 0 \leq \min \indic{x},
	\end{align*}
	which, by \ref{def:LTO:DominatesMin}, contradicts the initial assumption that $(I + \delta \lowrateop)$ is a lower transition operator.
	This allows us to conclude that if $(I + \delta \lowrateop)$ is a lower transition operator, then $\delta \norm{\lowrateop} \leq 2$ .
\end{proof}
\begin{proof}[Proof of Proposition~\ref{prop:LTRO:PropositionNorm}]
	This proposition simply states \ref{prop:LTRO:Norm} of Proposition~\ref{prop:LowerTransitionRateOperator:Properties}.
\end{proof}

\begin{proof}[Proof of Example~\ref{binex:LTRO}]
	We can immediately verify that $\lowrateop$ satisfies \ref{def:LTRO:Constant}--\ref{def:LTRO:Sign}, such that it is indeed a lower transition rate operator.
\end{proof}

	\section{Extra material for Section~\ref{sec:MCs}}
We here give a slightly more detailed description of the differential equation of interest.
Recall from the beginning of Section~\ref{sec:MCs} that \cite{2015Skulj} proved that for any lower transition rate operator $\lowrateop$ and any $f \in \setoffna$, the differential equation
\begin{equation*}
	\frac{\mathrm{d}}{\mathrm{d} t} f_{t} = \lowrateop f_{t}
\end{equation*}
with initial condition $f_{0} \coloneqq f$ has a unique solution for all $t \in \nonnegreals$.
As mentioned by \cite{2017DeBock}, this differential equation actually determines a time-dependent operator $\lowtranopa{t}$: for all $t \in \nonnegreals$, $\lowtranopa{t} f \coloneqq f_{t}$.
Even more, \cite[Proposition~9]{2017DeBock} states that for all $t \in \nonnegreals$, the time-dependent operator $\lowtranopa{t}$ itself satisfies the differential equation
\begin{equation}
\label{eqn:TDLTO:DifferentialEquation}
	\frac{\mathrm{d} }{\mathrm{d} t} \lowtranopa{t} = \lowrateop \lowtranopa{t}
\end{equation}
with initial condition $\lowtranopa{0} \coloneqq I$.
\cite{2017DeBock} also shows that this operator $\lowtranopa{t}$ is a lower transition operator, and that it satisfies the semi-group property: for all $t_1, t_2 \in\nonnegreals$,
\begin{equation}
\label{eqn:TDLTO:SemiGroup}
	\lowtranopa{t_1+t_2} = \lowtranopa{t_1} \lowtranopa{t_2}.
\end{equation}

For a transition rate matrix, Eqn.~\eqref{eqn:TDLTO:DifferentialEquation} reduces to the linear differential equation
\[
	\frac{\mathrm{d}}{\mathrm{d} t} T_t = Q T_t
\]
with initial condition $T_{0} \coloneqq I$.
This differential equation is essential to precise continuous-time Markov chains, and is often referred to as the \emph{forward Kolmogorov} equation.
The solution to this differential equation is called the \emph{matrix exponential}, and is denoted by $T_t = e^{t Q}$.

\begin{proof}[Proof of \cref{binex:AnalyticalExpressionsForAppliedLTO}]
	Fix any $\delta \in \nonnegreals$ such that $\delta \norm{\lowrateop} \leq 2$, and let $f$ be an arbitrary element of $\setoffna$.
	We immediately obtain that if $f(0) \geq f(1)$, then
	\begin{align*}
		[\Phi(\delta) f](0)
		&= f(0) - \delta \upq{0} (f(0) - f(1))
		= f(0) - \delta \upq{0} \norm{f}_{v}, \\
		[\Phi(\delta) f](1)
		&= f(1) + \delta \lowq{1} (f(0) - f(1))
		= f(1) + \delta \lowq{1} \norm{f}_{v}.
	\intertext{Similarly, if $f(0) \leq f(1)$, then}
		[\Phi(\delta) f](0)
		&= f(0) + \delta \lowq{0} \norm{f}_{v}, \\
		[\Phi(\delta) f](1)
		&= f(1) - \delta \upq{1} \norm{f}_{v}.
	\end{align*}
	Therefore, if $f(0) \geq f(1)$ then
	\begin{align*}
		[\Phi(\delta) f](0) - [\Phi(\delta) f](1)
		&= \norm{f}_{v} (1 - \delta (\upq{0} + \lowq{1})),
	\intertext{and similarly if $f(0) \leq f(1)$, then}
		[\Phi(\delta) f](1) - [\Phi(\delta) f](0)
		&= \norm{f}_{v} (1 - \delta (\lowq{0} + \upq{1})).
	\end{align*}
	Consequently
	\begin{align*}
		f(0) \geq f(1) &\Rightarrow
		\begin{cases}
			[\Phi(\delta) f](0) \geq [\Phi(\delta) f](1) &\text{if } \delta (\upq{0} + \lowq{1}) \leq 1, \\
			[\Phi(\delta) f](0) \leq [\Phi(\delta) f](1) &\text{if } \delta (\upq{0} + \lowq{1}) \geq 1,
		\end{cases}
	\intertext{and}
		f(0) \leq f(1) &\Rightarrow
		\begin{cases}
			[\Phi(\delta) f](0) \leq [\Phi(\delta) f](1) &\text{if } \delta (\lowq{0} + \upq{1}) \leq 1, \\
			[\Phi(\delta) f](0) \geq [\Phi(\delta) f](1) &\text{if } \delta (\lowq{0} + \upq{1}) \geq 1.
		\end{cases}
	\end{align*}

	Fix some $f \in \setoffna$, some $t \in \nonnegreals$ and let $n \in \nats$ such that
	\begin{align*}
		t (\upq{0} + \lowq{1}) \leq n,
		t (\lowq{0} + \upq{1}) \leq n
		~\text{and}~
		t \norm{\lowrateop}\leq 2 n.
	\end{align*}
	In this case, we can use the results obtained above to obtain an analytical expression for $\Psi_{t}(n) f$.
	If $f(0) \geq f(1)$, then
	\begin{align*}
		[\Psi_{t}(n) f](0)
		&= f(0) - \frac{t}{n} \upq{0} \norm{f}_{v} \sum_{i = 0}^{n-1} \left(1 - \frac{t}{n} (\upq{0} + \lowq{1})\right)^{i}, \\
		[\Psi_{t}(n) f](1)
		&= f(1) + \frac{t}{n} \lowq{1} \norm{f}_{v} \sum_{i = 0}^{n-1} \left(1 - \frac{t}{n} (\upq{0} + \lowq{1})\right)^{i}.
	\intertext{Similarly, if $f(0) \leq f(1)$, then}
		[\Psi_{t}(n) f](0)
		&= f(0) + \frac{t}{n} \lowq{0} \norm{f}_{v} \sum_{i = 0}^{n-1} \left(1 - \frac{t}{n} (\lowq{0} + \upq{1})\right)^{i}, \\
		[\Psi_{t}(n) f](1)
		&= f(1) - \frac{t}{n} \upq{1} \norm{f}_{v} \sum_{i = 0}^{n-1} \left(1 - \frac{t}{n} (\lowq{0} + \upq{1})\right)^{i}.
	\end{align*}

	We now use Eqn.~\eqref{eqn:TDLTO:LimitFormula} to derive analytical expressions for the components of $\lowtranopa{t} f$.
	If $f(0) \geq f(1)$, then
	\begin{align*}
		[\lowtranopa{t} f](0)
		&= \lim_{n \to \infty} [\Psi_{t}(n) f](0) \\
		&= \lim_{n \to \infty}\Bigg( f(0) - \frac{t}{n} \upq{0} \norm{f}_{v} \sum_{i = 0}^{n-1} \left(1 - \frac{t}{n} (\upq{0} + \lowq{1})\right)^{i} \Bigg)\\
		&= f(0) - \upq{0} \norm{f}_{v} \lim_{n \to \infty} \frac{t}{n} \sum_{i = 0}^{n-1} \left(1 - \frac{t}{n} (\upq{0} + \lowq{1})\right)^{i}.
	\intertext{%
		Let us now assume that $\upq{0} + \lowq{1}>0$.
		If $t\neq0$ and $n$ is greater than the lower bounds mentioned above, the expression inside the parenthesis is bounded below by $0$ and strictly bounded above by $1$.
		Therefore,
	}
		[\lowtranopa{t} f](0)
		&= f(0) - \upq{0} \norm{f}_{v} \lim_{n \to \infty} \frac{t}{n} \frac{1 - \left(1 - \frac{t}{n} (\upq{0} + \lowq{1})\right)^{n}}{1 - \left(1 - \frac{t}{n} (\upq{0} + \lowq{1})\right)} \\
		&= f(0) - \frac{\upq{0}}{\upq{0} + \lowq{1}} \norm{f}_{v} \lim_{n \to \infty} \left(1 - \left(1 - \frac{t}{n} (\upq{0} + \lowq{1})\right)^{n} \right) \\
		&= f(0) - \frac{\upq{0}}{\upq{0} + \lowq{1}} \norm{f}_{v} \left(1 - e^{-t (\upq{0} + \lowq{1})} \right),
	\intertext{and}
		[\lowtranopa{t} f](1)
		&= f(1) + \frac{\lowq{1}}{\upq{0} + \lowq{1}} \norm{f}_{v} \left(1 - e^{-t (\upq{0} + \lowq{1})} \right).
	\end{align*}
	If $t=0$, the obtained expressions hold trivially.
	Completely analogous, if $\lowq{0} + \upq{1}>0$, the case $f(0) \leq f(1)$ yields
	\begin{align*}
		[\lowtranopa{t} f](0)
		&= f(0) + \frac{\lowq{0}}{\lowq{0} + \upq{1}} \norm{f}_{v} \left(1 - e^{-t (\lowq{0} + \upq{1})} \right) \\
		[\lowtranopa{t} f](1)
		&= f(1) - \frac{\upq{1}}{\lowq{0} + \upq{1}} \norm{f}_{v} \left(1 - e^{-t (\lowq{0} + \upq{1})} \right). \qedhere
	\end{align*}
\end{proof}

	\section{Extra material and proofs for Section~\ref{sec:EfficientComputation}}
\label{app:EfficientComputation}

In many of the following proofs, we frequently use the following lemma.
\begin{lemma}[Lemma~F.9 in \citep{2016Krak}]
\label{lem:BoundForErrorOfIPlusDeltaQ}
	Let $\lowrateop$ be a lower transition rate operator.
	For any $\delta \in \nonnegreals$, $\norm{\lowtranopa{\delta} - (I + \delta \lowrateop)} \leq \delta^2 \norm{\lowrateop}^2$.
\end{lemma}

\begin{lemma}
\label{lem:ExplicitErrorBound}
	Let $\lowrateop$ be a lower transition rate operator, $f\in\setoffna$ and $t \in \nonnegreals$.
	Let $s \coloneqq (\delta_1, \dots, \delta_k)$ be any sequence in $\nonnegreals$ such that $\sum_{i = 1}^{k} \delta_i = t$ and, for all $i \in \{ 1, \dots, k \}$, $\delta_{i} \norm{\lowrateop} \leq 2$.
	Then
	\begin{align*}
		\norm{\lowtranopa{t} f - \Phi(s) f}
		&\leq \sum_{i = 1}^{k} \delta_i^2 \norm{\lowrateop}^2 \norm{\Phi_{i-1} f}_{c}
	\intertext{and}
		\norm{\lowtranopa{t} f - \Phi(s) f}
		&\leq \sum_{i = 1}^{k} \delta_i^2 \norm{\lowrateop}^2 \norm{\lowtranopa{\Delta_{i-1}} f}_{c},
	\end{align*}
	where $\Phi_{0} \coloneqq I$ and $\Delta_{0} = 0$, and for all $i \in \{1, \dots, k \}$, $\Phi_{i} \coloneqq (I + \delta_i \lowrateop) \Phi_{i-1}$ and $\Delta_{i} \coloneqq \Delta_{i-1} + \delta_i$.
\end{lemma}
\begin{proof}
	By the semi-group property of Eqn.~\eqref{eqn:TDLTO:SemiGroup},
	\begin{align*}
		\norm{\lowtranopa{t} f - \Phi(s) f}
		&= \norm{\lowtranopa{\delta_k} \lowtranopa{t - \delta_k} f - (I + \delta_k \lowrateop) \Phi_{k-1} f}.
	\intertext{%
		By Proposition~\ref{prop:IPlusDeltaQLowTranOp}, the operator $(I + \delta_{i} \lowrateop)$ is a lower transition operator for all $i \in \{ 1, \dots, k \}$.
		Even more, \ref{prop:LTO:CompositionIsAlsoLTO} implies that the operator $\Phi_{i-1}$ is a lower transition transition operator for all $i \in \{ 1, \dots, k \}$.
		Recall that $\lowtranopa{\delta_k}$ and $\lowtranopa{t-\delta_{k}}$ are lower transition operators by definition, such that using \ref{prop:LTO:BoundOnDifferenceTTfSSf} and Lemma~\ref{lem:BoundForErrorOfIPlusDeltaQ} yields
	}
		\norm{\lowtranopa{t} f - \Phi(s) f}
		&\leq \norm{\lowtranopa{\delta_k} - (I + \delta_k \lowrateop)} \norm{\Phi_{k-1} f}_{c} + \norm{\lowtranopa{t - \delta_k} f - \Phi_{k-1} f} \\
		&\leq \delta_k^2 \norm{\lowrateop}^2 \norm{\Phi_{k-1} f}_{c} + \norm{\lowtranopa{t - \delta_k} f - \Phi_{k-1} f}.
	\intertext{Repeated application of the same trick yields}
		\norm{\lowtranopa{t} f - \Phi(s) f}
		&\leq \sum_{i = 1}^{k} \delta_i^2 \norm{\lowrateop}^2 \norm{\Phi_{i-1} f}_{c}.
	\end{align*}

	The second inequality of the statement can be proved in a completely similar manner.
\end{proof}

\begin{lemma}
\label{lem:LTRO:SpecialNoApproximationCase}
	Let $\lowrateop$ be a lower transition rate operator, $t \in \nonnegreals$ and $f \in \setoffna$.
	If $\norm{f}_{c} = 0$, $\norm{\lowrateop} = 0$ or $t = 0$, then $\norm{\lowtranopa{t} f - \Psi_{t}(0) f}=\norm{\lowtranopa{t} f - f} = 0$.
\end{lemma}
\begin{proof}
	If $\norm{f}_{c} = 0$, then $\min f = \max f$, or equivalently $f$ is a constant function.
	From \ref{prop:LTO:BoundedByMinAndMax} it follows that in this case $\lowtranopa{t} f = f$ for all $t \in \nonnegreals$.
	If $\norm{\lowrateop} = 0$, then $\lowrateop g = 0$ for all $g \in \setoffna$.
	Therefore
	\[
		\frac{\mathrm{d}}{\mathrm{d} t} \lowtranopa{t} f = \lowrateop \lowtranopa{t} f = 0 \text{ for all } t \in \nonnegreals.
	\]
	Consequently, $\lowtranopa{t} f = \lowtranopa{0} f = I f = f$.
	If $t = 0$, then we can simply use the initial condition: $\lowtranopa{t} f = \lowtranopa{0} f = I f = f$.

	In all three cases we find that $\lowtranopa{t} f = f$, and hence
	\[
		\norm{\lowtranopa{t} f - \Psi_{t}(0) f} = \norm{\lowtranopa{t} f - f} = \norm{f-f} = 0. \qedhere
	\]
\end{proof}

\begin{lemma}
\label{lem:UniformApproximationWithError}
	Let $\lowrateop$ be a lower transition rate operator, $f\in\setoffna$, $t \in \nonnegreals$, $\epsilon \in \strictlyposreals$ and $n \in \nats$, and define $\delta \coloneqq t / n$.
	If
	\[
		n \geq \max \left\{ \frac{t \norm{\lowrateop}}{2}, \frac{t^2 \norm{\lowrateop}^2 \norm{f}_{c} }{\epsilon} \right\},
	\]
	then we are guaranteed that
	\[
		\norm{\lowtranopa{t} f - \Psi_{t}(n) f}
		\leq \epsilon'
		\coloneqq \delta^2 \norm{\lowrateop}^2 \sum_{i=0}^{n-1} \norm{\left(I + \delta \lowrateop\right)^{i} f}_{c}
		\leq \epsilon.
	\]
\end{lemma}
\begin{proof}
	By Proposition~\ref{prop:IPlusDeltaQLowTranOp}, the operator $(I + \delta \lowrateop)$ is a lower transition operator if and only if $\delta \norm{\lowrateop} \leq 2$, or equivalently if and only if
	\begin{equation}
	\label{eqn:UniformApproximationWithError:Ineq1}
		n \geq \frac{t \norm{\lowrateop}}{2}.
	\end{equation}
	From now on, we assume that $n$ satisfies this inequality.
	Therefore, we may use Lemma~\ref{lem:ExplicitErrorBound} to yield
	\begin{equation}
	\label{eqn:UniformApproximationWithError:UpperBoundError}
		\norm{\lowtranopa{t} f - \Psi_{t}(n) f}
		\leq \sum_{i = 0}^{n-1} \delta^2 \norm{\lowrateop}^2 \norm{(I+\delta \lowrateop)^{i} f}_{c}.
	\end{equation}
	Note that for any $i \in \{0, \dots, n-1\}$, $(I + \delta \lowrateop)^{i}$ is a lower transition operator by \ref{prop:LTO:CompositionIsAlsoLTO}; hence it follows from \ref{prop:LTO:VarNormTf} that $\norm{(I+\delta\lowrateop)^{i} f}_{c} \leq \norm{f}_{c}$.
	Therefore
	\begin{align*}
		\norm{\lowtranopa{t} f - \Psi_{t}(n) f}
		\leq \sum_{i = 0}^{n-1} \delta^2 \norm{\lowrateop}^2 \norm{f}_{c} = \frac{t^2 \norm{\lowrateop}^2 \norm{f}_{c}}{n}.
	\end{align*}
	It is now obvious that if
	\begin{equation}
	\label{eqn:UniformApproximationWithError:Ineq2}
		n \geq \frac{t^2 \norm{\lowrateop}^2 \norm{f}_{c}}{\epsilon},
	\end{equation}
	then $\norm{\lowtranopa{t} f - \Psi_{t}(n) f} \leq \epsilon$.
	It also follows almost immediately from Eqn.~\eqref{eqn:UniformApproximationWithError:UpperBoundError} that if $n$ satisfies both Eqns.~\eqref{eqn:UniformApproximationWithError:Ineq1} and \eqref{eqn:UniformApproximationWithError:Ineq2}, then
	\[
		\norm{\lowtranopa{t} f - \Psi_t(n) f}
		\leq \epsilon'
		\coloneqq \delta^2 \norm{\lowrateop}^2 \sum_{i=0}^{n-1} \norm{(I + \delta \lowrateop)^{i} f}_{c}
		\leq \epsilon. \qedhere
	\]
\end{proof}

\begin{proof}[Proof of Theorem~\ref{the:UniformApproximationWithError}]
	First, we assume $t = 0$, $\norm{\lowrateop} = 0$ or $\norm{f}_{c} = 0$.
	In this case, $n = 0$ and $\delta = 0$.
	By Lemma~\ref{lem:LTRO:SpecialNoApproximationCase}, we find that
	\[
		\norm{\lowtranopa{t} f - g_{(0)}} = \norm{\lowtranopa{t} f - \Psi_{t}(0) f} = 0 < \epsilon.
	\]

	Next, we assume $t > 0$, $\norm{\lowrateop} > 0$ and $\norm{f}_{c} > 0$.
	In this case, the integer $n$ that is determined on line~\ref{line:Uniform:DetermineN} of Algorithm~\ref{alg:Uniform} is just the lowest natural number that satisfies the requirement of Lemma~\ref{lem:UniformApproximationWithError}, from which the stated follows immediately.
\end{proof}

\begin{lemma}
\label{lem:AdaptiveApproximation}
	Let $\lowrateop$ be a lower transition operator, $f \in \setoffna$, $t' \in \nonnegreals$, $\epsilon \in \strictlyposreals$, $n, m, k \in \nats$ and let $\delta_{1}, \dots, \delta_{n}$ be a sequence in $\nonnegreals$.
	If (i) $k \leq m$, (ii) $k \delta_{n} + \sum_{i = 1}^{n-1} m \delta_i = t'$, and (iii) for all $i \in \{ 1, \dots, n \}$, $\delta_{i} \norm{\lowrateop} \leq 2$ and
	\[
		t' \norm{\lowrateop}^2 \norm{\Phi_{i-1} f}_{c} \delta_i \leq \epsilon,
	\]
	where $\Phi_0 \coloneqq I$ and for all $i \in \{ 1, \dots, n-1 \}$, $\Phi_i \coloneqq (I + \delta_i \lowrateop)^{m} \Phi_{i-1}$; then
	\begin{align*}
		\norm{\lowtranopa{t'} f - \Phi_{m,k}(\delta_1,\dots,\delta_n) f}
		&\leq \epsilon'
		\coloneqq \sum_{i = 1}^{n} \delta_i^2 \norm{\lowrateop}^2 \sum_{j=0}^{k_i - 1} \norm{(I + \delta_{i} \lowrateop)^{j} \Phi_{i-1} f}_{c} \\
		&\leq \sum_{i = 1}^{n} k_{i} \delta_{i}^{2} \norm{\lowrateop}^2 \norm{\Phi_{i-1} f}_{c}
		\leq \epsilon,
	\end{align*}
	where $k_i \coloneqq m$ for all $i \in \{ 1, \dots, n-1 \}$ and $k_{n} \coloneqq k$.
\end{lemma}
\begin{proof}
	Assume that (i) $1 \leq k \leq m$, (ii) $k \delta_{n} + \sum_{i=1}^{n-1} m \delta_{i} = t'$, and (iii) for all $i \in \{ 1, \dots, n \}$, $\delta_{i} \norm{\lowrateop} \leq 2$.
	Observe that by Proposition~\ref{prop:IPlusDeltaQLowTranOp} and \ref{prop:LTO:CompositionIsAlsoLTO}, the operators $\Phi_0, \dots, \Phi_{n-1}$ are all lower transition operators.
	From Lemma~\ref{lem:ExplicitErrorBound}, it follows that
	\begin{align}
	\label{eqn:AdaptiveApproximation:BoundOnError}
		\norm{\lowtranopa{t'} f - \Phi_{m,k}(\delta_1,\dots,\delta_n) f}
		\leq \sum_{i = 1}^{n} \delta_{i}^2 \norm{\lowrateop}^2 \sum_{j = 0}^{k_{i}-1} \norm{(I + \delta_{i} \lowrateop)^{j} \Phi_{i-1} f}_{c}.
	\end{align}
	Hence, it is obvious that the contribution of the $i$-th approximation step to (the upper bound of) the error is
	\begin{equation}
	\label{eqn:AdaptiveApproximation:UpperBoundOnError}
		\delta_i^2 \norm{\lowrateop}^2 \sum_{j= 0}^{k_{i}-1} \norm{(I + \delta_{i} \lowrateop)^{j} \Phi_{i-1} f}_{c}
		\leq k_{i} \delta_i^2 \norm{\lowrateop}^2 \norm{ \Phi_{i-1} f}_{c},
	\end{equation}
	where the inequality follows from \ref{prop:LTO:VarNormTf}.
	We want that the contribution of the $i$-th approximation step to the error is proportional to its length $k_{i} \delta_i$.
	Therefore, we demand that the contribution of the $i$-th approximation step is bounded above by $k_{i} \delta_i \epsilon / t'$, which yields the condition
	\begin{equation}
	\label{eqn:AdaptiveApproximation:ConditionOnDelta}
		t' \delta_i \norm{\lowrateop}^2 \norm{\Phi_{i-1} f}_{c} \leq \epsilon.
	\end{equation}
	It is obvious that the conditions we have imposed on $\delta_{1}, \dots, \delta_{n}$ are those of the statement.
	Combining Eqns.~\eqref{eqn:AdaptiveApproximation:BoundOnError}, \eqref{eqn:AdaptiveApproximation:UpperBoundOnError} and \eqref{eqn:AdaptiveApproximation:ConditionOnDelta} yields
	\begin{align*}
		\norm{\lowtranopa{t} f - \Phi_{m,k}(\delta_1,\dots,\delta_n) f}
		&\leq \epsilon' \coloneqq \sum_{i = 1}^{n} \delta_i^2 \norm{\lowrateop}^2 \sum_{j = 0}^{k_{i}-1} \norm{(I + \delta_{i} \lowrateop)^{j} \Phi_{i-1} f}_{c} \\
		&\leq \sum_{i = 1}^{n} k_{i} \delta_{i}^{2} \norm{\lowrateop}^2 \norm{\Phi_{i-1} f}_{c}
		\leq \epsilon. \qedhere
	\end{align*}
\end{proof}

\begin{proof}[Proof of Theorem~\ref{the:AdaptiveApproximation}]
	We use Algorithm~\ref{alg:Adaptive} to determine $n$ and $k$, and if applicable also $k_i$, $\delta_{i}$ and $g_{(i,j)}$.
	If $\norm{f}_{c} = 0$, $\norm{\lowrateop} = 0$ or $t = 0$, then by Lemma~\ref{lem:LTRO:SpecialNoApproximationCase}
	\[
		\norm{\lowtranopa{t} f - g_{(0,m)}} = \norm{\lowtranopa{t} f - f} = 0 < \epsilon.
	\]

	We therefore assume that $\norm{f}_{c} > 0$, $\smash{\norm{\lowrateop} > 0}$ and $t > 0$, and let $\delta_1, \dots, \delta_n \in \strictlyposreals$ and $k \in \nats$ be determined by running Algorithm~\ref{alg:Adaptive}.
	Let $t' \coloneqq k \delta_n + \sum_{i=1}^{n-1} m \delta_{i} \leq t$.
	It is then a matter of straightforward verification that $\delta_1,\dots, \delta_n$ and $k$ satisfy the requirements of Lemma~\ref{lem:AdaptiveApproximation}: (i) $1 \leq k \leq m$, (ii) $k \delta_n + \sum_{j = 1}^{n-1} m \delta_j = t'$, and (iii) for all $i \in \{ 1, \dots, n\}$, $\delta_i\norm{\lowrateop}\leq2$ and
	\[
		t' \delta_{i} \norm{\lowrateop}^2 \norm{\Phi_{i-1} f}_{c}\leq t \delta_{i} \norm{\lowrateop}^2 \norm{\Phi_{i-1} f}_{c} = t \delta_{i} \norm{\lowrateop}^2 \norm{g_{(i-1,m)}}_{c} \leq \epsilon.
	\]
	Therefore,
	\begin{equation}
		\norm{\lowtranopa{t'} f - g_{(n,k)}}
		\leq \sum_{i=1}^{n} \delta_i^2 \norm{\lowrateop}^2 \sum_{j = 0}^{k_i - 1} \norm{g_{(i,j)}}_{c}
		\leq \sum_{i=1}^{n} k_i \delta_i^2 \norm{\lowrateop}^2 \norm{g_{(i-1,m)}}_{c}
		\leq \epsilon.\label{eq:tprimeformula}
	\end{equation}
	If $t'=t$, this concludes the proof of the first part of the statement.
	If $t'<t$, we have that $\norm{g_{(n,k)}}_{c} = 0$, which implies that there is some $\mu\in\reals$ such that $g_{(n,k)}=\mu$.
	Hence, it follows that
	\begin{equation*}
		\norm{\lowtranopa{t}f-g_{(n,k)}}
		=
		\norm{\lowtranopa{t}f-\mu}
		=\norm{\lowtranopa{t-t'}\lowtranopa{t'}f-\lowtranopa{t-t'}\mu}
		\leq\norm{\lowtranopa{t'}f-\mu}
		=\norm{\lowtranopa{t'}f-g_{(n,k)}},
	\end{equation*}
	where the second equality follows from Eqn.~\eqref{eqn:TDLTO:SemiGroup} and \ref{prop:LTO:AdditionOfConstant} and where the inequality follows from \ref{prop:LTO:NonExpansiveness}.
	Combined with Eqn.~\eqref{eq:tprimeformula}, this again implies the first part of the statement.

	To prove the final part of the statement, we assume that $\norm{f}_{c} > 0$, $\norm{\lowrateop} > 0$ and $t > 0$, and let $\delta_1, \dots, \delta_n \in \strictlyposreals$ and $k \in \nats$ be constructed by running Algorithm~\ref{alg:Adaptive}.
	We let $n_{u}$ denote the number of iterations of the uniform method:
	\[
		n_{u}
		\coloneqq \left\lceil \max \left\{ \frac{t \norm{\lowrateop}}{2}, \frac{t^2 \norm{\lowrateop}^2 \norm{f}_{c}}{\epsilon} \right\} \right\rceil.
	\]
	If we let $\delta_u \coloneqq t / n_{u}$, then obviously
	\[
		0 < \delta_{u} \leq \min \left\{ t, \frac{2}{\norm{\lowrateop}}, \frac{\epsilon}{t \norm{\lowrateop}^2 \norm{f}_{c}} \right\}.
	\]


	We now consider two cases: $n = 1$ and $n > 1$.
	We start with the case $n=1$.
	Let
	\[
		\delta_{1}^{*} \coloneqq \min \left\{ t, \frac{2}{\norm{\lowrateop}}, \frac{\epsilon}{t \norm{\lowrateop}^2 \norm{f}_{c}} \right\}.
	\]
	Since $n=1$, it then holds that $t \leq m \delta_{1}^{*}$ and/or $\norm{g_{(1,m)}}_{c} = 0$.
	We first assume that $t \leq m \delta_{1}^{*}$.
	Note that $\delta_{1}^{*}$ is strictly positive as we have assumed that $\norm{f}_{c}$, $\norm{\lowrateop}$ and $t$ are strictly positive.
	We let $k \coloneqq \left\lceil t / \delta_{1}^{*} \right\rceil$ and $\delta_{1} \coloneqq t / k$, such that
	\[
		k = \left\lceil \frac{t}{\delta_{1}^{*}} \right\rceil = \left\lceil \max \left\{ 1, \frac{t \norm{\lowrateop}}{2}, \frac{t^{2} \norm{\lowrateop}^2 \norm{f}_{c}}{\epsilon} \right\} \right\rceil.
	\]
	As in this case the definitions of $n_{u}$ and $k$ are equivalent, we find that $k + m (n-1) = k = n_{u}$.

	Next, we assume that $n = 1$ but $t > m \delta_{1}^{*}$.
	This can only be the case if $\norm{g_{(1,m)}}_{c} = 0$ and
	\[
		\delta_{1} \coloneqq \delta_{1}^{*} = \min \left\{ \frac{2}{\norm{\lowrateop}}, \frac{\epsilon}{t \norm{\lowrateop}^2 \norm{f}_{c}} \right\}.
	\]
	Therefore $\delta_{u} \leq \delta_{1}$, such that $n_u \geq t / \delta_{1} > m$.
	As the total number of iterations is $k = m$, it immediately follows that $m (n-1) + k = m < n_u$.

	Next, we consider the case $n > 1$.
	For all $i \in \{ 1, \dots, n-1 \}$,
	\[
		\delta_{i}
		\coloneqq \min \left\{ \frac{2}{\norm{\lowrateop}}, \frac{\epsilon}{t \norm{\lowrateop}^2 \norm{\Phi_{i-1} f}_{c}} \right\},
	\]
	where $\Phi_0 \coloneqq I$ and $\Phi_{i} \coloneqq (I + \delta_{i}\lowrateop)^{m} \Phi_{i-1}$ and this definition is valid because we previously assumed that $\norm{f}_{c} > 0$, $\norm{\lowrateop} > 0$ and $t > 0$.
	Note that our definition of $\delta_{i}$ differs from that of line~\ref{line:Adaptive:delta} in Algorithm~\ref{alg:Adaptive}: we have left out the upper bound $\Delta = t - \sum_{j = 1}^{i-1} m \delta_{j}$ because this upper bound only plays a part for the final step $\delta_{n}$.
	As by \ref{prop:LTO:BoundedByMinAndMax} $\norm{\Phi_{i} f}_{c} \leq \norm{\Phi_{i-1} f}_{c}$, we find that
	\[
		\delta_{u} \leq \delta_1 \leq \delta_2 \leq \cdots \leq \delta_{n-1},
	\]
	where the first inequality follows from the definition of $\delta_u$.
	As the step sizes that are used are all larger than the uniform step size, we intuitively expect that the number of necessary iterations will be bounded above by $n_{u}$.
	To formally prove this, we again distinguish two sub-cases: $k \delta_{n} + \sum_{i = 1}^{n-1} m \delta_{i} < t$ and $k \delta_{n} + \sum_{i = 1}^{n-1} m \delta_{i} = t$.

	We first consider the sub-case $k \delta_{n} + \sum_{i = 1}^{n-1} m \delta_{i} < t$.
	This can only occur if $\norm{g_{(n,m)}}_{c} = 0$ and $k = m$.
	As $m \delta_{n} < t - \sum_{i = 1}^{n-1} m \delta_{i}$ and $\norm{g_{(n-1,m)}}_{c} = \norm{\Phi_{n-1} f}_{c} > 0$,
	\[
		\delta_{n} = \left\{ \frac{2}{\norm{\lowrateop}}, \frac{\epsilon}{t \norm{\lowrateop}^2 \norm{\Phi_{n-1} f}_{c}} \right\} \geq \delta_{n-1},
	\]
	where the inequality follows from $\norm{\Phi_{n-2} f}_{c} \geq \norm{\Phi_{n-1} f}_{c}$.
	Note that
	\[
		m n \delta_{1} = \left(k + m (n-1)\right) \delta_{1} \leq k \delta_{n} + \sum_{i=1}^{n-1} m \delta_{i} < t = n_{u} \delta_{u},
	\]
	where the first inequality follows from the increasing character of $\delta_1, \dots, \delta_n$.
	If we divide both sides of the inequality by $\delta_{1}$, then we find that $m n < n_{u} \delta_{u} / \delta_{1}$.
	Using that $\delta_{u} \leq \delta_{1}$ now yields that the total number of iterations $k + (n-1) m = m n$ is strictly smaller than $n_{u}$.

	Next, we consider the sub-case $k \delta_{n} + \sum_{i = 1}^{n-1} m \delta_{i} = t$.
	Because $1 \leq k \leq m$ and $\delta_{n}>0$, $\sum_{i=1}^{n-1} m \delta_{i} < t = n_{u} \delta_{u}$.
	Hence, there is some $n_{u}' < n_{u}$ such that $n_{u}' \delta_{u} < \sum_{i=1}^{n-1} m \delta_{i} \leq (n_{u}' + 1) \delta_{u}$.
	The final step size $\delta_{n}$ is derived from the remaining time
	\begin{align*}
		t - \sum_{i=1}^{n-1} m \delta_{i}
		\eqqcolon \Delta
		&\geq n_{u} \delta_{u} - (n_{u}' + 1) \delta_{u} = (n_{u} - n_{u}' - 1) \delta_{u}, \\
		\Delta
		&< n_{u} \delta_{u} - n_{u}' \delta_{u} = (n_{u} - n_{u}') \delta_{u},
	\end{align*}
	where the first inequality follows from $\sum_{i=1}^{n-1} m\delta_{i} \leq (n_{u}' + 1) \delta_{u}$ and the second inequality follows from $\sum_{i=1}^{n-1} m\delta_{i} > n_{u}' \delta_{u}$.
	We first determine the maximal allowable final step size
	\[
		\delta_{n}^{*}
		\coloneqq \min \left\{ \Delta, \frac{2}{\norm{\lowrateop}}, \frac{\epsilon}{t \norm{\lowrateop}^{2} \norm{\Phi_{n-1} f}_{c}} \right\},
	\]
	and then determine the actual final step size as $\delta_{n} \coloneqq \Delta / k$, with $1 \leq k \coloneqq \lceil \Delta / \delta_{n}^{*} \rceil \leq m$.

	If $(n_{u} - n_{u}' - 1) > 0$, then $\Delta \geq (n_{u} - n_{u}' - 1) \delta_{u} \geq \delta_{u}$.
	Therefore, and because the two other upper bounds of $\delta_{n}^{*}$ are also greater than $\delta_{u}$, we find that $\delta_{n}^{*} \geq \delta_{u}$.
	From this, we infer that $k = \left\lceil \nicefrac{\Delta}{\delta_{n}^{*}} \right\rceil \leq \left\lceil \nicefrac{\Delta}{\delta_{u}} \right\rceil$.
	As $\Delta < (n_{u} - n_{u}') \delta_{u}$, we now find that $k \leq (n_{u} - n_{u}')$.
	Note that
	\[
		m(n-1) \delta_{1} \leq \sum_{i = 1}^{n-1} m \delta_{i} \leq (n_{u}' + 1) \delta_{u},
	\]
	where the first inequality follows from the non-decreasing character of $\delta_{1}, \dots, \delta_{n-1}$.
	Dividing both sides of the inequality by $\delta_{1}$ and using $\delta_{u} \leq \delta_{1}$ yields $m (n-1) \leq n_{u}' + 1$.

	If $m (n-1) < n_{u}' + 1$, then combining this strict inequality with the obtained upper bound for $k$ yields
	\[
		k + m(n-1) < (n_{u} - n_{u}') + (n_{u}' + 1) = n_{u} + 1,
	\]
	which implies that $k+m(n-1)\leq n_u$, as desired.

	If $m (n-1) = n_{u}' + 1$, then
	\[
		\Delta = t - \sum_{i = 1}^{n-1} m \delta_{i} \leq t - \sum_{i = 1}^{n-1} m \delta_{u} = (n_{u} - n_{u}' - 1) \delta_{u},
	\]
	where the inequality is allowed because $\delta_{u} \leq \delta_{1}, \dots, \delta_{n-1}$.
	As we previously proved that $\Delta \geq (n_{u} - n_{u}' - 1) \delta_{u}$, we obtain that $m (n-1) = n_{u}' + 1$ implies that $\Delta = (n_{u} - n_{u}' - 1) \delta_{u}$.
	As $\delta_{n}^{*} \geq \delta_{u}$, in this case we are guaranteed that $k = \lceil \nicefrac{\Delta}{\delta_{n}^{*}} \rceil =  \lceil \nicefrac{(n_{u} - n_{u}' - 1) \delta_{u}}{\delta_{n}^{*}} \rceil \leq (n_{u} - n_{u}' - 1)$.
	Hence, we again find that
	\[
		k + m(n-1) \leq (n_{u} - n_{u}' - 1) + (n_{u}' + 1) = n_{u},
	\]
	as desired.

	If $(n_{u} - n_{u}' - 1) = 0$, then $\Delta < (n_{u} - n_{u}') \delta_{u} = \delta_{u}$.
	As the two other upper bounds on $\delta_{n}^{*}$ are greater than $\delta_{u}$, this implies that $\delta_{n}^{*} = \Delta$.
	Consequently, $k = \lceil \nicefrac{\Delta}{\delta_{n}^{*}} \rceil = \lceil \nicefrac{\Delta}{\Delta} \rceil = 1$.
	Note that
	\[
		m(n-1) \delta_{1} \leq \sum_{i = 1}^{n-1} m \delta_{i} < n_{u} \delta_{u},
	\]
	from which it follows that $m (n-1) < n_{u}$.
	Hence, we find that $k + m (n-1) < 1 + n_{u}$, and therefore also, once more, that $k + m (n-1) \leq n_{u}$.
	This concludes the proof.
\end{proof}




	\section{A more thorough look at ergodicity}
\label{app:Extra:QualitativeErgodicity}
Before we prove the results of Section~\ref{sec:ergodicity}, we need to properly introduce the ergodicity of lower transition (rate) operators.
We explicitly chose not to do this in the main text, as the main focus of this contribution is approximating $\lowtranopa{t} f$.
Nevertheless, we now give a brief overview of the relevant literature, limiting ourselves to the qualitative point of view of \cite{decooman2009}, \cite{2012Hermans} and \cite{2017DeBock}.

	\subsection{Qualitatively characterising ergodicity of lower transition operators}
Recall that a lower transition rate operator is ergodic if and only if $\lowtranopa{t} f$ converges to a constant function for all $f \in \setoffna$.
\cite{2012Hermans} say something similar for lower transition operators.
\begin{definition}
	A lower transition operator $\lowtranop$ is \emph{ergodic} if, for all $f \in \setoffna$, the limit $\lim_{n \to \infty} \lowtranop^{n} f$ exists and is a constant function.
\end{definition}

The condition of this definition can, in general, not be checked in practice.
Nonetheless, \citet{2012Hermans} provide a necessary and sufficient condition for the ergodicity of a lower transition operator, based on the following definition.
\begin{definition}
\label{def:LowTranOp:RegularlyAbsorbing}
	The lower transition operator $\lowtranop$ is \emph{regularly absorbing} if it is (i) \emph{top class regular}, i.e.
	\[
		\statespacesub{PA}
		\coloneqq \left\{ x \in \statespace \colon (\exists n \in \nats)(\forall y \in \statespace)~[\uptranop^n \indic{x}](y) > 0 \right\} \neq 0,
	\]
	and (ii) \emph{top class absorbing}, i.e.
	\[
		(\forall y \in \statespace \setminus \statespacesub{PA})(\exists n \in \nats)~[\lowtranop^n \indic{\statespacesub{PA}}](y) > 0.
	\]
\end{definition}

\begin{proposition}[Proposition~3 from \citep{2012Hermans}]
\label{prop:DiscreteErgodicity:NecAndSuff}
	The lower transition operator $\lowtranop$ is ergodic if and only if it is regularly absorbing.
\end{proposition}

\cite{decooman2009} mention an equivalent way of looking at top class regularity that uses the ternary accessibility relation $\cdot \upreachda{\cdot} \cdot$.
\begin{definition}
\label{def:LowTranOp:PossiblyAccesible}
	Let $\lowtranop$ be any lower transition operator.
	For all $x,y \in \statespace$ and all $n \in \natz$, we say that \emph{$x$ is possibly accessible from $y$ in $n$ steps}, denoted by $y \upreachda{n} x$, if and only if $[\uptranop^n \indic{x}](y)>0$.
	If there is some $n \in \natz$ such that $y \upreachda{n} x$, then the state $x$ is simply said to be \emph{possibly accessible} from the state $y$, denoted by $y \upreachd x$.
\end{definition}

\begin{lemma}
\label{lem:LowTranOp:PossiblyAccesibleSequence}
	Let $\lowtranop$ be a lower transition operator, $x, y \in \statespace$ and $n \in \nats$.
	Then $y \upreachda{n} x$ if and only if there is a sequence $y = x_0, \dots, x_{n} = x$ in $\statespace$ such that for all $k \in \{ 1,\dots, n \}$, $[\uptranop \indic{x_{k}}](x_{k-1})>0$.
\end{lemma}
\begin{proof}
	Follows immediately from \cite[Proposition~4]{2012Hermans}.
\end{proof}

It can be almost immediately verified---for instance using Lemma~\ref{lem:LowTranOp:PossiblyAccesibleSequence}---that $\cdot \upreachda{\cdot} \cdot$ satisfies the three defining properties of a ternary accessibility relation:
\begin{enumerate}[label=A\arabic*:,ref=(A\arabic*)]
	\item $(\forall x,y \in \statespace)~x \upreachda{0} y \Leftrightarrow x = y$,
	\item \label{def:AccesRelation:xyz}
	$(\forall x,y,z \in \statespace) (\forall n,m \in \natz)~x \upreachda{n} y \text{ and } y \upreachda{m} z \Rightarrow x \upreachda{n+m} z$,
	\item $(\forall x \in \statespace)(\forall n \in \nats)(\exists y \in \statespace) x \upreachda{n} y$.
\end{enumerate}

The following proposition is the reason why we introduced the accessibility relation $\cdot \upreachda{\cdot} \cdot$.
\begin{proposition}[Proposition~4.3 from \citep{decooman2009}]
\label{prop:LowTranOp:AlternativeDefinitionOfTopClassRegularity}
	The lower transition operator $\lowtranop$ is top class regular if and only if
	\[
		\statespacesub{PA} = \{ x\in\statespace \colon (\exists n \in \nats)(\forall k \geq n)(\forall y \in \statespace)~y \upreachda{k} x \} \neq \emptyset.
	\]
\end{proposition}

\begin{lemma}
\label{lem:LowTranOp:TopClassRegularSpecialValues}
	If the lower transition operator $\lowtranop$ is top class regular, then for all $x \in \statespacesub{PA}$, all $y \in \statespacesub{PA}^{c}$ and all $k \in \nats$,
	\begin{align*}
		[\uptranop^k \indic{y}](x)
		&= 0
		& \text{and} & &
		[\lowtranop^k \indic{\statespacesub{PA}}](x)
		&= 1.
	\end{align*}
\end{lemma}
\begin{proof}
	Let $\lowtranop$ be a top class regular lower transition operator with regular top class $\statespacesub{PA}$.
	We first prove the first equality.
	To this end, we fix some arbitrary $x \in \statespacesub{PA}$ and $y \in \statespacesub{PA}^{c}$.
	Assume ex-absurdo that there is some $k \in \nats$ such that $[\uptranop^{k} \indic{y}](x) > 0$.
	By Definition~\ref{def:LowTranOp:PossiblyAccesible}, this assumption is equivalent to $x \upreachda{k} y$.
	By Proposition~\ref{prop:LowTranOp:AlternativeDefinitionOfTopClassRegularity}, there is some $n \in \nats$ such that for all $n \leq \ell \in \nats$ and $z \in \statespace$, $z \upreachda{\ell} x$.
	As a consequence of \ref{def:AccesRelation:xyz}, we find that for all $z \in \statespace$, $z \upreachda{\ell+k} y$, which in turn implies that $y \in \statespacesub{PA}$.
	However, this obviously contradicts our initial assumption, such that for all $k \in \nats$, $[\uptranop^{k} \indic{y}](x) = 0$.

	Next, we prove the second statement.
	From the conjugacy of $\lowtranop$ and $\uptranop$ and \ref{prop:LTO:AdditionOfConstant}, it follows that
	\begin{align*}
		\lowtranop \indic{\statespacesub{PA}}
		&= - \uptranop (-\indic{\statespacesub{PA}})
		= 1 - \uptranop (1 - \indic{\statespacesub{PA}})
		= 1 - \uptranop \indic{\statespacesub{PA}^{c}}.
	\end{align*}
	From the conjugacy of $\lowtranop$ and $\uptranop$ and \ref{def:LTO:SuperAdditive}, it follows that
	\[
		\uptranop \indic{\statespacesub{PA}^{c}}
		\leq \sum_{z \in \statespacesub{PA}^{c}} \uptranop \indic{z}.
	\]
	From the---already proven---first equality of the statement, we know that $\sum_{z \in \statespacesub{PA}^{c}} [\uptranop \indic{z}](x) = 0$, hence
	\[
		[\lowtranop \indic{\statespacesub{PA}}](x) = 1 - [\uptranop \indic{\statespacesub{PA}^{c}}](x)
		\geq 1 - \sum_{z \in \statespacesub{PA}^{c}} [\uptranop \indic{z}](x) = 1.
	\]
	Note that by \ref{prop:LTO:BoundedByMinAndMax}, $[\lowtranop \indic{\statespacesub{PA}}](x) \leq \max \indic{\statespacesub{PA}} = 1$.
	By combining the two obtained inequalities, we find that the the second equality of the statement holds for $k = 1$: $[\lowtranop \indic{\statespacesub{PA}}](x) = 1$.
	Next, fix some $k > 1$, and assume that the second equality holds for all $1 \leq \ell \leq k-1$.
	Then by the induction hypothesis and \ref{prop:LTO:BoundedByMinAndMax}, $\indic{\statespacesub{PA}} \leq \lowtranop^{k-1} \indic{\statespacesub{PA}}$.
	By \ref{prop:LTO:Monotonicity}, this implies that $\lowtranop \indic{\statespacesub{PA}} \leq \lowtranop^{k} \indic{\statespacesub{PA}}$.
	As by the induction hypothesis $[\lowtranop \indic{\statespacesub{PA}}](x) = 1$, we find that $[\lowtranop^{k} \indic{\statespacesub{PA}}](x) \geq 1$.
	It immediately follows from \ref{prop:LTO:BoundedByMinAndMax} and \ref{prop:LTO:CompositionIsAlsoLTO} that $\lowtranop^k \indic{\statespacesub{PA}} \leq 1$.
	Hence, we have shown that $[\lowtranop^{k} \indic{\statespacesub{PA}}](x) = 1$, which finalises the proof.
\end{proof}

The following proposition is an altered statement of \cite[Proposition~6]{2012Hermans}.
\begin{proposition}
\label{prop:LowTranOp:TopClassAbsorbingWithRecursion}
	Let $\lowtranop$ be a top class regular lower transition operator.
	Then $\lowtranop$ is top class absorbing if and only if $B_{n} = \statespace$, where $\{ B_k \}_{k\in\natz}$ is the sequence defined by the initial condition $B_0 \coloneqq \statespacesub{PA}$ and, for all $k \in \natz$, by the recursive relation
	\[
		B_{k+1}
		\coloneqq B_{k} \cup \big\{ x \in \statespace \setminus B_{k} \colon [\lowtranop \indic{B_{k}}](x) > 0 \big\} = \left\{ x \in \statespace \colon [\lowtranop \indic{B_{k}}](x) > 0 \right\},
	\]
	and where $n \leq \card{\statespace\setminus\statespacesub{PA}}$ is the first index such that $B_{n} = B_{n+1}$.
\end{proposition}
\begin{proof}
	Let $\lowtranop$ be a top class regular lower transition operator with regular top class $\statespacesub{PA}$.
	By \cite[Proposition~6]{2012Hermans}, $\lowtranop$ is top class absorbing if and only if $A_n = \emptyset$, where $A_n$ is the set determined by the initial condition $A_0 \coloneqq \statespace\setminus\statespacesub{PA}$ and, for all $k \in \natz$, by the recursive relation
	\begin{align*}
		A_{n+1}
		\coloneqq \{ x \in A_{k} \colon [\uptranop \indic{A_{k}}](x) = 1 \},
	\end{align*}
	and where $n \leq \card{\statespace\setminus\statespacesub{PA}}$ is the first index for which $A_{n} = A_{n+1}$.
	For any $k \in \natz$,
	\[
		\uptranop \indic{A_{k}} = - \lowtranop (- \indic{A_{k}}) =1 - \lowtranop (1 - \indic{A_{k}}) = 1 - \lowtranop \indic{\statespace\setminus A_{k}},
	\]
	where the first equality follows from the conjugacy of $\lowtranop$ and $\uptranop$ and the second equality follows from \ref{prop:LTO:AdditionOfConstant}.
	Therefore, for all $x \in A_k$, $[\uptranop \indic{A_{k}}](x) = 1$ if and only if $[\lowtranop \indic{\statespace\setminus A_{k}}](x) = 0$.
	Observe that $A_{k+1} \subseteq A_{k}$ and define $B_{k} \coloneqq \statespace \setminus A_{k}$ for all $k \in \natz$.
	Note that for all $k \in \natz$, $B_{k} \subseteq B_{k+1}$ and
	\[
		B_{k+1} \setminus B_{k}
		= A_{k} \setminus A_{k+1}
		= \{ x \in A_{k} \colon [\lowtranop \indic{\statespace\setminus A_{k}}](x) > 0 \}
		= \{ x \in \statespace \setminus B_{k} \colon [\lowtranop \indic{B_{k}}](x) > 0 \}.
	\]
	Observe that $B_0 = \statespace \setminus A_0 = \statespacesub{PA}$ and by the previous equality, for all $k \in \natz$,
	\[
		B_{k+1}
		= B_{k} \cup \{ x \in \statespace \setminus B_{k} \colon [\lowtranop \indic{B_{k}}](x) > 0 \}.
	\]

	We now prove by induction that
	\[
		B_{k+1}
		= \left\{ x \in \statespace \colon [\lowtranop \indic{B_{k}}](x) > 0 \right\} ~\text{for all}~k \in \natz.
	\]
	First, we consider the case $k = 0$.
	Recall from Lemma~\ref{lem:LowTranOp:TopClassRegularSpecialValues} that $[\lowtranop \indic{\statespacesub{PA}}](x_0) > 0$ for all $x_0 \in \statespacesub{PA}$.
	Hence,
	\begin{align*}
		B_{1}
		&= B_{0} \cup \{ x \in \statespace \setminus B_{0} \colon [\lowtranop \indic{B_{0}}](x) > 0 \} \\
		&= \{ x \in B_{0} \colon [\lowtranop \indic{B_{0}}](x) > 0 \} \cup \{ x \in \statespace \setminus B_{0} \colon [\lowtranop \indic{B_{0}}](x) > 0 \} \\
		&= \{ x \in \statespace \colon [\lowtranop \indic{B_{0}}](x) > 0 \}.
	\end{align*}
	Next, we fix some $i \in \nats$ and assume that the equality holds for all $k < i$.
	We now prove that the equality then also holds for $k = i$.
	Observe that $B_{k-1} \subseteq B_{k}$ implies $\indic{B_{k-1}} \leq \indic{B_{k}}$, which by \ref{prop:LTO:Monotonicity} implies that $\lowtranop \indic{B_{k-1}} \leq \lowtranop \indic{B_{k}}$.
	Therefore, for all $x\in B_{k}$, since the induction hypothesis implies that $[\lowtranop \indic{B_{k-1}}](x) > 0$, we find $[\lowtranop \indic{B_{k}}](x) \geq [\lowtranop \indic{B_{k-1}}](x) > 0$.
	Hence,
	\begin{align*}
		B_{k+1}
		&= B_{k} \cup \{ x \in \statespace \setminus B_{k} \colon [\lowtranop \indic{B_{k}}](x) > 0 \} \\
		&= \{ x \in B_{k} \colon [\lowtranop \indic{B_{k}}](x) > 0 \} \cup \{ x \in \statespace \setminus B_{k} \colon [\lowtranop \indic{B_{k}}](x) > 0 \} \\
		&= \{ x \in \statespace \colon [\lowtranop \indic{B_{k}}](x) > 0 \}. \qedhere
	\end{align*}
\end{proof}

The observant reader might have noticed that our definitions of top class regularity and top class absorption differ slightly from those in \cite{2012Hermans}, but they are actually entirely equivalent.
For top class regularity, we demand that there is some $n\in\nats$ such that $\uptranop^n \indic{x} > 0$.
By \ref{prop:LTO:BoundedByMinAndMax}, for any $k \geq n$ it then holds that $\smash{\uptranop^k \indic{x} > 0}$, which is what \citet{2012Hermans} demand.
For top class absorption, \citet{2012Hermans} demand that
\[
	(\forall y \in \statespacesub{PA}^{c})(\exists n\in\nats)~[\uptranop^n \indic{\statespacesub{PA}^{c}}](y) < 1,
\]
where $\statespacesub{PA}^{c}\coloneqq\statespace\setminus\statespacesub{PA}$.
Note that $[\uptranop^n \indic{\statespacesub{PA}^{c}}](y) = 1 - [\lowtranop^n \indic{\statespacesub{PA}}](y)$, such that their demand is equivalent to our demand
\[
	(\forall y \in \statespacesub{PA}^{c})(\exists n\in\nats)~[\lowtranop^n \indic{\statespacesub{PA}}](y) > 0.
\]
By Lemma~\ref{lem:LowTranOp:TopClassRegularSpecialValues}, for all $n\in\nats$ and all $y\in\statespacesub{PA}$, $[\lowtranop^n \indic{\statespacesub{PA}}](y) > 0$, such that we could actually demand that
\[
	(\forall y \in \statespace)(\exists n\in\nats)~[\lowtranop^n \indic{\statespacesub{PA}}](y) > 0.
\]

	\subsection{Qualitatively characterising ergodicity of lower transition rate operators}
We now turn to the ergodicity of imprecise continuous-time Markov chains.
A first and thorough study of the quantitative aspects concerning ergodicity was conducted by \citet{2017DeBock}.
We only recall the definitions and results from \citep{2017DeBock} that will be relevant to us in the remainder.

\begin{definition}
\label{def:LowRateOp:UpperReachable}
	A state $x\in\statespace$ is upper reachable from the state $y\in\statespace$, denoted by $y \upreachc x$, if (i) $x = y$, or (ii) there is some sequence $y=x_0, \dots, x_{n}=x$ in $\statespace$ of length $n+1 \geq 2$ such that for all $k\in\{1,\dots,n\}$, $[\uprateop \indic{x_{k}}](x_{k-1}) > 0$.
\end{definition}
Note that a state $x$ is always upper reachable from itself!
Rather remarkably, this definition of upper reachability is strikingly similar to the alternative condition of Lemma~\ref{lem:LowTranOp:PossiblyAccesibleSequence} for possible accessibility.
The links between these two definition will be made more explicit later.

\begin{lemma}
\label{lem:LowRateOp:UpperReachableShorterSequence}
	Let $\lowrateop$ be a lower rate operator, and $x,y\in\statespace$ such that $x \neq y$.
	Then $x$ is upper reachable from $y$ if and only if there is some sequence $y=x_0, \dots, x_{n}=x$ in $\statespace$ in which every state occurs at most once and for all $k\in\{1,\dots,n\}$, $[\uprateop \indic{x_{k}}](x_{k-1}) > 0$.
	Consequently, $n < \card{\statespace}$.
\end{lemma}
\begin{proof}
	The forward implication follows almost immediately from Definition~\ref{def:LowRateOp:UpperReachable}.
	Assume that $y \upreachc x$, then by Definition~\ref{def:LowRateOp:UpperReachable} there is some sequence $y=x_0, \dots, x_{n} = x$ in $\statespace$ such that for all $k\in\{1,\dots,n\}$, $[\uprateop \indic{x_{k}}](x_{k-1}) > 0$.
	Assume that there is a state $z \in \statespace$ that occurs more than once in this sequence.
	Then we can simply delete every element of the sequence from right after the the first occurrence of $z$ up to and including the last occurrence of $z$, and still have a valid sequence.
	If we continue this way, then we end up with a sequence in which every state occurs at most once.
	As every state occurs at most once, the length $n+1$ of the sequence is lower than or equal to $\card{\statespace}$.
	Consequently, $n < \card{\statespace}$.

	The reverse implication follows from the fact that the requirements if Definition~\ref{def:LowRateOp:UpperReachable} are trivially satisfied.
\end{proof}

\begin{lemma}
\label{lem:LowRateOp:PathOfArbitraryLength}
	Let $\lowrateop$ be a lower transition rate operator, and $x,y \in \statespace$ such that $y \upreachc x$.
	Then there is an integer $n < \card{\statespace}$ such that for all $k \geq n$ and all $\delta_1, \dots, \delta_k \in \strictlyposreals$ such that $\delta_{i} \norm{\lowrateop} < 2$ for all $i \in \{ 1, \dots, k \}$, there is a sequence $y = x_0, \dots, x_{k} = x$ in $\statespace$ such that $[(I + \delta_{i} \uprateop) \indic{x_{i}}](x_{i-1}) > 0$ for all $i \in \{ 1, \dots, k \}$.
\end{lemma}
\begin{proof}
	We first consider the special case $x = y$.
	For all $\delta \in \strictlyposreals$ such that $\delta \norm{\lowrateop} < 2$,
	\[
		[(I + \delta \uprateop) \indic{x}](x)
		= \indic{x}(x) + \delta [\uprateop \indic{x}](x)
		= 1 + \delta [\uprateop \indic{x}](x)
		> 0,
	\]
	where the inequality follows from \ref{prop:LTRO:Ixx}.
	Therefore, for all $k \in \nats$ and all $\delta_1, \dots, \delta_k \in \strictlyposreals$ such that for all $i \in \{ 1, \dots, k \}$, $\delta_{i} \norm{\lowrateop} < 2$, we find that $[(I + \delta_{i} \uprateop)\indic{x}](x) > 0$ for all $i \in \{ 1, \dots, k \}$.

	Next, we consider the case $y \neq x$.
	From Lemma~\ref{lem:LowRateOp:UpperReachableShorterSequence} we know that there is a sequence $S_y \coloneqq (y = x_0, \dots, x_{n} = x)$ in $\statespace$ such that every state occurs at most once---i.e. $n < \card{\statespace}$---and for all $i \in \{ 1,\dots,n \}$, $[\uprateop \indic{x_{i}}](x_{i-1}) > 0$.
	We fix an arbitrary $k \geq n$ and an arbitrary sequence $\delta_1, \dots, \delta_{k}$ in $\strictlyposreals$ such that for all $i \in \{ 1, \dots, k \}$, $\delta_{i} \norm{\lowrateop} < 2$.
	Note that for all $i \in \{ 1, \dots, n \}$,
	\[
		0 < \delta_i [\uprateop \indic{x_{i}}](x_{i-1})
		= \indic{x_{i}}(x_{i-1}) + \delta_{i} [\uprateop \indic{x_{i}}](x_{i-1})
		= [(I + \delta_i \uprateop) \indic{x_{i}}](x_{i-1}),
	\]
	where the inequality follows from $0 < \delta_i$ and the first equality is true because---by construction---$x_{i} \neq x_{i-1}$.
	Also, from the previous we know that for all $i \in \{ n+1, \dots, k \}$, $[(I + \delta_{i} \uprateop)\indic{x}](x) > 0$.
	Hence, appending the sequence $S_y$ with $(k - n)$ times $x$ yields a sequence $y = x_0, \dots, x_{k} = x$ in $\statespace$ such that for all $i \in \{ 1, \dots, k \}$, $[(I + \delta_{i} \uprateop) \indic{x_{i}}](x_{i-1}) > 0$.
\end{proof}

\begin{definition}
\label{def:LowRateOp:LowerReachable}
	A (non-empty) set of states $A \subseteq \statespace$ is lower reachable from the state $x$, denoted by $x \lowreachc A$, if $x \in B_n$, where $\{ B_k \}_{k \in \natz}$ is the sequence that is defined by the initial condition $B_0 \coloneqq A$ and for all $k \in \natz$ by the recursive relation
		\[
			B_{k+1}
			\coloneqq B_{k} \cup \left\{ y \in \statespace \setminus B_{k} \colon [\lowrateop \indic{B_{k}}](y) > 0 \right\},
		\]
		and $n \leq \card{\statespace \setminus A}$ is the first index for which $B_{k} = B_{k+1}$.
\end{definition}
Again, remark the striking similarity between Definition~\ref{def:LowRateOp:LowerReachable} and Proposition~\ref{prop:LowTranOp:TopClassAbsorbingWithRecursion}.

\begin{definition}
\label{def:LowRateOp:RegularlyAbsorbing}
	A lower transition rate operator $\lowrateop$ is \emph{regularly absorbing} if it is (i) \emph{top class regular}, i.e.
	\[
		\statespacesub{R}
		\coloneqq \left\{ x \in \statespace \colon (\forall y \in \statespace)~y \upreachc x \right\} \neq 0,
	\]
	and (ii) \emph{top class absorbing}, i.e.
	\[
		(\forall y \in \statespace\setminus\statespacesub{R})~y \lowreachc \statespace_R.
	\]
\end{definition}

\begin{theorem}[Theorem~19 in \citep{2017DeBock}]
\label{the:ContinuousErgodicity:NecessaryAndSufficient}
	A lower transition rate operator $\lowrateop$ is ergodic if and only if it is regularly absorbing.
\end{theorem}

Not surprisingly, these necessary and sufficient conditions for the ergodicity of lower rate matrices are rather similar to the necessary and sufficient conditions for ergodicity of lower transition operators given in Proposition~\ref{prop:DiscreteErgodicity:NecAndSuff}.

	\section{Extra material and proofs for Section~\ref{sec:ergodicity}}
\label{app:ergodicity}
Before we give any proofs, we first define the coefficient of ergodicity of an upper transition operator $\uptranop$:
\begin{equation}
\label{eqn:CoeffOfErgod:UpTranOpDefinition}
	\coefferga{\uptranop}
	\coloneqq \max \{ \norm{\uptranop f}_{v} \colon f \in \setoffna, 0 \leq f \leq 1 \}.
\end{equation}
\begin{proposition}
\label{prop:CoeffOfErgod:Properties}
	Let $\lowtranop$ and $\underline{S}$ be lower transition operators.
	For any $f \in \setoffna$,
	\begin{enumerate}[twocol, label=C\arabic*:, ref=(C\arabic*), series=CoeffErg]
		\item \label{prop:CoeffOfErgod:Bounds}
			$0 \leq \coefferga{\lowtranop} \leq 1$,
		\item \label{prop:CoeffOfErgod:BoundOnNormTf}
			$\norm{\lowtranop f}_{v} \leq \coefferga{\lowtranop} \norm{f}_{v}$,
		\item \label{prop:CoeffOfErgod:LowerEqualsUpper}
			$\coefferga{\uptranop} = \coefferga{\lowtranop}$,
		\item \label{prop:CoeffOfErgod:Composition}
			$\coefferga{\lowtranop \, \underline{S}} \leq \coefferga{\lowtranop} \coefferga{\underline{S}}$,
	\end{enumerate}
\end{proposition}
\begin{proof}
	\begin{enumerate}[label=C\arabic*:]
		\item
			Follows immediately from \ref{prop:LTO:BoundedByMinAndMax}.
		\item
			If $\norm{f}_{v} = 0$, then by \ref{prop:LTO:BoundedByMinAndMax} $\norm{\lowtranop f}_{v} = 0$, such that the stated holds.
			Therefore, we now assume---without loss of generality---that $\norm{f}_{v} > 0$.
			Note that $0 \leq (f - \min{f})/\norm{f}_{v} \leq 1$.
			Combining this with---in that order---\ref{prop:norm:VarAddConstant}, \ref{prop:LTO:AdditionOfConstant}, \ref{def:LTO:NonNegativelyHom}, \ref{def:Norm:ScalarMult} and Eqn.~\eqref{eqn:CoeffOfErgod}, we find that
			\begin{align*}
				\norm{\lowtranop f}_{v}
				&= \norm{\lowtranop f - \min{f}}_{v}
				= \norm{\lowtranop (f - \min{f})}_{v}
				= \norm{\norm{f}_{v} \lowtranop \left(\frac{f - \min{f}}{\norm{f}_{v}}\right)}_{v} \\
				&= \norm{\lowtranop \left(\frac{f - \min{f}}{\norm{f}_{v}}\right) }_{v} \norm{f}_{v} \\
				&\leq \coefferga{\lowtranop} \norm{f}_{v}.
			\end{align*}
		\item
			By Eqn.~\eqref{eqn:CoeffOfErgod},
			\begin{align*}
				\coefferga{\lowtranop}
				&= \max \left\{ \norm{\lowtranop f}_{v} \colon f \in \setoffna, 0 \leq f \leq 1 \right\} \\
				&= \max \left\{ \norm{1 - \lowtranop f}_{v} \colon f \in \setoffna, 0 \leq f \leq 1 \right\} \\
				&= \max \left\{ \norm{1 + \uptranop (-f)}_{v} \colon f \in \setoffna, 0 \leq f \leq 1 \right\} \\
				&= \max \left\{ \norm{\uptranop (1 - f)}_{v} \colon f \in \setoffna, 0 \leq f \leq 1 \right\} \\
				&= \max \left\{ \norm{\uptranop g}_{v} \colon g \in \setoffna, 0 \leq g \leq 1 \right\} \\
				&= \coefferga{\uptranop},
			\end{align*}
			where the second equality follows from \ref{prop:norm:VarAddConstant}, the third equality follows from the conjugacy of $\lowtranop$ and $\uptranop$, the fourth equality follows from \ref{prop:LTO:AdditionOfConstant}, the fifth equality follows from the fact that $0 \leq f \leq 1$ if and only if $0 \leq 1 - f \leq 1$, and the final equality follows from Eqn.~\eqref{eqn:CoeffOfErgod:UpTranOpDefinition}.
		\item
			By Eqn.~\eqref{eqn:CoeffOfErgod} and \ref{prop:CoeffOfErgod:BoundOnNormTf},
			\begin{align*}
				\coefferga{\lowtranop \, \underline{S}}
				&= \max \{ \norm{\lowtranop \, \underline{S} f}_{v} \colon f \in \setoffna, 0 \leq f \leq 1 \} \\
				&\leq \max \{ \coefferga{\lowtranop} \norm{\underline{S} f}_{v} \colon f \in \setoffna, 0 \leq f \leq 1 \} \\
				&= \coefferga{\lowtranop} \max \{ \norm{\underline{S} f}_{v} \colon f \in \setoffna, 0 \leq f \leq 1 \} = \coefferga{\lowtranop} \coefferga{\underline{S}}. \qedhere
			\end{align*}
	\end{enumerate}
\end{proof}

Theorem~21 in \cite{2013Skulj} highlights the usefulness of the coefficient of ergodicity.
\begin{theorem}[Theorem~21 in \citep{2013Skulj}]
\label{the:CoeffOfErg:StrictlySmallerThanOneIsNecAndSuffForErg}
	A lower transition operator $\lowtranop$ is ergodic if and only if there is some $k\in\nats$ such that $\coefferga{\lowtranop^k} < 1$.
\end{theorem}

\begin{proposition}
\label{prop:CoeffOfErgod:AlternativeFunctions}
	Let $\lowtranop$ be a a lower transition operator.
	Then
	\begin{align}
		\coefferga{\lowtranop}
		&= \max \left\{ \norm{\lowtranop f}_{v} \colon f\in\setoffna, \max f = 1, \min f = 0 \right\} \label{eqn:CoeffOfErgod:WithMax} \\
		&= \max \left\{ \norm{\lowtranop f}_{c} \colon f\in\setoffna, -1 \leq f \leq 1 \right\} \label{eqn:CoeffOfErgod:WithCentered} \\
		&= \max \left\{ \norm{\lowtranop f}_{c} \colon f\in\setoffna, \max f = 1, \min f = -1 \right\}. \label{eqn:CoeffOfErgod:WithCenteredAndMax}
	\end{align}
\end{proposition}
\begin{proof}[Proof of Proposition~\ref{prop:CoeffOfErgod:AlternativeFunctions}]
	Because of Eqn.~\eqref{eqn:CoeffOfErgod}, there is some $g\in\setoffna$ such that $0 \leq g \leq 1$ and $\norm{\lowtranop g}_{v} = \coefferga{\lowtranop}$.
	By \ref{prop:CoeffOfErgod:BoundOnNormTf}, $\norm{\lowtranop g}_{v} \leq \coefferga{\lowtranop} \norm{g}_{v}$, such that $\norm{g}_{v} = 1$, or equivalently $\max g = 1$ and $\min g = 0$.
	Hence, it follows from Eqn.~\eqref{eqn:CoeffOfErgod} that
	\begin{align*}
		\coefferga{\lowtranop}
		&= \max \{ \norm{\lowtranop f}_{v} \colon f \in \setoffna, \max f = 1, \min f = 0 \}.
	\intertext{Next, manipulating Eqn.~\eqref{eqn:CoeffOfErgod} yields}
		\coefferga{\lowtranop}
		&= \max \{ \norm{\lowtranop f}_{v} \colon f \in \setoffna, 0 \leq f \leq 1 \} \\
		&= \max \left\{ \norm{\lowtranop \left(f - \frac{1}{2}\right)}_{v} \colon f \in \setoffna, 0 \leq f \leq 1 \right\} \\
		&= \max \left\{ \frac{2}{2} \norm{\lowtranop \left(f - \frac{1}{2}\right)}_{v} \colon f \in \setoffna, 0 \leq f \leq 1 \right\} \\
		&= \max \left\{ \frac{1}{2} \norm{\lowtranop \left(2 f - 1\right)}_{v} \colon f \in \setoffna, 0 \leq f \leq 1 \right \} \\
		&= \max \left\{ \norm{\lowtranop \left(2 f - 1 \right)}_{c} \colon f \in \setoffna, 0 \leq f \leq 1 \right\},
	\intertext{%
	where the second equality follows from \ref{prop:norm:VarAddConstant} and \ref{prop:LTO:AdditionOfConstant}, the fourth equality follows from \ref{def:Norm:ScalarMult} and \ref{def:LTO:NonNegativelyHom}, and the final equality follows from Eqn.~\eqref{eqn:CentredNorm}.
	Note that for all $f\in\setoffna$, $0 \leq f \leq 1$ is equivalent to $-1 \leq (2f - 1) \leq 1$.
	Hence,
	}
		\coefferga{\lowtranop}
		&= \max \{ \norm{\lowtranop f}_{c} \colon f \in \setoffna, -1 \leq f \leq 1 \}.
	\end{align*}

	The proof of the final equality of the statement is now similar to that of the first.
\end{proof}

The following lemma is a more general version of Lemma~\ref{lem:LowTranOp:PossiblyAccesibleSequence}.
\begin{lemma}
\label{lem:LowTranOp:SignOfCompositionWithIndicator}
	Let $k \in \nats$ and $x,y\in \statespace$.
	For all arbitrary upper transition operators $\uptranopa{1}, \dots, \uptranopa{k}$, we define $\uptranopa{1:k} \coloneqq \uptranopa{k} \cdots \uptranopa{1}$.
	Then
	\[
		[\uptranopa{1:k} \indic{x}](y) \geq [\uptranopa{1} \indic{z_{1}}](z_{2}) \cdots [\uptranopa{k} \indic{z_{k}}](z_{k+1}),
	\]
	for any sequence $y = z_{k+1}, \dots, z_{1} = x$ in $\statespace$.
	Furthermore, $[\uptranopa{1:k} \indic{x}](y) > 0$ if and only if there is some sequence $y = z_{k+1}, \dots, z_{1} = x$ in $\statespace$ such that for all $i \in \{ 1, \dots, k \}$, $[\uptranopa{i} \indic{z_{i}}](z_{i+1}) > 0$.
\end{lemma}
\begin{proof}
	This proof is a straightforward generalisation of the proof of Proposition~4 in \citep{2012Hermans}.
	Fix some $k \in \nats$, some $x, y \in \statespace$ and some arbitrary upper transition operators $\uptranopa{1}, \dots, \uptranopa{k}$.
	We also define $\uptranopa{1:k} \coloneqq \uptranopa{k} \cdots \uptranopa{1}$, and note that by \ref{prop:LTO:CompositionIsAlsoLTO} this is also an upper transition operator.

	To prove the first part of the statement, we note that for all $i \in \{ 1, \dots, k \}$ and all $z_{i}, z_{i+1} \in \statespace$,
	\[
		\uptranopa{i} \indic{z_{i}}
		= \sum_{z \in \statespace} [\uptranopa{i} \indic{z_{i}}](z) \indic{z}
		\geq [\uptranopa{i} \indic{z_{i}}](z_{i-1}) \indic{z_{i+1}},
	\]
	where the inequality is allowed because by \ref{prop:LTO:BoundedByMinAndMax} the sum contains only non-negative terms.
	We fix any $z_{2} \in \statespace$, and use \ref{prop:LTO:Monotonicity} and this inequality to yield
	\begin{align*}
		\uptranopa{1:k} \indic{x}
		&= \uptranopa{1:k-1} \uptranopa{1} \indic{x}
		\geq \uptranopa{1:k-1} \left([\uptranopa{1} \indic{x}](z_{2}) \indic{z_{2}}\right)
		= [\uptranopa{1} \indic{x}](z_{2}) \uptranopa{1:k-1} \indic{z_{2}},
	\end{align*}
	where $\uptranopa{1:k-1} \coloneqq \uptranopa{k} \cdots \uptranopa{2}$---which by \ref{prop:LTO:CompositionIsAlsoLTO} is also an upper transition operator---and the final equality follows from \ref{def:LTO:NonNegativelyHom} and \ref{prop:LTO:BoundedByMinAndMax}.
	Repeated application of the same reasoning yields
	\[
		[\uptranopa{1:k} \indic{x}](y)
		\geq [\uptranopa{1} \indic{z_1}](z_{2}) \cdots [\uptranopa{k} \indic{z_{k}}](z_{k+1}),
	\]
	where $z_{k+1} \coloneqq y$, $z_{1} \coloneqq x$, and $z_{2}, \dots, z_{k}$ are arbitrary elements of $\statespace$.
	This proves the first part of the statement.

	The reverse implication of the second part of the statement follows immediately from the first part.
	We therefore only need to prove that the forward implication holds as well.
	To that end, we first note that
	\[
		[\uptranopa{1:k} \indic{x}](y) = \left[\uptranopa{k} \left( \sum_{z_{k} \in \statespace} [ \uptranopa{2:k} \indic{x}](z_{k}) \indic{z_{k}} \right) \right](y)
		\leq \sum_{z_{k} \in \statespace} [\uptranopa{2:k} \indic{x}](z_{k}) [\uptranopa{1} \indic{z_{k}}](y),
	\]
	where $\uptranopa{2:k} \coloneqq \uptranopa{k} \cdots \uptranopa{2}$ and the inequality follows from \ref{def:LTO:SuperAdditive}.
	Repeating this same reasoning another $(k-2)$ times yields
	\[
		[\uptranopa{1:k} \indic{x}](y) \leq \sum_{z_{2} \in \statespace} \sum_{z_{3} \in \statespace} \cdots \sum_{z_{k} \in \statespace} [\uptranopa{1} \indic{x}](z_{2}) [\uptranopa{2} \indic{z_{2}}](z_{3}) \cdots [\uptranopa{k} \indic{z_{k}}](y).
	\]
	If now $[\uptranopa{1:k} \indic{x}](y) > 0$, then---because all terms are non-negative due to \ref{prop:LTO:BoundedByMinAndMax}---at least one of the terms of the sum on the right hand side has to be strictly positive.
	Therefore, $[\uptranopa{1:k} \indic{x}](y) > 0$ implies that there is at least one sequence $y = z_{k+1}, \dots, z_{1} = x$ in $\statespace$ such that for all $i \in \{ 1, \dots, k\}$, $[\uptranopa{i} \indic{z_{i}}](z_{i+1}) > 0$.
\end{proof}

\begin{lemma}
\label{lem:LowTranOp:SignOfCompositionWithEvent}
	Let $k \in \nats$ and $A \subseteq \statespace$.
	For all arbitrary lower transition operators $\lowtranopa{1}, \dots, \lowtranopa{k}$, we define $\lowtranopa{1:k} \coloneqq \lowtranopa{k} \cdots \lowtranopa{1}$.
	Then
	\[
		c_1 \cdots c_k \indic{A_k} \leq \lowtranopa{1:k} \indic{A} \leq \indic{A_k}.
	\]
	In this expression, $A_k \subseteq \statespace$ is derived from the initial condition $A_{0} \coloneqq A$ and, for all $i \in \{ 1, \dots, k \}$, from the recursive relation
	\[
		A_{i}
		\coloneqq \{ x \in \statespace \colon [\lowtranopa{i} \indic{A_{i-1}}](x) > 0 \}.
	\]
	The non-negative real numbers $c_1, \dots, c_{k}$ are defined as
	\[
		c_{i}
		\coloneqq \min \left\{ [\lowtranopa{i} \indic{A_{i-1}}](x) \colon x \in A_{i} \right\} \text{ for all } i \in \{ 1, \dots, k \},
	\]
	with the convention that the minimum of an empty set is zero.
	Also, $A_k = \emptyset$ if and only if $c_i = 0$ for some $i \in \{ 1, \dots, k \}$.
\end{lemma}
\begin{proof}
	Let $\lowtranop$ be an arbitrary lower transition operator, and fix an arbitrary $A \subset \statespace$.
	We define the set $A' \coloneqq \{ x \in \statespace \colon [\lowtranop \indic{A}](x) > 0 \}$.
	On the one hand, from \ref{prop:LTO:BoundedByMinAndMax} it follows that $\lowtranop \indic{A} \leq \indic{A'}$.
	On the other hand, $\lowtranop \indic{A} \geq c \indic{A'}$, where we let
	\[
		c
		\coloneqq \min \left\{ [\lowtranop \indic{A}](x) \colon x \in A' \right\},
	\]
	with the convention that the minimum of an empty set is zero.
	Note that by \ref{prop:LTO:BoundedByMinAndMax}, $0 \leq c \leq 1$.
	Combining these two inequalities yields $c \indic{A'} \leq \lowtranop \indic{A} \leq \indic{A'}$.
	Proving the first part of the statement is now fairly trivial; we simply need to apply both inequalities and \ref{prop:LTO:Monotonicity} $k$ times.

	To prove the second part of the statement, we observe that $c_i = 0$ is equivalent to $A_{i} = \emptyset$.
	Therefore, we assume that there is some $i \in \{ 1, \dots, k \}$ for which $c_i = 0$ and $A_{i} = \emptyset$.
	If $c_k = 0$, then obviously $A_k = \emptyset$ and the stated is true.
	We therefore assume that $i < k$, and observe that by \ref{prop:LTO:BoundedByMinAndMax}, $\lowtranopa{i+1} \indic{A_{i}} = \lowtranopa{i+1} \indic{\emptyset} = 0$, and therefore $A_{i+1} = \emptyset$.
	Repeating the same reasoning, we find that $A_{j} = \emptyset$ and $c_{j} = 0$ for all $j \in \{ i, \dots, k \}$, which proves the stated.
\end{proof}

The following lemma is an alternate, slightly extended version of Proposition~\ref{prop:LowTranOp:TopClassAbsorbingWithRecursion}.
\begin{lemma}
\label{lem:LowTranOp:StrongerNecAndSuffTopClassAbsorption}
	Let $\lowtranop$ be a top class regular lower transition operator.
	Then $\lowtranop$ is top class absorbing if and only if $B_{n} = \statespace$, where $\{ B_{i} \}_{i \in \natz}$ is the sequence defined by the initial condition $B_{0} \coloneqq \statespacesub{PA}$ and the recursive relation
	\[
		B_{i}
		= B_{i-1} \cup \left\{ x \in \statespace \setminus B_{i-1} \colon [\lowtranop \indic{B_{i-1}}](x) > 0 \right\} ~~\text{for all}~i \in \nats,
	\]
	and where $n \leq \card{\statespace \setminus \statespacesub{PA}}$ is the first index such that $B_{n} = B_{n+1}$.
	Alternatively, $\lowtranop$ is top class absorbing if and only if there is some $m \in \natz$ such that $\lowtranop^{m} \indic{\statespacesub{PA}} > 0$, and in this case $n$ is the lowest such $m$.
\end{lemma}

\begin{proof}
	We first prove the forward implication.
	By Proposition~\ref{prop:LowTranOp:TopClassAbsorbingWithRecursion}, if $\lowtranop$ is top class absorbing then $B_{n} = \statespace$, where the sequence $\{B_{i}\}_{i \in \natz}$ is defined from the initial condition $B_{0} \coloneqq \statespacesub{PA}$ and, for all $i \in \nats$, from the recursive relation
	\[
		B_{i} \coloneqq B_{i-1} \cup \left\{ x \in \statespace \setminus B_{i-1} \colon [\lowtranop \indic{B_{i-1}}](x) > 0 \right\} = \left\{ x \in \statespace \colon [\lowtranop \indic{B_{i-1}}](x) > 0 \right\},
	\]
	and where $n \leq \card{\statespace \setminus \statespacesub{PA}}$ is the first index such that $B_{n} = B_{n+1}$.
	We can immediately verify that $\statespacesub{PA} = B_{0} \subseteq B_{1} \subseteq \cdots \subseteq B_{n} = \statespace$ and $B_{i} \setminus B_{i-1} \neq \emptyset$ for all $i \in \{ 1, \dots, n \}$.

	Observe that the sequence $B_{0}, \dots, B_{n}$ satisfies the conditions of Lemma~\ref{lem:LowTranOp:SignOfCompositionWithEvent}, such that or all $i \in \{ 1, \dots, n \}$,
	\[
		c_{1} \cdots c_{i} \indic{B_{i}} \leq \lowtranop^{i} \indic{\statespacesub{PA}} \leq \indic{B_{i}},
	\]
	where $c_1, \dots, c_{n}$ are strictly positive real numbers because $\emptyset \neq B_{1}, \dots, B_{n}$.
	From this we infer that $\min \lowtranop^{i} \indic{\statespacesub{PA}} > 0$ if and only if $B_{i} = \statespace$.
	As $B_{n} = \statespace$ and $B_{0}, \dots, B_{n-1} \neq B_{n}$, this confirms that indeed $\lowtranop^{n} \indic{\statespacesub{PA}} > 0$ and that $n$ is the lowest non-negative natural number for which this holds.

	Next, we prove the reverse implication.
	Let $B_{0}, \dots, B_{n}$ and $n$ be defined as in the statement.
	From the definition, it is obvious that $B_{i-1} \subseteq B_{i}$ for all $i \in \nats$.
	Also, if $n$ is the first index such that $B_{n} = B_{n+1}$, then $B_{i-1} \neq B_{i}$ for all $i \in \{ 1, \dots, n \}$ and $B_{n} = B_{n+i}$ for all $i \in \nats$.
	From $B_{0} = \statespacesub{PA}$ and $B_{i} \setminus B_{i-1} \neq \emptyset$ for all $i \in \{ 1, \dots, n \}$, we infer that indeed $n \leq \card{\statespace \setminus \statespacesub{PA}}$.
	If $B_{n} = \statespace$, then the sequence $B_0, \dots, B_{n}$ satisfies the conditions of Proposition~\ref{prop:LowTranOp:TopClassAbsorbingWithRecursion}, such that $\lowtranop$ is indeed top class absorbing.

	Let $B_{0}, \dots, B_{n},\dots$ be the sequence as defined in the statement.
	Similar to what we did in the proof of Proposition~\ref{prop:LowTranOp:TopClassAbsorbingWithRecursion}, we now verify using induction that
	\[
		B_{i} = \left\{ x \in \statespace \colon [\lowtranop \indic{B_{i-1}}](x) > 0 \right\} ~\text{for all}~i \in\nats.
	\]
	We first consider the case $i = 1$.
	By Lemma~\ref{lem:LowTranOp:TopClassRegularSpecialValues}, we know that $[\lowtranop \indic{\statespacesub{PA}}](x) > 0$ for all $x \in \statespacesub{PA}$.
	Hence,
	\begin{align*}
		B_{1}
		&= B_{0} \cup \left\{ x \in \statespace \setminus B_{0} \colon [\lowtranopa{\indic{B_{0}}}](x) > 0 \right\} \\
		&= \left\{ x \in B_{0} \colon [\lowtranopa{\indic{B_{0}}}](x) > 0 \right\} \cup \left\{ x \in \statespace \setminus B_{0} \colon [\lowtranopa{\indic{B_{0}}}](x) > 0 \right\} \\
		&= \left\{ x \in \statespace \colon [\lowtranop \indic{B_{0}}](x) > 0 \right\},
	\end{align*}
	where the second equality follows from the initial condition $B_{0} = \statespacesub{PA}$.
	Fix some $k \in \{ 1, \dots, n-1\}$, and assume that the alternate definition holds for all $i \leq k$.
	We now argue that in that case the stated also holds for $i = k+1$.
	By the induction hypothesis, $B_{k}$ contains all $x \in \statespace$ for which $[\lowtranop \indic{B_{k-1}}](x) > 0$.
	Also, it holds by definition that $B_{k-1} \subseteq B_{k}$.
	Using \ref{prop:LTO:Monotonicity}, we infer from $\indic{B_{k}} \geq \indic{B_{k-1}}$ that $[\lowtranop \indic{B_{k}}](x) \geq [\lowtranop \indic{B_{k-1}}](x) > 0$ for all $x\in B_{k}$.
	Hence,
	\begin{align*}
		B_{k+1}
		&= B_{k} \cup \left\{ x \in \statespace \setminus B_{k} \colon [\lowtranop \indic{B_{k}}](x) > 0 \right\} \\
		&= \left\{ x \in B_{k} \colon [\lowtranop \indic{B_{k}}](x) > 0 \right\} \cup \left\{ x \in \statespace \setminus B_{k} \colon [\lowtranop \indic{B_{k}}](x) > 0 \right\} \\
		&= \left\{ x \in \statespace \colon [\lowtranop \indic{B_{k}}](x) > 0 \right\}.
	\end{align*}

	Now that we know that
	\[
		B_{i} = \left\{ x \in \statespace \colon [\lowtranop \indic{B_{i-1}}](x) > 0 \right\} ~\text{for all}~i \in \nats,
	\]
	we observe that this equivalent definition of the sequence satisfies the conditions of the sequence in Lemma~\ref{lem:LowTranOp:SignOfCompositionWithEvent}.
	Moreover, as $\emptyset \neq B_{0} \subseteq B_{1} \subseteq \dots$, it follows from the second part of Lemma~\ref{lem:LowTranOp:SignOfCompositionWithEvent} that $c_{i} > 0$ for all $i \in \nats$.
	Also from Lemma~\ref{lem:LowTranOp:SignOfCompositionWithEvent}, we know that
	\[
		c_1 \cdots c_{i} \indic{B_{i}} \leq \lowtranop^{i} \indic{\statespacesub{PA}} \leq \indic{B_{i}} ~\text{for all}~ i \in \nats.
	\]
	Assume now that there is some $m \in \natz$ such that $\lowtranop^{m} \indic{\statespacesub{PA}} > 0$, and let $n$ be the lowest such $m$.
	Then for all $y \in \statespace \setminus \statespacesub{PA}$, $[\lowtranop^{n} \indic{\statespacesub{PA}}](y) > 0$, such that the second condition of Definition~\ref{def:LowTranOp:RegularlyAbsorbing} is satisfied and $\lowtranop$ is indeed top class absorbing.
	If $n = 0$, then $\statespacesub{PA} = \statespace = B_{0}$, and $n$ is indeed the first index for which $B_{n} = B_{n+1}$.
	If $n > 0$, then from the strict positivity of $c_{1}, \dots, c_{n}$ and the lower and upper bound for $\lowtranop^{i} \indic{\statespacesub{PA}}$ we infer that $B_{1}, \dots, B_{n-1} \neq \statespace$ and $B_{n} = \statespace$.
	We deduce from the recursive relation between $B_{0}, \dots, B_{n}, B_{n+1}$ that $n$ is indeed the first index for which $B_{n} = B_{n+1}$, which finalises this proof.
\end{proof}

\begin{proof}[Proof of Theorem~\ref{the:ContinuousErgodicity:CoefficientOfErgodicityOfApproximation}]
	We first prove the forward implication.
	To this end, we let $\lowrateop$ be an ergodic lower transition rate operator, and $n \coloneqq \card{\statespace} - 1$---we ignore the case $\card{\statespace} = 1$, as this case is trivially ergodic.
	We furthermore fix some $k \geq n$ and some $\delta_1, \dots, \delta_k$ in $\strictlyposreals$ such that for all $i \in \{ 1, \dots, k \}$, $\delta_{i} \norm{\lowrateop} < 2$.
	For all $i \in \{ 1, \dots, k\}$, we define $\lowtranopa{i} \coloneqq (I + \delta_{i} \lowrateop)$.
	By Proposition~\ref{prop:IPlusDeltaQLowTranOp}, the operators $\lowtranopa{1}, \dots, \lowtranopa{k}$ are lower transition operators, such that by \ref{prop:LTO:CompositionIsAlsoLTO} their composition $\lowtranopa{1:k} \coloneqq \lowtranopa{k} \cdots \lowtranopa{1}$ is also a lower transition operator.
	Note that the same holds for their conjugate upper transition operators, defined as $\uptranopa{i} \coloneqq (I + \delta_i \uprateop)$ and $\uptranopa{1:k} \coloneqq \uptranopa{k} \cdots \uptranopa{1}$.

	We now assume ex-absurdo that $\coefferga{\Phi(\delta_{1}, \dots, \delta_{k})} = \coefferga{\lowtranopa{1:k}} = 1$.
	As a consequence of Proposition~\ref{prop:CoeffOfErgod:AlternativeFunctions}, there is some $f^{*} \in \setoffna$ with $\min f^{*} = 0$ and $\max f^{*} = 1$ such that $\norm{\lowtranopa{1:k} f^{*}}_{v} = 1$.
	By construction and \ref{prop:LTO:BoundedByMinAndMax}, there are now some $y_0, y_1 \in \statespace$ such that $[\lowtranopa{1:k} f^{*}](y_0) = 0$ and $[\lowtranopa{1:k} f^{*}](y_1) = 1$.

	We define the---obviously non-empty---set
	\[
		\statespace^{*}
		\coloneqq \left\{ x \in \statespace \colon f^{*}(x) = 0 \right\},
	\]
	and distinguish two cases: either $\statespacesub{R} \cap \statespace^{*} \neq \emptyset$ or $\statespacesub{R} \cap \statespace^{*} = \emptyset$.

	We first consider the case $\statespacesub{R} \cap \statespace^{*} \neq \emptyset$, and fix any arbitrary $x^{*} \in \statespacesub{R} \cap \statespace^{*}$.
	Note that, by construction, $\indic{x^{*}} \leq 1 - f^{*}$.
	Using the conjugacy of $\lowtranopa{1:k}$ and $\uptranopa{1:k}$ and \ref{prop:LTO:Monotonicity}, we find that
	\[
		\uptranopa{1:k} \indic{x^{*}} \leq \uptranopa{1:k} (1 - f^{*}) = 1 + \uptranopa{1:k} (- f^{*}) = 1 - \lowtranopa{1:k} f^{*},
	\]
	where the first equality follows from \ref{prop:LTO:AdditionOfConstant} and the second equality follows from the conjugacy.
	From the previous inequality and \ref{prop:LTO:BoundedByMinAndMax}, it follows that
	\[
		0 \leq [\uptranopa{1:k} \indic{x^{*}}](y_1) \leq 1 - [\lowtranopa{1:k} f^{*}](y_1) = 0,
	\]
	and hence $[\uptranopa{1:k} \indic{x^{*}}](y_1) = 0$.
	From Lemma~\ref{lem:LowTranOp:SignOfCompositionWithIndicator}, it now follows that that
	\begin{equation}
	\label{eqn:ErogidictyOfApproximation:Contradiction1}
		0 = [\uptranopa{1:k} \indic{x^{*}}](y_1) \geq \prod_{i=1}^{k} \uptranopa{i} \indic{z_{i}}(z_{i+1}) = \prod_{i=1}^{k} [(I + \delta_{i} \uprateop) \indic{z_{i}}](z_{i+1})
	\end{equation}
	for any arbitrary sequence $y_1 = z_{k+1}, z_2, \dots, z_1 = x^{*}$ in $\statespace$.
	On the other hand, as $k \geq n = \card{\statespace} - 1$ and $x^{*} \in \statespacesub{R}$ it follows from Lemma~\ref{lem:LowRateOp:PathOfArbitraryLength} that there exists a sequence $y_1 = x_{k+1}, x_{k}, \dots, x_{1} = x^{*}$ in $\statespace$ such that $[(I + \delta_{i} \uprateop) \indic{x_{i}}](x_{i+1}) > 0$ for all $i \in \{1, \dots, k \}$.
	This obviously contradicts Eqn.~\eqref{eqn:ErogidictyOfApproximation:Contradiction1}.

	Next, we consider the case $\statespacesub{R} \cap \statespace^{*} = \emptyset$.
	In this case, $c \indic{\statespacesub{R}} \leq f^{*}$, where we let
	\[
		c \coloneqq \min \{ f^{*}(x) \colon x \in \statespacesub{R} \} > 0.
	\]
	From Lemma~\ref{lem:LowTranOp:SignOfCompositionWithEvent}, we know that
	\[
		c_{1} \cdots c_{k} \indic{A_k} \leq \lowtranopa{1:k} \indic{\statespacesub{R}},
	\]
	where $A_{0} \coloneqq \statespacesub{R}$ and, for all $i \in \{ 1, \dots, k \}$,
	\begin{align*}
		A_{i}
		&\coloneqq \{ x \in \statespace \colon [\lowtranopa{i} \indic{A_{i-1}}](x) > 0\}
		&\text{and} & &
		c_{i}
		&\coloneqq \min \{ [\lowtranopa{i} \indic{A_{i-1}}](x) \colon x \in A_{i} \}.
	\end{align*}
	As $c > 0$ and $c \indic{\statespacesub{R}} \leq f^{*}$, it follows from \ref{def:LTO:NonNegativelyHom} and  \ref{prop:LTO:Monotonicity} that $c \lowtranopa{1:k} \indic{\statespacesub{R}} \leq \lowtranopa{1:k} f^{*}$.
	Combining the two obtained inequalities yields
	\[
		c c_1 \cdots c_{k} \indica{A_{k}}{y_0} \leq c [\lowtranopa{1:k} \indic{\statespacesub{R}}](y_0) \leq [\lowtranopa{1:k} f^{*}](y_0) = 0.
	\]
	From the second part of Lemma~\ref{lem:LowTranOp:SignOfCompositionWithEvent}, it now follows that $y_{0} \notin A_{k}$.

	Nonetheless, we now prove that $A_{k} = \statespace$, an obvious contradiction.
	To that end, observe that for all $i \in \{ 1, \dots, k \}$,
	\begin{align*}
		A_{i}
		&= \{ x \in \statespace \colon [(I + \delta_{i} \lowrateop) \indic{A_{i-1}}](x) > 0 \} \\
		&= \{ x \in A_{i-1} \colon [(I + \delta_{i} \lowrateop) \indic{A_{i-1}}](x) > 0 \} \cup \{ x \in \statespace \setminus A_{i-1} \colon [(I + \delta_{i} \lowrateop) \indic{A_{i-1}}](x) > 0 \}.
	\intertext{%
		Note that for all $x_{i-1} \in A_{i-1}$, $\indic{A_{i-1}} \geq \indic{x_{i-1}}$.
		Also, from \ref{prop:LTRO:Ixx} it follows that $[(I + \delta_i \lowrateop) \indic{x_{i-1}}](x_{i-1}) > 0$.
		Using \ref{prop:LTO:Monotonicity} allows us to conclude that for all $x_{i-1} \in A_{i-1}$, $[(I + \delta_{i} \lowrateop) \indic{A_{i-1}}](x_{i-1}) > 0$.
		Therefore,
	}
		A_{i}
		&= A_{i-1} \cup \{ x \in \statespace \setminus A_{i-1} \colon [(I + \delta_{i} \lowrateop) \indic{A_{i-1}}](x) > 0 \} \\
		&= A_{i-1} \cup \{ x \in \statespace \setminus A_{i-1} \colon \indica{A_{i-1}}{x} + \delta_{i} [\lowrateop \indic{A_{i-1}}](x) > 0 \} \\
		&= A_{i-1} \cup \{ x \in \statespace \setminus A_{i-1} \colon [\lowrateop \indic{A_{i-1}}](x) > 0 \},
	\end{align*}
	where the third equality is allowed because $\delta_i > 0$.
	From this recursive relation, it is obvious that $\statespacesub{R} \subseteq A_{k}$.
	Even more, we can prove that $\statespacesub{R}^{c} \subseteq A_{k}$, which implies that $\statespacesub{R} \cup \statespacesub{R}^{c} = \statespace \subseteq A_{k} \subseteq \statespace$, and consequently $A_{k} = \statespace$.
	Indeed, note that the sequence $A_{0}, \dots, A_{k}$ is equal to the first $(k+1)$ terms of the sequence $\{B_{i}\}_{i \in \natz}$ that is defined in Definition~\ref{def:LowRateOp:LowerReachable} for $B_{0} = \statespacesub{R}$.
	As $\lowrateop$ was assumed to be ergodic and $k \geq \card{\statespace} - 1 \geq \card{\statespace \setminus \statespacesub{R}}$, it follows from Definitions~\ref{def:LowRateOp:LowerReachable} and \ref{def:LowRateOp:RegularlyAbsorbing} and Theorem~\ref{the:ContinuousErgodicity:NecessaryAndSufficient} that $\statespacesub{R}^{c} \subseteq B_{k}$.

	For both $\statespacesub{R} \cap \statespace^{*} \neq \emptyset$ and $\statespacesub{R} \cap \statespace^{*} = \emptyset$ we have obtained a contradiction, such that the ergodicity of $\lowrateop$ indeed implies the stated.

	Next, we prove the reverse implication.
	Fix some lower transition rate operator $\lowrateop$, and assume that there is some $k < \card{\statespace}$ and some $\delta_1, \dots, \delta_k \in \strictlyposreals$ such that $\delta_{i} \norm{\lowrateop} < 2$ for all $i \in \{ 1, \dots, k \}$ and
	\[
		\coefferga{\Phi(\delta_{1}, \dots, \delta_{k})} < 1.
	\]
	By Proposition~\ref{the:CoeffOfErg:StrictlySmallerThanOneIsNecAndSuffForErg} this implies that the lower transition operator $\lowtranopa{1:k} \coloneqq (I + \delta_{k} \lowrateop) \cdots (I + \delta_{1} \lowrateop)$ is ergodic.
	By Proposition \ref{prop:DiscreteErgodicity:NecAndSuff}, the ergodicity of $\lowtranopa{1:k}$ is equivalent to $\lowtranopa{1:k}$ being regularly absorbing, in the sense that
	\begin{enumerate}[label=(\roman*)]
		\item $\statespace_{1:k} \coloneqq \left\{ x \in \statespace \colon (\exists n \in \nats)(\forall y \in \statespace)~[(\uptranopa{1:k})^{n} \indic{x}](y) > 0 \right\} \neq \emptyset$;
		\item $(\forall y \in \statespace \setminus \statespace_{1:k})(\exists n \in \nats)~[(\lowtranopa{1:k})^{n} \indic{\statespace_{1:k}}](y) > 0$.
	\end{enumerate}

	Fix some $x^{*} \in \statespace_{1:k}$, and let $m \in \nats$ such that $(\uptranopa{1:k})^{m} \indic{x^{*}} > 0$.
	Fix an arbitrary $y \in \statespace$.
	Then by Lemma~\ref{lem:LowTranOp:SignOfCompositionWithIndicator}, there exists a sequence $y = x_{m+1}, \dots, x_{1} = x^{*}$ in $\statespace$ such that for all $i \in \{1, \dots, m\}$, $[\uptranopa{1:k} \indic{x_{i}}](x_{i+1}) > 0$.
	Again using Lemma~\ref{lem:LowTranOp:SignOfCompositionWithIndicator}, this implies that for all $i \in \{ 1, \dots, m \}$ there is a sequence $x_{i+1} = x_{i,k+1}, \dots, x_{i,1} = x_{i}$ in $\statespace$ such that for all $j \in \{ 1, \dots, k \}$,
	\[
		[(I + \delta_{j} \uprateop) \indic{x_{i,j}}](x_{i,j+1}) > 0.
	\]
	As such, we have now constructed one long sequence
	\[
		y = x_{m, k+1}, x_{m,k} \dots, x_{m,1} = x_{m-1, k+1}, x_{m-1,k}, \dots, x_{m-1,1} = x_{m-2, k+1} \dots, x_{1, 1} = x
	\]
	in $\statespace$.
	From this sequence we remove all ``loops'' (as we previously did in the proof of Lemma~\ref{lem:LowRateOp:UpperReachableShorterSequence}), and denote this shortened sequence by $y = z_{n'+1}, \dots, z_{1} = x^{*}$ with corresponding time steps $\delta_{n'}', \dots, \delta_{1}'$.
	Then for all $i \in \{ 1, \dots, n' \}$,
	\[
		0 < [(I + \delta_{i}' \uprateop) \indic{z_{i}}](z_{i+1}) = \indic{z_{i}}(z_{i+1}) + \delta_{i}' [\uprateop \indic{z_{i}}](z_{i+1}) = \delta_{i}' [\uprateop \indic{z_{i}}](z_{i+1}).
	\]
	As all $\delta_{i}'$ are strictly positive, we find that for all $i \in \{ 1, \dots, n' \}$, $[\uprateop \indic{z_{i}}](z_{i+1}) > 0$.
	By Definition~\ref{def:LowRateOp:UpperReachable}, this means that $y \upreachc x^{*}$.
	As $y$ was an arbitrary element of $\statespace$ and $x^{*}$ an arbitrary element of $\statespace_{1:k}$, $\statespace_{1:k} \subseteq \statespacesub{R}$ and hence $\lowrateop$ is top class regular.
	Furthermore, we can show that $\statespacesub{R} \subseteq \statespace_{1:k}$, such that $\statespacesub{R} = \statespace_{1:k}$.
	To that end, assume that $\statespacesub{R} \setminus \statespace_{1:k} \neq \emptyset$ and fix some arbitrary $x^{*} \in \statespacesub{R} \setminus \statespace_{1:k}$.
	Then by Definition~\ref{def:LowRateOp:RegularlyAbsorbing}, $y \upreachc x^{*}$ for all $y \in \statespace$.
	By Lemmas~\ref{lem:LowRateOp:PathOfArbitraryLength} and \ref{lem:LowTranOp:SignOfCompositionWithIndicator}, for all $y \in \statespace$ there is an integer $n_{y}$ such that for all $\ell \geq n_{y}$, $[(\uptranopa{1:k})^{\ell} \indic{x^{*}}](y) > 0$.
	Hence, if we let $m \coloneqq \max \{ n_{y} \colon y \in \statespace \}$, then $[(\uptranopa{1:k})^{m} \indic{x^{*}}](y) > 0$ for all $y \in \statespace$.
	By Definition~\ref{def:LowTranOp:RegularlyAbsorbing}, this implies that $x^{*} \in \statespace_{1:k}$.
	However, this contradicts our assumption that $x^{*} \in \statespacesub{R} \setminus \statespace_{1:k}$, such that $\statespacesub{R} \setminus \statespace_{1:k} = \emptyset$ and hence indeed $\statespacesub{R} \subseteq \statespace_{1:k}$.

	We now show that (ii) implies that $\lowrateop$ is top class absorbing.
	Since $\lowtranopa{1:k}$ is top class regular and absorbing, and because $\statespace_{1:k}=\statespace_{R}$, it follows from Lemma~\ref{lem:LowTranOp:StrongerNecAndSuffTopClassAbsorption} that there is some $m \in \natz$ such that $(\lowtranopa{1:k})^{m} \indic{\statespacesub{R}} > 0$.
	Also, we know that $B_{m} = \statespace$, where $B_{0} = \statespacesub{R}$ and
	\[
		B_{i+1}
		\coloneqq B_{i} \cup \left\{ x \in \statespace \setminus B_{i} \colon [\lowtranopa{1:k} \indic{B_{i}}](x) > 0 \right\} ~\text{for all}~i \in \{ 0, \dots, m-1 \}.
	\]
	For any $i \in \{ 0, \dots, m-1 \}$ and any $x \in \statespace$, it follows from Lemma~\ref{lem:LowTranOp:SignOfCompositionWithEvent} that $[\lowtranopa{1:k} \indic{B_{i}}](x) > 0$ if and only if $x \in B_{i,k}$, where $B_{i,k}$ is derived from the initial condition $B_{i,0} \coloneqq B_{i}$ and, for all $j \in \{ 1, \dots, k \}$, from the recursive relation
	\begin{align*}
		B_{i,j}
		&= \left\{ x \in \statespace \colon [(I + \delta_{j} \lowrateop) \indic{B_{i,j-1}}](z) > 0 \right\}.
	\intertext{Similar to what we did before, we can rewrite this recursive relation as}
		B_{i,j}
		&= \left\{ x \in B_{i,j-1} \colon [(I + \delta_{j} \lowrateop) \indic{B_{i,j-1}}](z)>0 \right\} \cup \left\{ x \in \statespace \setminus B_{i,j-1} \colon [(I + \delta_{j} \lowrateop) \indic{B_{i,j-1}}](z) > 0 \right\} \\
		&= \left\{ x \in B_{i,j-1} \colon 1 + \delta_{j} [\lowrateop \indic{B_{i,j-1}}](z) > 0 \right\} \cup \left\{ x \in \statespace \setminus B_{i,j-1} \colon \delta_{j} [\lowrateop \indic{B_{i,j-1}}](z) > 0 \right\}.
	\intertext{As before, we can verify that $1 + \delta_{j} [\lowrateop \indic{B_{i,j-1}}](z) > 0$ for all $x \in B_{i,j-1}$.
	Hence,}
		B_{i,j}
		&= B_{i,j-1} \cup \left\{ x \in \statespace \setminus B_{i,j-1} \colon [\lowrateop \indic{B_{i,j-1}}](z) > 0 \right\}.
	\end{align*}
	This way, we have constructed a sequence of sets
	\[
		B_{0} = B_{0, 0}, B_{0, 1}, \dots, B_{0,k} = B_1 = B_{1,0}, B_{1,1}, \dots, B_{1,k} = B_{2} = B_{2,0}, \dots, B_{m-1, k} = B_{m}
	\]
	with $B_{0} = \statespacesub{R}$ and $B_{m} = \statespace$.
	Denote this sequence by $A_{0}, \dots, A_{mk + 1}$ and let $A_{mk+2}\coloneqq\statespace$.
	Then $A_{0} = \statespacesub{R}$, $A_{mk + 1} =A_{mk + 2}=\statespace$ and for all $i \in \{0, \dots, mk+1\}$,
	\[
		A_{i+1} = A_{i} \cup \left\{ x \in \statespace \setminus A_{i} \colon [\lowrateop \indic{A_{i}}](z) > 0 \right\}.
	\]
	Let $n\in\{0,\dots,mk+1\}$ be the first index for which $A_{n} = A_{n+1}$.
	From the recursive relation between $A_{n}, \dots, A_{m k + 1},A_{m k + 2}$, we infer that $A_{n} = A_{n+1} = \cdots = A_{m k + 2} = \statespace$.
	Fix an arbitrary $y{*} \in \statespace \setminus \statespacesub{R}$.
	Then the sequence $\statespacesub{R} = A_{0}, \dots, A_{n}, A_{n+1}$ satisfies the recursive relation of Definition~\ref{def:LowRateOp:LowerReachable} and $y^{*}\in\statespace=A_n$, so $y^{*} \lowreachc \statespacesub{R}$.
	As $y^{*}$ was an arbitrary element of $\statespace \setminus \statespacesub{R}$, it follows that $\lowrateop$ is top class absorbing.

	We have proven that if there is some $k < \card{\statespace}$ and some sequence $\delta_1, \dots, \delta_k$ in $\strictlyposreals$ such that $\delta_{i} \norm{\lowrateop} < 2$ for all $i \in \{ 1, \dots, k \}$ and $\coefferga{\Phi(\delta_{1},\dots,\delta_{k})} < 1$, then $\lowrateop$ is both top class regular and top class absorbing.
	As an immediate consequence of Theorem~\ref{the:ContinuousErgodicity:NecessaryAndSufficient}, this implies that $\lowrateop$ is ergodic.
\end{proof}

\begin{proof}[Proof of Proposition~\ref{prop:UniformApproximationErgodicError}]
	From the requirements on $\delta$, \ref{prop:LTO:CompositionIsAlsoLTO} and Proposition~\ref{prop:IPlusDeltaQLowTranOp}, it follows that $(I + \delta \lowrateop)^{i}$ is a lower transition operator for all $i \in \nats$.
	By Lemma~\ref{lem:ExplicitErrorBound},
	\begin{align*}
		\norm{\lowtranopa{t} f - \Psi_{t}(n)}
		&\leq \delta^2 \norm{\lowrateop}^{2} \sum_{i = 0}^{n-1} \norm{(I + \delta \lowrateop)^{i} f}_{c} \\
		&= \delta^2 \norm{\lowrateop}^{2} \sum_{i = 0}^{k-1} \sum_{j = 0}^{m-1} \norm{(I + \delta \lowrateop)^{j}(I + \delta \lowrateop)^{m i} f}_{c}.
	\intertext{%
		We use \ref{prop:LTO:VarNormTf} to yield
	}
		\norm{\lowtranopa{t} f - \Psi_{t}(n)}
		&\leq m \delta^2 \norm{\lowrateop}^{2} \sum_{i = 0}^{k-1} \norm{(I + \delta \lowrateop)^{m i} f}_{c}.
	\intertext{Next, we simply use \ref{prop:CoeffOfErgod:BoundOnNormTf} and \ref{prop:CoeffOfErgod:Composition} to yield}
		\norm{\lowtranopa{t} f - \Psi_{t}(n)}
		&\leq m \delta^2 \norm{\lowrateop}^{2} \norm{f}_{c} \sum_{i = 0}^{k-1} \coefferga{(I + \delta \lowrateop)^{m}}^{i}.
	\end{align*}
	For any $a \in [0, 1)$ and any $\ell \in \nats$, it is well known that
	\[
		\sum_{i = 0}^{\ell} a^{i} = \frac{1 - a^{\ell+1}}{1 - a} \leq \frac{1}{1-a}.
	\]
	If $\beta \coloneqq \coefferga{(I + \delta \lowrateop)^{m}} < 1$, then we can use this well-known relation to yield
	\begin{align*}
		\norm{\lowtranopa{t} f - \Psi_{t}(n)}
		&\leq m \delta^2 \norm{\lowrateop}^{2} \norm{f}_{c} \frac{1 - \beta^{k}}{1 - \beta}
		\leq \frac{m \delta^2 \norm{\lowrateop}^{2} \norm{f}_{c}}{1 - \beta}.
	\end{align*}

	The proof for $\beta = \coefferga{\lowtranopa{m \delta}}$ is entirely analoguous.
	We can use the second inequality of Lemma~\ref{lem:ExplicitErrorBound}, the semi-group property and \ref{prop:LTO:VarNormTf}, which yields
	\[
		\norm{\lowtranopa{t} f - \Psi_{t}(n)}
		\leq m \delta^2 \norm{\lowrateop}^{2} \sum_{i = 0}^{k-1}\norm{(\lowtranopa{m \delta})^{i} f}_{c}.
	\]
	Next, we again use \ref{prop:CoeffOfErgod:BoundOnNormTf} and \ref{prop:CoeffOfErgod:Composition} to yield
	\[
		\norm{\lowtranopa{t} f - \Psi_{t}(n)}
		\leq m \delta^2 \norm{\lowrateop}^{2} \norm{f}_{c} \sum_{i = 0}^{k-1} \coefferga{\lowtranopa{m \delta}}^{i}. \qedhere
	\]

\end{proof}

\begin{proof}[Proof of Example~\ref{binex:UniformErgodic}]
	Let $\delta \in \nonnegreals$ such that $\delta \norm{\lowrateop} \leq 2$.
	Using Proposition~\ref{prop:CoeffOfErgod:AlternativeFunctions} yields
	\begin{align*}
		\coefferga{\Phi(\delta)}
		&= \max \{ \norm{\Phi(\delta) f}_{v} \colon f \in \setoffna, \max f = 1, \min f = 0 \}.
	\end{align*}
	In the special case of a binary state space, only two functions satisfy this requirement: $\indic{0}$ and $\indic{1}$.
	Therefore
	\begin{align*}
		\coefferga{\Phi(\delta)}
		&= \max \big\{ \abs{[\Phi(\delta) \indic{0}](0) - [\Phi(\delta) \indic{0}](1)}, \abs{[\Phi(\delta) \indic{1}](0) - [\Phi(\delta) \indic{1}](1)} \big\}.
	\end{align*}
	Recall that in the Proof of Example~\ref{binex:AnalyticalExpressionsForAppliedLTO} we proved that for all $\delta \in \nonnegreals$ such that $\delta \norm{\lowrateop} \leq 2$ and all $f \in \setoffna$,
	\[
		[\Phi(\delta) f](0) - [\Phi(\delta) f](1)
		= \begin{cases}
			\norm{f}_{v} (1 - \delta (\upq{0} + \lowq{1})) &\text{if } f(0) \geq f(1), \\
			\norm{f}_{v} (1 - \delta (\lowq{0} + \upq{1})) &\text{if } f(0) \leq f(1).
		\end{cases}
	\]
	As $\norm{\indic{0}}_{v} = 1 = \norm{\indic{1}}_{v}$, this yields
	\begin{equation*}
		\coefferga{I + \delta \lowrateop}
		= \coefferga{\Phi(\delta)}
		= \max \left\{ \abs{1 - \delta (\upq{0} + \lowq{1})}, \abs{1 - \delta (\lowq{0} + \upq{1})} \right\}. \qedhere
	\end{equation*}
\end{proof}

For the proof of Theorem~\ref{the:CoeffOfErgod:Approximation}, we need some definitions and results from the theory of imprecise probabilities.
The reason for this is that, as \cite{decooman2009} already mention, the functional $[\lowtranop \cdot](x)$ is actually a coherent (conditional) lower expectation.
For a more thorough discussion of coherent lower expectations---often also called coherent lower previsions---we refer to the seminal work of \cite{1991Walley} and the more recent treatment of \cite{2014LowPrev}.
\begin{definition}
	A functional $\lowprev$ that maps $\setoffna$ to $\reals$ is a \emph{coherent lower expectation} if for all $f, g \in \setoffna$ and all $\mu \in \nonnegreals$:
	\begin{enumerate}[label=\underline{E}\arabic*:, ref=(\underline{E}\arabic*)]
		\item $\lowpreva{f} \geq \min f$; \label{def:CLP:BoundedByMin}
		\item $\lowpreva{f + g} \geq \lowpreva{f} + \lowpreva{g}$; \label{def:CLP:subadditive}
		\item $\lowpreva{\mu f} = \mu \lowpreva{f}$. \label{def:CLP:PositiveHomogeneity}
	\end{enumerate}
\end{definition}
The conjugate \emph{coherent upper expectation} is defined for all $f \in \setoffna$ as
\[
	\uppreva{f} = - \lowpreva{-f}.
\]
If for all $f \in \setoffna$, $\uppreva{f} = \lowpreva{f} = \linpreva{f}$, then we call $\prev$ a \emph{linear expectation}.
The reason for this terminology is that the inequality in \ref{def:CLP:subadditive} can then be replaced by an equality, and the condition $\mu \in \nonnegreals$ for \ref{def:CLP:PositiveHomogeneity} can be relaxed to $\mu \in \reals$.

The following corollary highlights the link between the components of a lower transition operator and coherent lower previsions.
\begin{corollary}
\label{cor:LTOisCLP}
	Let $\lowtranop$ be a lower transition operator and $x \in \statespace$.
	Then the functional $[\lowtranop \cdot](x) \colon f \in \setoffna \mapsto [\lowtranop f](x)$ is a coherent lower prevision.
\end{corollary}
\begin{proof}
	The operator $[\lowtranop \cdot](x)$ indeed maps $\setoffna$ to $\reals$.
	Furthermore, \ref{def:CLP:BoundedByMin} follows from \ref{def:LTO:DominatesMin}, \ref{def:CLP:subadditive} follows from \ref{def:LTO:SuperAdditive} and \ref{def:CLP:PositiveHomogeneity} follows from \ref{def:LTO:NonNegativelyHom}.
	Hence, the operator is indeed a coherent lower prevision.
\end{proof}

For any coherent lower expectation $\lowprev$, the set $\credseta{\lowprev}$ of dominating linear expectations, defined as
\[
	\credseta{\lowprev} \coloneqq \{ \prev \text{ a linear expectation operator} \colon (\forall f \in \setoffna)~\lowpreva{f} \leq \linpreva{f} \},
\]
is non-empty.
Moreover, from \cite[Section~3.3.3]{1991Walley} it follows that $\lowprev$ is the lower envelope of $\credseta{\lowprev}$, in the sense that for all $f \in \setoffna$,
\[
	\lowpreva{f} = \min \{ \linpreva{f} \colon \prev \in \credseta{\lowprev} \}.
\]

\begin{lemma}[Alternative statement of Proposition~1 in \cite{2013Skulj}]
\label{lem:MaxDifferenceBetweenLinearPrevisionsForIndicators}
	If\/ $\prev_1$ and $\prev_2$ are two linear expectation operators, then
	\[
		\max \{ \prev_{1}(f) - \prev_{2}(f) \colon f \in \setoffna, 0 \leq f \leq 1 \}
		= \max \{ \prev_{1}(\indic{A}) - \prev_{2}(\indic{A}) \colon f \in \setoffna, \emptyset \neq A \subset \statespace \}.
	\]
\end{lemma}
\begin{proof}
	Let $\prev_1$ and $\prev_2$ be any two linear expectation operators on $\setoffna$.
	Then
	\begin{align*}
		&\max \{ \prev_{1}(f) - \prev_{2}(f) \colon f \in \setoffna, 0 \leq f \leq 1 \} \\
		&\qquad = \max \left\{ \sum_{x \in \statespace} (\prev_{1}(\indic{x}) - \prev_{2}(\indic{x})) f(x) \colon f \in \setoffna, 0 \leq f \leq 1 \right\} \\
		&\qquad = \sum_{x \in A^*} (\prev_{1}(\indic{x}) - \prev_{2}(\indic{x}))
		= \prev_{1}(\indic{A^*}) - \prev_{2}(\indic{A^*}),
	\end{align*}
	where $A^* \subset \statespace$ is defined as $A^* \coloneqq \left\{ x \in \statespace \colon \prev_{1}(\indic{x}) > \prev_{2}(\indic{x}) \right\}$.
\end{proof}

\begin{lemma}
\label{lem:MaxDifferenceBetweenLowerPrevisionsUpperBoundWithIndicators}
	If\/ $\lowprev_{1}$ and $\lowprev_{2}$ are two coherent lower expectations on $\setoffna$, then
	\[
		\max \{ \lowprev_1(f) - \lowprev_2(f) \colon f \in \setoffna, 0 \leq f \leq 1 \}
		\leq \max \{ \upprev_1(\indic{A}) - \lowprev_2(\indic{A}) \colon 0 \neq A \subset \statespace \}.
	\]
\end{lemma}
\begin{proof}
	Define $\credset_1 \coloneqq \credseta{\lowprev_1}$ and $\credset_2 \coloneqq \credseta{\lowprev_2}$.
	Note that for all $f \in \setoffna$,
	\[
		0 =\lowprev_{1}(0) = \lowprev_{1}(f - f) \geq \lowprev_{1}(f) + \lowprev_{1}(-f),
	\]
	where the first equality follows from \ref{def:CLP:PositiveHomogeneity}---with $\mu = 0$ and $f = 0$---and the first inequality follows from \ref{def:CLP:subadditive}.
	Bringing the second term to the left hand side and using the conjugacy relation between $\lowprev_{1}$ and $\upprev_{1}$, we find $\upprev_{1}(f) \geq \lowprev_{1}(f)$.
	Hence
	\[
		\lowprev_1(f) - \lowprev_2(f)
		\leq \upprev_1(f) - \lowprev_2(f),
	\]
	and consequently
	\[
		\max \{ \lowprev_1(f) - \lowprev_2(f) \colon f \in \setoffna, 0 \leq f \leq 1 \}
		\leq \max \{ \upprev_1(f) - \lowprev_2(f) \colon f \in \setoffna, 0 \leq f \leq 1 \}.
	\]
	Recall that for any $f \in \setoffna$, $\lowprev_i(f) = \min_{\prev_i \in \credset_{i}} \prev_i(f)$, so
	\begin{align*}
		\upprev_1(f) - \lowprev_2(f)
		&= \max_{\prev_1 \in \credset_{1}} \prev_1(f) - \min_{\prev_2 \in \credset_{2}} \prev_2(f)
		= \max_{\prev_1 \in \credset_{1}} \max_{\prev_2 \in \credset_{2}} \prev_1(f)-\prev_2(f).
	\end{align*}
	We use the previous equality to rewrite the right hand side of the previous inequality:
	\begin{align*}
		&\max \{ \upprev_1(f) - \lowprev_2(f) \colon f \in \setoffna, 0 \leq f \leq 1 \} \\
		&\qquad = \max \big\{ \max_{\prev_1 \in \credset_{1}} \max_{\prev_2 \in \credset_{2}} (\prev_{1}(f) - \prev_{2}(f)) \colon f \in \setoffna, 0 \leq f \leq 1 \big\} \\
		&\qquad = \max_{\prev_1 \in \credset_{1}} \max_{\prev_2 \in \credset_{2}} \max \big\{ (\prev_{1}(f) - \prev_{2}(f)) \colon f \in \setoffna, 0 \leq f \leq 1 \big\}.
	\end{align*}
	Next, we use Lemma~\ref{lem:MaxDifferenceBetweenLinearPrevisionsForIndicators} to yield
	\begin{align*}
		&\max \{ \lowprev_1(f) - \lowprev_2(f) \colon f \in \setoffna, 0 \leq f \leq 1 \} \\
		&\qquad \leq \max_{\prev_1 \in \credset_{1}} \max_{\prev_2 \in \credset_{2}} \max \big\{ (\prev_{1}(f) - \prev_{2}(f)) \colon f \in \setoffna, 0 \leq f \leq 1 \big\} \\
		&\qquad = \max_{\prev_1 \in \credset_{1}} \max_{\prev_2 \in \credset_{2}} \max \big\{ (\prev_{1}(\indic{A}) - \prev_{2}(\indic{A})) \colon 0 \neq A \subset \statespace \big\} \\
		&\qquad = \max \big\{ \max_{\prev_1 \in \credset_{1}} \max_{\prev_2 \in \credset_{2}} (\prev_{1}(\indic{A}) - \prev_{2}(\indic{A})) \colon 0 \neq A \subset \statespace \big\} \\
		&\qquad = \max \{ \upprev_{1}(\indic{A}) - \lowprev_{2}(\indic{A}) \colon 0 \neq A \subset \statespace \}. \qedhere
	\end{align*}
\end{proof}

\begin{proof}[Proof of \cref{the:CoeffOfErgod:Approximation}]
	Fix some lower transition operator $\lowtranop$.
	The lower bound on $\coefferga{\lowtranop}$ follows from the fact that for any $\emptyset \neq A \subset \statespace$, $0 \leq \indic{A} \leq 1$.
	Recall from Corollary~\ref{cor:LTOisCLP} that for any $x\in\statespace$, $[\lowtranop \cdot](x)$ is a coherent lower prevision.
	Therefore, we can use Lemma~\ref{lem:MaxDifferenceBetweenLowerPrevisionsUpperBoundWithIndicators} to yield the upper bound:
	\begin{align*}
		\coefferga{\lowtranop}
		&= \max \{ \norm{\lowtranop f}_{v} \colon f \in \setoffna, 0 \leq f \leq 1 \} \\
		&= \max \big\{ \max \{ [\lowtranop f](x) - [\lowtranop f](y) \colon x,y \in \statespace \} \colon f \in \setoffna, 0 \leq f \leq 1 \big\} \\
		&= \max \big\{ \max \{ [\lowtranop f](x) - [\lowtranop f](y) \colon f \in \setoffna, 0 \leq f \leq 1 \} \colon x,y \in \statespace \big\} \\
		&\leq \max \big\{ \max \{ [\uptranop \indic{A}](x) - [\lowtranop \indic{A}](y) \colon \emptyset \neq A \subseteq \statespace \} \colon x,y \in \statespace \big\} \\
		&= \max \big\{ \max \{ [\uptranop \indic{A}](x) - [\lowtranop \indic{A}](y) \colon x,y \in \statespace \} \colon \emptyset \neq A \subseteq \statespace\big\}. \qedhere
	\end{align*}
\end{proof}

\begin{proof}[Proof of the counterexample for \cref{prop:CoeffOfErgod:ErgodicUpperBound}]
	We first verify that \(\lowrateop\) is ergodic.
	Since \([\uprateop \indic{1}](0) = \upq{0} > 0\) and \([\uprateop \indic{0}](1) = \upq{1} > 0\), it follows from \cref{def:LowRateOp:UpperReachable,def:LowRateOp:RegularlyAbsorbing} that \(\statespacesub{R} = \statespace\), such that \(\lowrateop\) is regularly absorbing.
	Hence, by also invoking \cref{the:ContinuousErgodicity:NecessaryAndSufficient} we can conclude that \(\lowrateop\) is ergodic.

	Next, we fix some \(\delta \in \strictlyposreals{}\) such that \(\delta \norm{\lowrateop} < 2\).
	Recall from \cref{prop:IPlusDeltaQLowTranOp} that \((I + \delta \lowrateop)\) is a lower transition operator.
	Consequently, we can use \cref{eqn:CoeffOfErgod:UpperBound} to compute an upper bound for \(\coefferga{I + \delta \lowrateop}\).
	In this case, there are clearly only two possibilities for \(A\) in the optimisation of \cref{eqn:CoeffOfErgod:UpperBound}: \(A = \{0\}\) and \(A = \{1\}\).
	For \(A = \{0\}\), some straightforward calculations yield
	\begin{align*}
		[(I + \delta \uprateop) \indic{0}](0)
		&= 1, &
		[(I + \delta \uprateop) \indic{0}](1)
		&= \delta \upq{1}, \\
		[(I + \delta \lowrateop) \indic{0}](0)
		&= 1 - \delta \upq{0}, &
		[(I + \delta \lowrateop) \indic{0}](1)
		&= 0.
	\end{align*}
	This implies that
	\begin{multline*}
		\max \big\{ \max \{ [(I + \delta \uprateop) \indic{A}](x) - [(I + \delta \lowrateop) \indic{A}](y) \colon x,y \in \statespace \} \colon \emptyset \neq A \subset \statespace\big\} \\
		\geq [(I + \delta \uprateop) \indic{0}](0) - [(I + \delta \lowrateop) \indic{0}](1)
		= 1.
		\qedhere
	\end{multline*}
\end{proof}

If the lower transition operator is linear, then the lower and upper bounds of Theorem~\ref{the:CoeffOfErgod:Approximation} are equal.
Moreover, from this special case we can immediately verify that the ergodic coefficient we use is a proper generalisation of an ergodic coefficient---the delta coefficient $\delta$ of \cite{1991Anderson}, which is equivalent to $\tau_1$, one of the proper coefficients of ergodicity discussed by \cite{1981Seneta}---used in the study of precise Markov chains.
\begin{corollary}
\label{the:LowTranOp:ApproximationOfCoeffOfErgod}
	Let $T$ be a transition matrix, then
	\begin{align*}
		\coefferga{T}
		&= \max \left\{ \frac{1}{2} \sum_{z \in \statespace} \abs{T(x,z) - T(y,z)} \colon x,y \in \statespace \right\}.
	\end{align*}
\end{corollary}
\begin{proof}
	For a transition matrix, the the upper bound of Eqn.~\eqref{eqn:CoeffOfErgod:UpperBound} and the lower bound of Eqn.~\eqref{eqn:CoeffOfErgod:LowerBound} in Theorem~\ref{the:CoeffOfErgod:Approximation} are equal.
	Therefore
	\begin{align*}
		\coefferga{T}
		&= \max \left\{ \max \{ [T \indic{A}](x) - [T \indic{A}](y) \colon x,y \in \statespace \} \colon \emptyset \neq A \subset \statespace \right\} \\
		&= \max \left\{ \max \left\{ \frac{1}{2} [T (2 \indic{A})](x) - \frac{1}{2} [T (2 \indic{A})](y) \colon x,y \in \statespace \right\} \colon \emptyset \neq A \subset \statespace \right\} \\
		&= \max \left\{ \max \left\{ \frac{1}{2} [T (2 \indic{A} - 1)](x) - [T (2 \indic{A} - 1)](y) \colon x,y \in \statespace \right\} \colon \emptyset \neq A \subset \statespace \right\},
		\intertext{%
			where the first equality follows from Theorem~\ref{the:CoeffOfErgod:Approximation}, the second equality follows from \ref{def:LTO:NonNegativelyHom} and the third equality follows from \ref{prop:LTO:AdditionOfConstant}.
			From the linearity of $T$, it follows that $[T f](x) = \sum_{z \in \statespace} f(z) [T \indic{z}](x) = \sum_{z \in \statespace} f(z) T(x,z)$, such that
		}
		\coefferga{T}
		&= \max \left\{ \max \left\{ \frac{1}{2} \sum_{z \in \statespace} [2 \indic{A} - 1](z) \left(T(x,z) - T(y,z) \right) \colon x,y \in \statespace \right\} \colon \emptyset \neq A \subset \statespace \right\}\\
		&= \max \left\{ \max \left\{ \frac{1}{2} \sum_{z \in \statespace} [2 \indic{A} - 1](z) \left(T(x,z) - T(y,z) \right) \colon \emptyset \neq A \subset \statespace \right\} \colon x,y \in \statespace \right\}.
		\intertext{%
			Solving the inner maximisation problem for some fixed $x, y \in \statespace$ is trivial: the maximising $A$ is $\{ z \in \statespace \colon T(x,z) \geq T(y,z) \}$ as for all $z \in \statespace$, $[2\indic{A} - 1](z)$ is $1$ if $z \in A$ or $-1$ if $z \notin A$.
			This results in
		}
		\coefferga{T}
		&= \max \left\{ \frac{1}{2} \sum_{z \in \statespace} \abs{T(x,z) - T(y,z)} \colon x,y \in \statespace \right\},
	\end{align*}
	which proves that $\coefferga{T}$ is indeed equal to $\delta(T)$ of \cite{1991Anderson} or $\tau_1(T)$ of \cite{1981Seneta}.
\end{proof}

Linear transition operators are not the only lower transition operators for which the lower bound of Theorem~\ref{the:CoeffOfErgod:Approximation} is the actual value of the coefficient of ergodicity.
\cite{2013Skulj} show that this is also the case for lower transition operators defined using Choquet integrals.
Let $\{ L_x \}_{x\in\statespace}$ be a family of Choquet capacities, and assume that for all $x\in\statespace$, $[\lowtranop \cdot ](x)$ is the Choquet integral with respect to $L_x$.
By \cite[Corollary~23]{2013Skulj},
\begin{align}
	\coefferga{\lowtranop}
	&= \max \big\{ \max\{ L_{x}(A) - L_{y}(A) \colon x,y \in \statespace \} \colon 0 \neq A \subset \statespace \big\}.
	\label{eqn:CoeffOfErgod:Choquet}
\end{align}
This result allows us to exactly compute $\coefferga{\lowtranop}$.
However, we are often interested in $\coefferga{\lowtranop^{k}}$, where $k > 1$ is an integer.
Let $k\in\nats$ and $x\in\statespace$, then we define the Choquet capacity $L^{k}_{x}$ for all $A \subseteq \statespace$ as $L^{k}_{x}(A) \coloneqq [\lowtranop^k \indic{A}](x)$.
In general, the coherent lower expectation $[\lowtranop^k \cdot](x)$ is \emph{not} a Choquet integral with respect to the Choquet capacity $L^{k}_{x}$, a fact that is seemingly overlooked in \cite[Section~5.5]{2013Skulj}.
What is definitely true is that
\[
	\max \{ \max \{ L^{k}_{x}(A) - L^{k}_{y}(A) \colon x,y \in \statespace \} \colon \emptyset \neq A \subset \statespace \}
\]
is a lower bound of $\coefferga{\lowtranop^{k}}$, as it is equal to the lower bound of Theorem~\ref{the:CoeffOfErgod:Approximation}.

\begin{lemma}
\label{lem:StoppigCriterionWithConvergence}
	Let $\lowrateop$ be a lower transition rate operator and assume that $f$ is an element of $\setoffna$ such that $\lowtranopa{\infty} f \coloneqq \lim_{t \to \infty} \lowtranopa{t} f$ is a constant function.
	We let $t \in \nonnegreals$, $\epsilon \in \strictlyposreals$ and $\delta_1, \dots, \delta_k \in \nonnegreals$ such that $\sum_{i=1}^{k} \delta_i = t$ and for all $i \in \{ 1, \dots, k \}$, $\delta_{i} \norm{\lowrateop} \leq 2$, and define $g \coloneqq \Phi(\delta_1,\dots,\delta_k) f$.
	If $\norm{\lowtranopa{t} f - g} \leq \epsilon$ and $\norm{g}_{c} \leq \epsilon$, then
	\[
		\norm{\lowtranopa{\infty} f - \tilde{g}} \leq 2 \epsilon
	\]
	and for all $\Delta \in \nonnegreals$,
	\[
		\norm{\lowtranopa{t+\Delta} f - \tilde{g}} \leq 2 \epsilon,
	\]
	where $\tilde{g} \coloneqq (\max \Phi(\delta_1, \dots, \delta_k) f + \min \Phi(\delta_1, \dots, \delta_k) f) / 2$.
\end{lemma}
\begin{proof}
Note that by \ref{prop:LTO:BoundedByMinAndMax},
	\[
		\min \lowtranopa{t} f \leq \min \lowtranopa{t + \Delta} f \leq \lowtranopa{\infty} f \leq \max \lowtranopa{t + \Delta} f \leq \lowtranopa{t} f.
	\]
	If we let $g \coloneqq \Phi(\delta_1,\dots,\delta_k) f$ and assume that $\norm{\lowtranopa{t} f - g } \leq \epsilon$, then
	\[
		\min g - \epsilon \leq \min \lowtranopa{t} f \leq \min \lowtranopa{t + \Delta} \leq \lowtranopa{\infty} f \leq \max \lowtranopa{t + \Delta} f \leq \lowtranopa{t} f \leq \max g + \epsilon.
	\]
	Hence,
	\[
		\lowtranopa{\infty} f - \tilde{g} = \lowtranopa{\infty} f - \max g + \frac{\max g - \min g}{2} \leq \epsilon + \norm{g}_{c},
	\]
	and
	\[
		\lowtranopa{\infty} f - \tilde{g} = \lowtranopa{\infty} f - \min g - \frac{\max g - \min g}{2} \geq - \epsilon- \norm{g}_{c},
	\]
	where $\tilde{g} \coloneqq (\max g + \min g) / 2$.
	Therefore, if $\norm{g}_{c} \leq \epsilon$, then
	\[
		\norm{\lowtranopa{\infty} f - \tilde{g}} \leq 2 \epsilon,
	\]
	which proves the first inequality of the statement.
	The proof of the second inequality of the statement is almost entirely similar.
\end{proof}

\begin{proof}[Proof of Proposition~\ref{prop:StoppingCriterionWithErgodicity}]
	If $\lowrateop$ is ergodic, then by definition $\lim_{t \to \infty} \lowtranopa{t} f$ is a constant function for all $f \in \setoffna$.
	Therefore, the stated follows immediately from Lemma~\ref{lem:StoppigCriterionWithConvergence}.
\end{proof}

In Example~\ref{binex:AdaptiveApproximation}, we have observed that keeping track of $\epsilon'$ increases the duration of the computations.
The following proposition shows that, even if one is not really interested in the value of $\epsilon'$, there is still a reason why one nevertheless would want to keep track of $\epsilon'$: it could be that we can stop the approximation because we have already attained the desired maximal error.
\begin{proposition}
\label{prop:StoppingCriterionWithRemainingCost}
	Let $\lowrateop$ be a lower transition rate operator, $f \in \setoffna$ and $t, \epsilon \in \strictlyposreals$.
	Let $s$ denote some sequence $\delta_1, \dots, \delta_k$ in $\nonnegreals$ such that $\smash{t' \coloneqq \sum_{i = i}^{k} \delta_i \leq t}$ and, for all $i \in \{ 1, \dots, k \}$, $\smash{\delta_i \norm{\lowrateop} \leq 2}$.
	If $\epsilon' \leq \epsilon$ is an upper bound for $\norm{\lowtranopa{t'} f - \Phi(s) f}$ and $\norm{\Phi(s) f}_{v} \leq \epsilon - \epsilon'$, then
	\[
		\norm{\lowtranopa{t} f - \Phi(s) f}
		\leq \epsilon.
	\]
\end{proposition}
\begin{proof}
First, note that by the semi-group property $\lowtranopa{t} f = \lowtranopa{t - t'} \lowtranopa{t'} f$.
	Using \ref{prop:LTO:BoundedByMinAndMax} yields
	\[
		\min \lowtranopa{t'} f \leq \lowtranopa{t} f \leq \max \lowtranopa{t'} f.
	\]
	Hence
	\begin{align*}
		\norm{\lowtranopa{t} f - \Phi(s)f}
		&= \max \{ \abs{[\lowtranopa{t} f](x) - [\Phi(s)f](x) } \colon x \in \statespace \} \\
		&\leq \max \{ \max \{\abs{\max \lowtranopa{t'} f - [\Phi(s)f](x) }, \abs{\min \lowtranopa{t'} f - [\Phi(s)f](x) } \} \colon x \in \statespace \},
	\end{align*}
	where the inequality follows from the obtained bounds on $\lowtranopa{t} f$.
	Let $x^{+} \in \statespace$ such that $[\lowtranopa{t'} f](x^{+}) = \max \lowtranopa{t'} f$.
	Then for all $x \in \statespace$,
	\begin{align*}
		\abs{\max \lowtranopa{t'} f - [\Phi(s)f](x) }
		&= \abs{[\lowtranopa{t'} f](x^{+}) - [\Phi(s)f](x) - [\Phi(s)f](x^{+}) + [\Phi(s)f](x^{+}) } \\
		&\leq \abs{[\lowtranopa{t'} f](x^{+}) - [\Phi(s)f](x^{+}) } + \abs{[\Phi(s)f](x) - [\Phi(s)f](x^{+})} \\
		&\leq \norm{\lowtranopa{t'} f - \Phi(s) f} + \norm{\Phi(s) f}_{v}.
	\intertext{Similarly, 
	}
		\abs{\min \lowtranopa{t'} f - [\Phi(s)f](x) }
		&\leq \norm{\lowtranopa{t'} f - \Phi(s) f} + \norm{\Phi(s) f}_{v}.
	\end{align*}
	Therefore,
	\begin{align*}
		\norm{\lowtranopa{t} f - \Phi(s)f}
		&\leq \max \{ \max \{\abs{\max \lowtranopa{t'} f - [\Phi(s)](x) }, \abs{\min \lowtranopa{t'} f - [\Phi(s)](x) } \} \colon x \in \statespace \} \\
		&\leq \max \{ \norm{\lowtranopa{t'} f - \Phi(s) f} + \norm{\Phi(s) f}_{v}  \colon x \in \statespace \} \\
		&= \norm{\lowtranopa{t'} f - \Phi(s) f} + \norm{\Phi(s) f}_{v}.
	\end{align*}

	If we now assume that $\norm{\lowtranopa{t'} f - \Phi(s) f} \leq \epsilon' \leq \epsilon$ and $\norm{\Phi(s) f}_{v} \leq \epsilon - \epsilon'$, then
	\[
		\norm{\lowtranopa{t} f - \Phi(s)f}
		\leq \epsilon. \qedhere
	\]
\end{proof}

\end{document}